\documentclass[UTF-8,reqno]{amsart}
\usepackage{enumerate, bbm}
\usepackage{amssymb,url,color, booktabs}
\usepackage{tikz}
\usepackage{tikz-cd}
\makeatletter \protected@xdef\x{spam\c{e}eggs} \makeatother
\usepackage{eurosym}
\usepackage{mathrsfs}
\usepackage{dutchcal}
\usepackage{appendix}
\usetikzlibrary{arrows, shapes.geometric}
\usetikzlibrary{decorations.pathreplacing}
\usepackage{graphicx}
\usepackage{pifont}
\usepackage[all]{xy}
\usetikzlibrary{positioning, shapes.geometric}
\usepackage{pgfkeys}
\usepackage{pgffor}
\usepackage{pgfcalendar}
\usepackage{pgfpages}

\usepackage{color}
\usepackage[colorlinks=true]{hyperref}
\hypersetup{
    linkcolor=blue,          
    citecolor=red,        
    filecolor=blue,      
    urlcolor=cyan
}

\usepackage{color}
\usepackage{ulem}

\usetikzlibrary{decorations,arrows}
\usetikzlibrary{decorations.markings}
\usepackage{wrapfig}

\definecolor{MyDarkBlue}{cmyk}{0.8,0.3,0.8,0.4}
\definecolor{yellow}{rgb}{0.99,0.99,0.70}
\definecolor{orange}{rgb}{0.00,0.00,0.00}

\definecolor{white}{rgb}{1.0,1.0,1.0}
\definecolor{black}{rgb}{0.00,0.00,0.00}
\definecolor{MColor}{rgb}{0.65, 0.2, 0.16}
\definecolor{MColor1}{rgb}{0.5, 0.1, 0.8}
\definecolor{Dgreen}{rgb}{0.0,0.0,0.0}
\definecolor{Dkgreen}{rgb}{0.4, 0.69, 0.2}
\newcommand{\black}{\color{black}}
\newcommand{\red}{\color{MColor}}
\newcommand{\ZXC}{\color{black}}

\numberwithin{equation}{section}
\def\theequation{\arabic{section}.\arabic{equation}}
\newcommand{\be}{\begin{eqnarray}}
\newcommand{\ee}{\end{eqnarray}}
\newcommand{\ce}{\begin{eqnarray*}}
\newcommand{\de}{\end{eqnarray*}}
\newtheorem{theorem}{Theorem}[section]
\newtheorem{lemma}[theorem]{Lemma}
\newtheorem{remark}[theorem]{Remark}
\newtheorem{definition}[theorem]{Definition}
\newtheorem{proposition}[theorem]{Proposition}
\newtheorem{Examples}[theorem]{Example}
\newtheorem{corollary}[theorem]{Corollary}

\def\eps{\varepsilon}

\def\e{\mathrm{e}}

\def\p{\partial}

\def\[{{\Big[}}
\def\]{{\Big]}}
\def\<{{\langle}}
\def\>{{\rangle}}
\def\({{\big(}}
\def\){{\big)}}

\def\Law{{\mathord{{\rm Law}}}}
\def\dif{{\mathord{{\rm d}}}}

\def\bbp{{\boldsymbol{p}}}

\def\bbq{{\boldsymbol{q}}}
\def\bba{{\boldsymbol{a}}}

\def\Gap{{\Lambda}}
\def\bb2{{\boldsymbol{2}}}
\def\no{\nonumber}
\def\={&\!\!=\!\!&}

\def\bB{{\mathbf B}}

\def\bC{{\mathbf C}}

\def\cA{{\mathcal A}}

\def\cD{{\mathcal D}}

\def\cH{{\mathcal H}}
\def\cI{{\mathcal I}}

\def\cL{{\mathcal L}}
\def\cM{{\mathcal M}}
\def\cN{{\mathcal N}}

\def\cP{{\mathcal P}}

\def\cR{{\mathcal R}}
\def\cS{{\mathcal S}}
\def\cT{{\mathcal T}}
\def\cU{{\mathcal U}}

\def\mB{{\mathbb B}}

\def\mE{{\mathbb E}}

\def\mL{{\mathbb L}}

\def\mN{{\mathbb N}}

\def\mP{{\mathbb P}}

\def\mR{{\mathbb R}}
\def\mS{{\mathbb S}}

\def\b1{{\mathbbm 1}}

\def\sB{{\mathscr B}}

\def\sI{{\mathscr I}}

\def\sR{{\mathscr R}}
\def\sS{{\mathscr S}}

\def\E{\mathbb E}

\def\geq{\geqslant}
\def\leq{\leqslant}
\def\ge{\geqslant}
\def\le{\leqslant}

\def\c{\mathord{{\bf c}}}
\def\div{\mathord{{\rm div}}}

\def\eps{\varepsilon}

\def\e{\mathrm{e}}

\def\p{\partial}

\def\[{{\Big[}}
\def\]{{\Big]}}
\def\<{{\langle}}
\def\>{{\rangle}}

\def\Law{{\mathord{{\rm Law}}}}

\def\no{\nonumber}
\def\={&\!\!=\!\!&}
\def\bt{\begin{theorem}}
\def\et{\end{theorem}}
\def\bl{\begin{lemma}}
\def\el{\end{lemma}}
\def\br{\begin{remark}}
\def\er{\end{remark}}
\def\bd{\begin{definition}}
\def\ed{\end{definition}}
\def\bp{\begin{proposition}}
\def\ep{\end{proposition}}
\def\bc{\begin{corollary}}
\def\ec{\end{corollary}}

\def\geq{\geqslant}
\def\leq{\leqslant}
\def\ge{\geqslant}
\def\le{\leqslant}

\def\c{\mathord{{\bf c}}}
\def\div{\mathord{{\rm div}}}

 \def\R{\mathbb R}
 \def\R{\mathbb R}

\def\<{\langle} \def\>{\rangle}

\def\x{{\mathbf x}}

\def\bbp{{\boldsymbol{p}}}

\def\bbq{{\boldsymbol{q}}}
\def\bbd{{\boldsymbol{d}}}
\def\bba{{\boldsymbol{a}}}
\def\bbk{{\boldsymbol{k}}}

\def\bb2{{\boldsymbol{2}}}
\def\bbb1{\boldsymbol{1}}
\def\bbinfty{\boldsymbol{\infty}}
\def\no{\nonumber}
\def\={&\!\!=\!\!&}

\def\nalpha{{\alpha}} 
\def\modulateorder{{\zeta}} 
\def\rate{{\Upsilon}}
\def\ind{{~\hbox{\rm l\kern-.4em\hbox{\rm l}}}~} 
\def\transOp{{\Gamma}}
\def\mstable{{\mathfrak m}}%

\def\Levy{L}
\def\Tup{T_0}

\newcommand{\comm}[1]{}
\begin{document}

\title[Propagation of chaos of McKean-Vlasov SDEs]
{Propagation of chaos for moderately interacting particle systems related to 
	singular kinetic McKean-Vlasov SDEs}

\author{Zimo Hao, Jean-Francois Jabir, St\'ephane Menozzi, Michael R\"ockner, Xicheng Zhang}

\address{Zimo Hao:
School of Mathematics and Statistics, Beijing Institute of Technology, Beijing 100081, China.\\
Email: zimo\_hao@163.com} 

\address{Jean-Francois Jabir:
	Department of Statistics and Data Analysis, and Laboratory of Stochastic Analysis and its Application, Moscow, Russia Federation. \\
	Email address: jjabir@hse.ru }

\address{St\'ephane Menozzi:
	LaMME, Universit\'e d'Evry Val d'Essonne, Universit\'e Paris-Saclay, CNRS (UMR 8071), 23 Boulevard de France
	91037 Evry, France.\\
	Email address: stephane.menozzi@univ-evry.fr}

\address{Michael R\"ockner:
Fakult\"at f\"ur Mathematik, Universit\"at Bielefeld,
33615, Bielefeld, Germany\\
and Academy of Mathematics and Systems Science, CAS, Beijing, China\\
Email: roeckner@math.uni-bielefeld.de
 }

\address{Xicheng Zhang:
School of Mathematics and Statistics, Beijing Institute of Technology, Beijing 100081, China\\
Email: XichengZhang@gmail.com
 }

\thanks{
This work is partially supported by the National Natural Science Foundation of China (NSFC, Grant No. 12595282)  and by the German Research Foundation (DFG) through the 
Collaborative Research Centre(CRC) 1283/2 2021 - 317210226 ``Taming uncertainty and profiting from randomness and low regularity in analysis, stochastics and their applications". The second author acknowledges 
 the support of the Russian Science Foundation project (project No. 24-11-00123).}

\begin{abstract}

We study the propagation of chaos \textcolor{Dgreen}{in} a class of moderately interacting particle systems for the approximation of singular kinetic McKean-Vlasov SDEs driven by $\alpha$-stable processes.
 Diffusion parts include Brownian ($\alpha=2$) and pure-jump ($\alpha\in(1,2))$ perturbations and interaction kernels are considered in a non-smooth anisotropic Besov space. Using Duhamel formula, sharp density estimates (recently issued in \cite{HRZ23}), and suitable martingale functional inequalities, we obtain direct estimates on the \textcolor{black}{convergence rate} between the empirical measure of the particle systems toward the McKean-Vlasov distribution. These estimates further lead to quantitative propagation of chaos results in the weak and strong sense.



\end{abstract}
\maketitle \rm


\section{Introduction}
\subsection{McKean-Vlasov SDEs} 
In this paper, we establish quantitative propagation of chaos results (in the weak and pathwise sense) for a class of moderately interacting particle systems related to the (formal) second order stable-driven McKean-Vlasov SDEs given, up to some (possibly infinite) time horizon $T$, by \begin{align}\label{MV00}
		\ddot X_t=(b_t*\mu_{t})(X_t,\dot X_t)+\dot \Levy^{\nalpha}_t,
		\ \ (X_0,\dot X_0)\sim \mu_0, \ \ 0\le t\le T,
	\end{align}
	$\mu_t$ standing for the joint law of $(X_t,\dot X_t)$.
	The driving noise $L^{\nalpha}$ is defined as an $\R^d$-valued isotropic $\nalpha$-stable L\'evy process, $\mu_0\in\cP(\mR^{2d})$ corresponds to the fixed initial probability distribution 
	and $b$ is a time-dependent $\mR^{d}$-valued (Schwartz) distribution over the phase space $\mR^{2d}$. Finally, the component $b_t*\mu_t$ denotes\textcolor{Dgreen}{, whenever it exists,} the \textcolor{Dgreen}{convolution between $b_t$ and the law of the McKean-Vlasov SDE on the phase space for a.e. $t$.} 
	
	The stable noise is assumed to be given with a stability parameter $\nalpha$ in $(1,2]$ (the special case $\nalpha=2$ corresponding to the classical Gaussian noise with $\Levy^\alpha_t=\sqrt{2}W_t$ for $W$ a $d$-dimensional Brownian motion, and the case  $\nalpha\in(1,2)$ to the pure-jump case, the associated infinitesimal generator being given by the 
	fractional Laplacian $\Delta^{\nalpha/2}\textcolor{Dgreen}{:=- (-\Delta)^{\nalpha/2}}$). 
	
	Weak and strong wellposedness of \eqref{MV00} have been established in \cite{HRZ23}, in the case where $t\mapsto b_t$ lies in a mixed Besov space {\black of negative regularity}:
	$$
	b\in L^{q}([0,T];\bB^{\beta_b}_{\bbp_b;\bba}(\R^{2d}))=L^{q}_T\bB^{\beta_b}_{\bbp_b;\bba},
	$$
	where $L^q$ denotes the classical Lebesgue space on the time interval $[0,T]$ and $\bB^{\beta_b}_{\bbp_b;\bba}$ denotes an anisotropic Besov space  (see Section \ref{SEC_BESOV} for a precise definition)
,  \textcolor{black}{for
} 
a suitable set of integrability and regularity parameters $ q$, ${\color{black}\bbp_b}$ and $\beta_b <0$ respectively (the index $\bba $, given in \eqref{DEF_A}, will reflect  the intrinsic scales of the underlying kinetic system). 
	Notably, the authors in \cite{HRZ23} 
	established that $\mu_t$ sits in a balanced duality with $b$, the resulting drift component $b *\mu$ belonging to $L^s([0,T];L^\infty(\mR^{2d}))$ for some appropriate \textcolor{Dgreen}{$s>2$}. \textcolor{Dgreen}{(See also \cite{CdRJM-22,CdRJM-23} for similar results in the non-degenerate setting.)} 
	
	Empirically, Equation \eqref{MV00} describes the momentum of a generic body located at the position $X_t$ at any time $t$, evolving in \textcolor{Dgreen}{a} (possibly anomalous for $\alpha\in (1,2) $) medium and subject to the action of the distribution dependent force field $b_t*\mu_t$, where $b$ models a given interaction kernel. Formally, \eqref{MV00} arises as the mean field limit of the interacting particle system: 
	\begin{align}\label{MV00_PART}
	\left\{
	\begin{aligned}
		&\ddot X^{N,i}_t=\frac 1{N}\sum_{j=1}^N b_t(X^{N,i}_t-X^{N,j}_t,\dot X^{N,i}_t-\dot X^{N,j}_t)+\dot \Levy^{\nalpha,i}_t,\\ 
		&(X^{N,i}_0,\dot X^{N,i}_0)\sim \mu_0, \ \ 1\le i\le N,\ \ 0\le t\le T.
	\end{aligned}
	\right.
	\end{align}

	Introducing the velocity component $V_t:=\dot X_t$, \eqref{MV00} can be written as the system of first-order degenerate SDEs: 
	\begin{align}\label{MV1}
		\left\{
		\begin{aligned}
			&X_t=X_0+\int_0^t V_s d  s, \ \ 0\le t\le T,\\
			&V_t=V_0+\int^t_0(b_s*\mu_{s})(X_s,V_s) d  s+ \Levy^{\nalpha}_t.
		\end{aligned}
		\right.
	\end{align}
	If $b$ does not depend on the position variable $x$ 
	(i.e. $b_{t}(x,v)=b_{t}(v)$)
	then SDE \eqref{MV1} reduces to the following (first-order) non-degenerate (autonomous) McKean-Vlasov SDE:
	\begin{align}\label{MV2}
		 d  V_t=(b_t*\mu_t)(V_t) d  t+d\Levy^{\nalpha}_t, \ \ 0\le t\le T.
	\end{align}
	Formally, the one-time marginal laws $\{\mu_t\}_{t\ge 0}$ of $Z_t:=(X_t,V_t)$ give a distributional solution to the following kinetic nonlinear Fokker-Planck equation:
	\begin{align}\label{S1:FPE}
		\text{for a.e.} \ 0\le t\le T,\ \ \p_t \mu_t=(\Delta^{\frac{\nalpha}{2}}_v-v\cdot\nabla_x)\mu_t -\div_v ((b_t*\mu_t)\mu_t)\ \text{on}\ \R^{2d}. 
	\end{align}
	Establishing wellposedness of \eqref{S1:FPE}, in the class of anisotropic Besov spaces mentioned above and described in Section \ref{SEC_BESOV}, is a key step to derive weak and/or strong wellposedness of the nonlinear SDE \eqref{MV00}. 
	From a modeling point of view, while Gaussian noise remains central for the representation of diffusion for molecular motions, the interest for considering more general stable noises has been growing in many applicative fields such as  Physics, Biology or advances in deep learning. Experiments have notably evidenced that much more general non-Gaussian noises are at play in nature. We can for instance refer to heavy tailed and the ``jump and tumble" phenomena in cell motions, \cite{Perthame-04}; fractal or L\'evy flights pattern in turbulence mode\textcolor{Dgreen}{l}ling, \cite{ShZaFr-95} and the interest to learning processes based on stochastic gradient descent methods which naturally appear in the occurrence of fast and large excursions facilitating the exploration of multiple  minima, \cite{SiSaGu-19}, \cite{ZhLoWu-22}. Likewise, the interest for stochastic kinetic models naturally appears in the microscopic description of aggregative social and economical population dynamics (\cite{ParTos-14} and references therein), cell motions (see again \cite{Perthame-04}), in Computational Fluid Dynamics and the Lagrangian mode\textcolor{Dgreen}{l}ing of turbulent flows (\cite{Pope-00}, \cite{BBCJR-10}) or for the design of underdamped stochastic gradient descents (\cite{CCBJ-18}).

\noindent
	
	Let us point out that, for a singular kernel, giving a rigorous meaning to the formal particle system in \eqref{MV00_PART} is a rather involved task. For Lebesgue spaces this has been done in a non degenerate (i.e. no dependence in the $x$ variable for the kernel $b$) Brownian setting  in \cite{toma:23} and in \cite{HHMT-20} through Girsanov type arguments, and in \cite{HRZ22} through a Picard linearization approximation. It remains an open problem to extend the result to other stable noises in that setting or to more singular drifts.

A natural approach to circumvent this difficulty consists in introducing a moderately interacting version of the particle systems, where the initial interaction kernel is replaced by its convolution with an appropriate mollifier.
	This mollification operation provides the advantage that there are no issues about the definition of the corresponding particle system which is intrinsically well-posed. Propagation of chaos properties can be then estimated by attuning the order of the mollification with the particles size $N$. This procedure has been methodically employed in the literature, from the seminal papers \cite{Oe84,Oe85,Oe87} (from which the terminology {\it moderately interacting particle} used thereafter originates) to the recent works \cite{FOS20} and \cite{OlRiTo-21}. Notably, in \cite{Oe85,Oe87} (see also \cite{JM98}), the author derives propagation of chaos results from a fluctuation analysis between a (non-degenerate) case of McKean-Vlasov SDE with local interaction and a moderately interacting particle approximation. To briefly illustrate this derivation, taking \eqref{MV2} as a toy model, fluctuations are established from a non-asymptotic control of the distance, set in a suitable functional space, between $\text{Law}(V_t)$ and $\phi_N* \frac 1N\sum_{i=1}^N\delta_{\{V^{N,i}_t\}}$ ($\phi_N$ denoting the mollifier \textcolor{Dgreen}{for} the interaction). This way of capturing propagation of chaos naturally quantifies the cost of the mollification of the interactions, and is further natural for numerical applications and the validation of stochastic particle methods for nonlinear PDEs (see \cite{BT97,BFP97}, for the particular case of the one-dimensional Burgers equation). The more recent paper \cite{OlRiTo-21} propose\textcolor{Dgreen}{s} a further generalisation of Oelschl\"ager's approach and establish\textcolor{Dgreen}{es} original quantitative propagation of chaos results for cut-off moderately interacting particle systems with initial local $L^p$- or Riesz-like - interaction kernels (we may also refer to \cite{OlRiTo-22} for more particular physical and biological applications). 
	 Our results below further systematize the original setting in \cite{OlRiTo-21}, addressing kinetic dynamics, distributional interaction kernels and enclosing, in a consistent way, driving noises ranging from the classical Brownian case to pure-jump situations. {\black (Regarding the specific case of kinetic systems driven by Brownian motion, we may refer to the recent paper \cite{BC25} which address original propagation of chaos results related to sufficiently regular interaction kernels.)} While some current technical issues (that we briefly discuss in Remark \ref{REM_DRIFT}) restrict the spread of our results, {\color{black}a few applications are discussed in Section \ref{sec:4}. Further} characteristic thresholds (e.g. \cite{JW16,BJS-25}) and practical instances of singular stochastic kinetic dynamics will be addressed in future work.
	
	To analyse quantitatively propagation of chaos \textcolor{Dgreen}{for} a mollified version of \eqref{MV00_PART}, the procedure  
	consists in comparing the Duhamel representation of the Fokker-Planck equation \eqref{S1:FPE} and the expansion of the convolution of the empirical measure $\frac 1N\sum_{i=1}^N\delta_{\{X^{N,i}_t,V^{N,i}_t\}}$ with the underlying mollifier evaluated along the transport associated with the kinetic operator (see Section \ref{sec:ModeratedPropagation} for details). This comparison somehow amounts to measure a {\it moderate propagation of chaos} (see Theorem \ref{main01} below) and can consequently be used to derive pathwise and weak propagation of chaos.  The expansion yields an SPDE. There are then {\black two sources of error between the McKean-Vlasov SDE and the interacting particle system:} 
		\begin{trivlist}	
			\item[-] the first one, which is of \textit{stability type} and roughly corresponds to the difference between the Duhamel expansion and the drift of the SPDE, is handled through a priori regularity controls on the Fokker-Planck equation. Such controls were obtained in the current kinetic setting in \cite{HRZ23}, and we can also refer to \cite{CdRJM-22,CdRJM-23} in the non-degenerate case. We can mention that SDEs with non-regular coefficients have been thoroughly studied, i.e. for time inhomogeneous distributional drifts without non-linear dependence on the law. We can \textcolor{Dgreen}{cite} e.g. \cite{flan:isso:russ:17}, \textcolor{Dgreen}{\cite{athr:butk:mytn:18}} or \cite{CdRM-22} in the Young regime and \textcolor{Dgreen}{\cite{dela:diel:16}} or \cite{krem:perk:20} in which the authors manage to go beyond the thresholds appearing in the previous references assuming some appropriate underlying rough path structure. In that \textcolor{Dgreen}{last} setting defining the drift is a rather delicate point. Importantly, the current McKean-Vlasov \textcolor{Dgreen}{framework}, which tackles singular kernels, also benefits from the regularizing effects of the law through the convolution. In particular, from the smoothness of the law, it was derived in the quoted work\textcolor{Dgreen}{s} \textcolor{Dgreen}{\cite{HRZ23}, \cite{CdRJM-22,CdRJM-23}} that the nonlinear drift could eventually be seen as a \textit{usual} drift in some Lebesgue space from which some strong uniqueness results, that are important for pathwise propagation of chaos (see e.g. Theorem \ref{S1:main01} below),  were derived.
			\item[-] the second source is associated with a stochastic integral coming from the expansion of the empirical measure. This  term  needs to be evaluated in a suitable function space, i.e. the one for which the error is investigated. This naturally induces consider\textcolor{Dgreen}{ing} martingale type inequalities in Banach spaces of $M$-type or enjoying the UMD  (unconditional martingale difference) property. We refer to Section \ref{sec:Type} and \cite{Pisier-16} for further details.
		\end{trivlist}
		The paper is organized as follows: In Sections \ref{sec:1.2} and \ref{SEC_BESOV}, we introduce the moderately interacting
		particle approximation of Equation \eqref{MV1} and essential preliminaries on anisotropic Besov spaces, including key properties that will be used throughout the paper. Our main results on weak and strong propagation of chaos properties of the moderately interacting particle systems are stated 	\textcolor{Dgreen}{in Theorem \ref{S1:main01}} in Section \ref{sec:Main}.  
		The core of the corresponding proof is presented in Section \ref{sec:ModeratedPropagation}, where we introduce and prove	 \textcolor{Dgreen}{the aforementioned {\it moderate propagation of chaos} in} Theorem \ref{main01}. The proof of Theorem \ref{main01} is given in Section \ref{sec:2.2}, along some auxiliary estimates (Lemmas \ref{lem:HN} and \ref{lem:MN}) effectively proven in Section \ref{sec:ProofMain}.
		 From Theorem \ref{main01}, the proof of Theorem \ref{S1:main01} is addressed in Section \ref{sec:ProofMainThm}. 
		 	In the final part of the paper, we append extra-technical results essentially related to anisotropic Besov spaces and used throughout the paper. These results are: about weighted anisotropic Besov and related heat kernel estimates (Appendix \ref{sec:WeightedEstimates}); about the scaling properties of mollifiers under the norm of the anisotropic Besov space (Appendix \ref{APP_BESOV}); sampling error of the initial condition (Appendix \ref{CONV_EMP_SAMP_MU0}); functional properties relating the space $\bB^{\beta}_{\bbp;\bba}$ \textcolor{Dgreen}{to UMD and $M$-type Banach spaces} (Appendix \ref{sec:Type}); and Gronwall's inequality of Volterra type (Appendix \ref{sec:AppD}).

\subsection{Moderately interacting particle systems }\label{sec:1.2}

Consider the following particle system, with \textit{\textcolor{Dgreen}{ moderate interactions}}, as an approximation for the McKean-Vlasov SDE \eqref{MV1}:
\begin{align}\label{S1:00}
	\left\{
	\begin{aligned}
		&X^{N,i}_t=\xi^1_i+\int_0^t V^{N,i}_s\, d  s,\ \ i=1,2,\cdots,N,\\
		& V^{N,i}_t=\xi^{2}_i+\int_0^t (b^N_s*\mu^N_s)(Z^{N,i}_s)\, d  s+\Levy^{\nalpha,i}_t, 
	\end{aligned}
	\right.
\end{align}
where $Z^{N,i}_t:=(X^{N,i}_t, V^{N,i}_t)$, $Z^{N,i}_0=\xi_i=(\xi^1_i,\xi^2_i)$, $\mu^N_t:=\frac1N\sum_{i=1}^N\delta_{\textcolor{Dgreen}{\{}Z^{N,i}_t\textcolor{Dgreen}{\}}}$,  and 
\begin{equation}\label{MollifiedInteraction}
b^N_t(z):=b_t*\transOp_t\phi_N(z),\ \ \transOp_t f(x,v):=f(x-tv,v).
\end{equation}
Here, the $\{\Levy^{\nalpha,i}\}_{i=1}^N$ are independent copies of  $\Levy^\nalpha$, and $\{\xi_i\}_{i=1}^N$ is a family of i.i.d.  
$\mR^{2d}$-valued random variables with common law $\mu_0$.
Eventually, for some $\modulateorder\in(0,1]$
 and a smooth compactly supported
 and symmetric probability density function $\phi:\mR^{2d}\to[0,\infty)$, we define the mollifier
\begin{align}\label{PHN}
 \phi_N(x,v):=N^{(2+\alpha)\modulateorder d}\phi(N^{(1+\alpha)\modulateorder} x, N^{\modulateorder} v).
\end{align}
The scales in the mollifier reflect the homogeneity of the underlying distance,  see \eqref{AnisoDist}, which will be used to define the corresponding anisotropic Besov spaces. {\black In parallel, the parameter $\zeta$ controls the strength of the mollification, introduced in the particle system, which is needed to regularize the irregular drift $b$.} 
The operator $\Gamma_t$ represents the characteristic kinetic transport 
\textcolor{Dgreen}{and its application to $\phi_N$ will become relevant in our computations later on. Note that the presence of the composition $\Gamma_t\phi_N$ does not alter the nature of the mollifier as, in the distributional sense, $\lim_{N\to\infty}\Gamma_t \phi_N=\delta_{\{0\}}$.}

\subsection{Anisotropic Besov space and kinetic semi-group}\label{SEC_BESOV}
In this section we recall the definition of anisotropic Besov spaces with mixed integrability indices as well as their  
basic properties (see \cite{trie:14}, \cite{ZZ21} and \cite{HRZ23}).

For a multi-index $\bbp=(p_x,p_v)\in[1,\infty]^2$, we first define the Bochner-type iterated space $\mL^{\bbp}=L^{p_v}(\R^d;L^{p_x}(\R^d))=\{f:\R^{2d}\rightarrow\R \ \text{Borel measurable}:\,\|f\|_{\mL^\bbp}<\infty\}$ for \textcolor{Dgreen}{$L^p(\mathbb R^d)$ denoting the Lebesgue space on $\mathbb R^d$ and}
\begin{align}\label{LP1}
	\|f\|_{\mL^\bbp}:=\|f\|_{\bbp}:=\left(\int_{\mR^d}\|f(\cdot,v)\|_{p_x}^{p_v} d  v\right)^{1/p_v}.
\end{align}
\textcolor{Dgreen}{We adopt here} the \textcolor{Dgreen}{usual} convention \textcolor{Dgreen}{when one integrability index or the two integrability indices are $\infty$. That is,} 
	   for $\bbp=(p_x,\infty)$,
$\|f\|_{\bbp}:=\text{esssup}_{v}\|f(\cdot,v)\|_{p_x}$ \textcolor{Dgreen}{and $\bbp=(\infty,\infty)$, the $\mL^\bbp$-norm corresponds to $\text{esssup}_{v,x}\|f(x,v)\|$}.\comm{JF 01/04: Just a small rewriting to remove any possible confusion: originally, only the case $\bbp=(p_x,\infty)$ was illustrated, falsely suggesting that it will be of interest, the most frequently used case yet being $\bbp=(\infty,\infty)$. Side note: We have been often switching between the notations $\|\cdot\|_{\bbp}$ and $\|\cdot\|_{\mL^\bbp}$. Now the paper is more or less ok on this point, the notation $\|\cdot\|_{\bbp}$ being prominently used in the paper except when we use the iterated space $\L^m(\Omega;\mL^\bbp)$ is used, e.g. in the proofs of the weak-strong rates \eqref{AA200}, \eqref{AA300} and in the estimates for the initial chaos (Appendix \ref{CONV_EMP_SAMP_MU0}) and the "martingale" (Section \ref{sec:341}))} Note that the above $\mL^\bbp$-norm 
is invariant under \textcolor{Dgreen}{the} translation operator $\transOp_t$. 
Namely  
\begin{align}\label{AA1}
	\|\transOp_t f\|_{\bbp}=\left(\int_{\mR^d}\|f(\cdot-tv,v)\|_{p_x}^{p_v} d  v\right)^{1/p_v}=
	\left(\int_{\mR^d}\|f(\cdot,v)\|_{p_x}^{p_v} d  v\right)^{1/p_v}=\|f\|_{\bbp}.
\end{align}
This property will be \textcolor{Dgreen}{extensively} used and the operator $\transOp_t$ naturally appears in the current degenerate setting, as it encodes the transport associated with the (linear) first order vector field \textcolor{black}{in \eqref{S1:FPE}}. Plugging it into the mollifier will actually allow to get rid of the degenerate terms in the stability analysis of Section \ref{sec:ModeratedPropagation}.

Let $\bba$ be the scaling parameter given by
\begin{align}\label{DEF_A}
	\bba:=(1+\nalpha,1).
\end{align}
When $\alpha=2 $, \textcolor{Dgreen}{i.e. \textcolor{black}{in} the} Brownian case, the vector $\bba $ encodes the time exponents of the variances of the entries of $(\int_0^t Z_s d  s,Z_t)=\sqrt 2(\int_0^t W_s  d  s,W_t) $, i.e. $\E[(\int_0^t W_s  d  s)^2]=\textcolor{Dgreen}{\frac{2t^3}{3}}, \E[W_t^2]=2t $. When $\alpha\in (1,2) $, $\bba $ still appears when considering invariance or scaling properties of the joint density (see \eqref{SCL1} below).

For notational convenience, we shall write from here on:
\begin{align*}
	\tfrac{1}{\bbp}:=\Big(\tfrac{1}{p_x},\tfrac1{p_v}\Big),\qquad  \bba\cdot \tfrac d\bbp:=\tfrac{(1+\nalpha)d}{p_x}+\tfrac{d}{p_v},
\end{align*}
for any $\bbp,\bbq\in[1,\infty]^2$,
\begin{align}\label{KJ1}
\bbp\ge\bbq\Leftrightarrow p_x\ge q_x,\ p_v\ge q_v,\qquad \cA_{\bbp,\bbq}:=\bba\cdot(\tfrac d\bbp-\tfrac d\bbq),
\end{align}
\textcolor{Dgreen}{as well as}
$$
\bbp\vee \textcolor{Dgreen}{\bb2}
=(p_x\vee 2,p_v\vee 2),\ \ \bbp\wedge \textcolor{Dgreen}{\bb2}
=(p_x\wedge 2,p_v\wedge 2).
$$
Similarly, we use bold \textcolor{Dgreen}{symbols} to denote constant vectors in $\mR^2$, \textcolor{Dgreen}{that is},
\begin{align*}
	\bbb1:=(1,1),\quad\textcolor{Dgreen}{\boldsymbol{2}=(2,2),}\quad \bbd:=(d,d),\quad\textcolor{Dgreen}{\boldsymbol{\alpha}:=(\alpha,\alpha), \ \boldsymbol{\infty}:=(\infty,\infty).}
\end{align*}
\textcolor{Dgreen}{Finally, we} use \textcolor{Dgreen}{the symbol} $\bbp'$ to denote the conjugate index of $\bbp$, i.e.,
\begin{equation}\label{KineticConjugate}
\tfrac{1}{\bbp}+\tfrac{1}{\bbp'}=\bbb1 \,\textcolor{Dgreen}{\Leftrightarrow \bbp'=(p_x',p_v')\,\text{with}\,\frac{1}{p_x}+\frac{1}{p_x'}=1=\frac{1}{p_v}+\frac{1}{p_v'}}.
\end{equation}
\textcolor{Dgreen}{These general notation\textcolor{black}{s} settled, we shall now introduce the class of anisotropic Besov spaces $\bB^{\beta,q}_{\bbp;\bba}$ characterizing the singularity of the interaction kernel $b$. To this aim, let us introduce some preliminaries\textcolor{black}{.}}\comm{JF 01/04: Addition of a small sentence, for the transition between the listing of our notation convention and the definition of the anisotropic Besov space.}
For an $L^1$-integrable function $f$ \textcolor{Dgreen}{o}n $\mR^{2d}$, let $\hat f$ be the Fourier transform of $f$ defined by
$$
\hat f(\xi):=
\int_{\mR^{2d}} \e^{-{\rm i}\xi\cdot z}f(z) d  z, \quad\xi\in\mR^{2d},
$$
and $\check f$ the Fourier  inverse transform of $f$ defined by
$$
\check f(z):=(2\pi)^{-\textcolor{Dgreen}{2}d}\int_{\mR^{2d}} \e^{{\rm i}\xi\cdot z}f(\xi) d \xi, \quad z\in\mR^{2d}.
$$
For $z=(x,v)$ and $z'=(x',v')$ in $\mR^{2d}$, we introduce the anisotropic distance 
\begin{equation}\label{AnisoDist}
	|z-z'|_\bba:=|x- x'|^{1/(1+\alpha)}+|v-v'|. 
\end{equation}
Note that $z\mapsto|z|_{\bba}$ is not smooth at the origin. This distance appears very naturally in connection with the homogeneity of the underlying linear degenerate operator. In the kinetic setting we can refer to the work by Priola \cite{prio:09}, or \cite{CheZha-19, HWZ20, IS21, HZZZ22, HRZ23}, who employed this distance \textcolor{Dgreen}{for establishing } Schauder type estimates. In the current work, it will be used in order to define the corresponding anisotropic Besov spaces (see Definition \ref{bs} below).

For $r>0$ and $z\in\mR^{2d}$, we also introduce the ball centred at $z$ and with radius $r$ with respect to the above distance
as follows:
$$
B^\bba_r(z):=\{z'\in\mR^{2d}:|z'-z|_\bba\leq r\},\ \ B^\bba_r:=B^\bba_r(0).
$$
Let $\chi^\bba_0$ be  a symmetric $C^{\infty}$-function  on $\mR^{2d}$ with
$$
\chi^\bba_0(\xi)=1\ \mathrm{for}\ \xi\in B^\bba_1\ \mathrm{and}\ \chi^\bba_0(\xi)=0\ \mathrm{for}\ \xi\notin B^\bba_2.
$$
We define
$$
\phi^\bba_j(\xi):=
\left\{
\begin{aligned}
	&\chi^\bba_0(2^{-j\bba}\xi)-\chi^\bba_0(2^{-(j-1)\bba}\xi),\ \ &j\geq 1,\\
	&\chi^\bba_0(\xi),\ \ &j=0,
\end{aligned}
\right.
$$
where for $s\in\mR$ and $\xi=(\xi_1,\xi_2)$,
$$
2^{s\bba }\xi=(2^{(1+\alpha)s}\xi_1, 
2^{s}\xi_2).
$$
Note that
\begin{align}\label{Cx8}
	{\rm supp}(\phi^\bba_j)\subset\big\{\xi: 2^{j-1}\leq|\xi|_\bba\leq 2^{j+1}\big\},\ j\geq 1,\ {\rm supp}(\phi^\bba_0)\subset B^\bba_2,
\end{align}
and
\begin{align}\label{AA13}
	\sum_{j\geq 0}\phi^\bba_j(\xi)=1,\ \ \forall\xi\in\mR^{2d}.
\end{align}

Let $\cS$ be the space of all Schwartz functions on $\mR^{2d}$ and $\cS'$ be the dual space of $\cS$, \textcolor{Dgreen}{that is} the tempered distribution space.
For given $j\geq 0$, the  dyadic anisotropic block operator  $\mathcal{R}^\bba_j$ is defined on $\cS'$ by
\begin{align}\label{Ph0}
	\mathcal{R}^\bba_jf(z):=(\phi^\bba_j\hat{f})\check{\ }(z)=\check{\phi}^\bba_j*f(z),
\end{align}
where the convolution is understood in the distributional sense and by scaling,
\begin{align}\label{SX4}
	\check{\phi}^\bba_j(z)=2^{(j-1)(2+\nalpha)d}\check{\phi}^\bba_1(2^{(j-1)\bba}z),\ \ j\geq 1.
\end{align}
Similarly, we can define the \textcolor{Dgreen}{classical} isotropic block operator $\cR_jf=\check\phi_j*f$ in $\mR^d$, \textcolor{Dgreen}{for $\{\phi_j\}_{j\ge 0}$ a smooth partition of unity of $\mR^d$} where
\begin{align}\label{Cx9}
	{\rm supp}(\phi_j)\subset\big\{\xi: 2^{j-1}\leq|\xi|\leq 2^{j+1}\big\},\ j\geq 1,\ {\rm supp}(\phi_0)\subset B_2\textcolor{Dgreen}{(0)}.
\end{align}
Now we introduce the following anisotropic Besov spaces \textcolor{Dgreen}{and mixed (anisotropic) Besov spaces}.
\begin{definition}\label{bs}
	Let $s\in\mR$, $q\in[1,\infty]$ and $\bbp\in[1,\infty]^2$. The  anisotropic Besov space is defined by
	$$
	\mathbf{B}^{s,q}_{\bbp;\bba}:=\left\{f\in \cS'\textcolor{Dgreen}{(\mR^{2d})}\ : \  \|f\|_{\mathbf{B}^{s,q}_{\bbp;\bba}}
	:= \left(\sum_{j\geq0}\big(2^{ js}\|\cR^\bba_{j}f\|_{\bbp}\big)^q\right)^{1/q}<\infty\right\},
	$$
	where $\|\cdot\|_\bbp$ is defined in \eqref{LP1}. Similarly, one defines the usual  isotropic Besov spaces  $\bB^{s,q}_{p}$ 
	in $\mR^d$ in terms of isotropic block operators $\cR_j$. If there is no confusion, we shall write
	$$
	\bB^s_{\bbp;\bba}:=\bB^{s,\infty}_{\bbp;\bba}.
	$$
	For  $s_0,s_1\in\mR$, the mixed Besov space is defined by
	$$
	\bB^{s_0,s_1}_{\bbp;x,\bba}:=\left\{f\in \cS'\textcolor{Dgreen}{(\mR^{2d})}\ : \ \|f\|_{\mathbf{B}^{s_0,s_1}_{\bbp;x,\bba}}
	:= \sup_{k,j\geq 0}2^{\frac{ks_0}{1+\alpha}}2^{ js_1}\|\cR^x_k\cR^\bba_{j}f\|_\bbp
	<\infty\right\},
	$$
	where for \textcolor{Dgreen}{any
	} $f:\mR^{2d}\to\mR$, 
	$$
	\cR^x_kf(x,v):=\cR_k f(\cdot,v)(x).
	$$
\end{definition}
\textcolor{Dgreen}{The space $\bB^{s}_{\bbp;\bba}$ (with $s<0$) will characterize the class of distributional interaction kernel $b$ and the irregularity-integrability ranges of $\mu_0$.} The mixed Besov space $\bB^{s_0,s_1}_{\bbp;x,\bba}$ will be essentially used in the study of strong convergence for the propagation of chaos.

\subsection{Main results}\label{sec:Main}
To state our main results, we introduce the following assumptions:
\textit{
\begin{enumerate}[{\bf (H)}]
    \item  \textit{$b\in L^\infty(\mR_+;\bB^{\beta_b}_{\bbp_b;\bba})$ and $\mu_0\in \mathcal P(\R^{2d}) \cap \bB^{\beta_0}_{\bbp_0;\bba}$ with $\bbp_0\textcolor{Dgreen}{=(p_{x,0},p_{v,0})}\in [1,\infty]^2$, 	$\bbp_b\textcolor{Dgreen}{=(p_{x,b},p_{v,b})}\in[1,\infty{\color{black}]}^2$ 
     \textcolor{Dgreen}{satisfying} $\frac1{\bbp_0}+\frac1{\bbp_b}\ge\bbb1$, and $\beta_b\le0$ and $\beta_0\in(-1,0)$ \textcolor{Dgreen}{are such that} 
    \begin{align}
        0<\Gap:=\cA_{\bbp_0,\bbp_b'}-\beta_0
        -\beta_b=\bba\cdot\tfrac{d}{\bbp_0}-\beta_0+\bba\cdot\tfrac{d}{\bbp_b}-\beta_b-(2+\nalpha)d
        <\alpha-1.\label{DEF_GAP}
    \end{align}}
    \end{enumerate}
}

We shall use the following parameter set:
$$
\Theta:=\big(\alpha,d, \bbp_0,\bbp_b,\beta_0,\beta_b,\|\mu_0\|_{\bB^{\beta_0}_{\bbp_0;\bba}},\|b\|_{
	\textcolor{Dgreen}{L}^\infty(\mR_+;\bB^{\beta_b}_{\bbp_b;\bba})}\big),
$$
\textcolor{Dgreen}{to highlight, when pertinent, the dependency of estimates on the parameters $\alpha$, $d$, etc., e.g. w}hen we say $C=C(\Theta)$, it means that the constant $C$ may depend on all or part of the parameters in $\Theta$. \textcolor{Dgreen}{Additional dependency on the order $\modulateorder$ of the mollification or other relevant parameters will be emphasized in the same way.}

Under the above assumptions,  by \cite[Theorem 1.3]{HRZ23}, there exists a finite time horizon $\Tup=\Tup(\Theta)>0$ for which the McKean-Vlasov SDE \eqref{MV1} admits a unique weak solution $Z_t\textcolor{Dgreen}{=(X_t,V_t)}$ on the 
time interval $[0,\Tup]$. Under the stronger condition that $b$ is in  $L^\infty([0,\Tup],\bB^{s_0,s_1}_{\bbp_0;x,\bba})$, where $\bB^{s_0,s_1}_{\bbp_0;x,\bba}$ is as in Definition \ref{bs},
 {\black and the regularity parameters $s_0$, $s_1$ satisfy }
\begin{equation}\label{StrongSol}
s_0=1+\tfrac\alpha2,\qquad s_1=\tfrac \alpha{2}-1+\beta_b,\qquad \Gap\in\Big(0,\tfrac {3}{2}\alpha-2\Big),
\end{equation}
{\black where $\Gap$ is as in {\bf (H)},}
pathwise uniqueness \textcolor{Dgreen}{also} holds (see again \cite[Theorem 1.3]{HRZ23}). 

In order to give the range of $\modulateorder$ in the definition of \textcolor{Dgreen}{the} mollifier $\phi_N$, we introduce two quantities $\mstable_\alpha$ and $\theta_\alpha$ for later use 
\begin{align}\label{limitstable}
	\mstable_\alpha:=1/((p_{x,0}\wedge p_{v,0}\wedge 2)\vee\alpha),\quad \theta_\alpha:={\color{black}\big[\cA_{{\bbb1},\bbp_0\vee{\boldsymbol\alpha}}\big]\vee \big[\cA_{\bbp_b,{\boldsymbol\infty}}-(1+\alpha)\beta_b\big],\ \alpha\in(1,2].} 
\end{align}

The main result of this paper is encompassed in the following:
\bt\label{S1:main01}
Suppose that {\bf (H)} holds, and that either $\bbp_0>\boldsymbol{\alpha}$, or $(\alpha,\bbp_0)=(2,\bbb1)$.
{\black Let $\{\xi_i\}_{i=1}^N$ be, as in \eqref{S1:00}, a family of i.i.d.  
$\mR^{2d}$-valued random variables with common law $\mu_0$, and assume that $\mu_0$ satisfies}:
$$
\mu_0(|\cdot|^{m}_\bba)\textcolor{Dgreen}{:=\mathbb E[|(X_0,V_0)|^m_{\bba}]}<\infty,\ \ \forall m\in\mN.
$$
Then, for any 
$\beta,\zeta>0$ with the following upper bounds
$$
\beta<\textcolor{orange}{\bar \beta_\alpha:=\begin{cases}1-\Lambda\ \textcolor{Dgreen}{\text{if}}\ \alpha=2,\\
(\alpha-1-\Lambda)\wedge((\textcolor{Dgreen}{\alpha+\beta_0}-\Lambda)/2)\ \textcolor{Dgreen}{\text{if}}\ \alpha\in (1,2),\end{cases}} \ 
{\rm and }\ \zeta<(1-\mstable_\alpha)/\theta_\alpha,
$$
\textcolor{Dgreen}{and, for any $\eps>0$}
there is a constant $C=C(\Theta\textcolor{Dgreen}{,m},\beta,\modulateorder,\eps)>0$ such that for all \textcolor{Dgreen}{$N\ge 1$}, and $t\in[0,\Tup]$,
\begin{align}\label{AA200}
\|\cL(Z^{N,1}_t)-\cL(Z_t)\|_{\rm var}\le&   CN^{-\beta\zeta}+CN^{\mstable_\alpha-1+\zeta\theta_\alpha+\eps}, 
\end{align}
where $\|\cdot\|_{\rm var}$ is the total variation distance on $\mathcal P(\R^{2d})$, {\black where $Z^{N,1}=(X^{N,1},V^{N,1})$ and $Z = (X, V)$ denote the solutions to the particle system~\eqref{S1:00} and the McKean-Vlasov SDE~\eqref{MV1}, respectively.} 
Moreover, if in addition $1-\alpha/2<(\alpha+\beta_0-\Lambda)/2$, $b$ lies in the mixed Besov space $ L^\infty(\mR_+;\bB^{s_0,s_1}_{\bbp_b;x,\bba})$ and \eqref{StrongSol} holds,  then for any \textcolor{Dgreen}{$\zeta$ as above and for $\beta$ in the range
}
$$
\beta\in(1-\alpha/2,\textcolor{orange}{\bar \beta_\alpha}
),
$$
\textcolor{Dgreen}{we have}
\begin{align}\label{AA300}
	\left\|\sup_{t\in[0,T_{0}]}|Z^{N,1}_t-Z^1_t|\right\|_{L^2(\Omega)}\le&  CN^{-\beta\zeta}+CN^{\mstable_\alpha-1+\zeta\theta_\alpha+\eps},
\end{align}
where $Z^1\textcolor{Dgreen}{=(X^1,V^1)}$ is the strong solution of \eqref{MV1} driven by the noise  $\Levy^{\alpha,1}$ and starting from $\xi_1$.
\et
\br[\textit{\textcolor{orange}{About the convergence rates.}}]\label{RK_RATES}
\textcolor{orange}{The convergence rate in the previous Theorem is mainly associated  with the characteristic parameters $\beta $ and $\mstable_\alpha $ (we recall that $\zeta $ appeared in the definition of the mollifier in the particle system, see \eqref{S1:00}-\eqref{PHN}).
\begin{trivlist}
\item[-] For the parameter $\beta $\textcolor{Dgreen}{,} two thresholds appear in the upper-bound. The first one, $\alpha-1-\Lambda $, comes from the smoothness of the underlying Fokker-Planck equation \textcolor{Dgreen}{\eqref{S1:FPE}} and corresponds to the \textit{gap} in the condition \eqref{DEF_GAP} (see e.g. \eqref{TO_EXPLAIN_FIRST_CTR_BETA}). The other contribution, $(\alpha+\beta_0-\Lambda)/2$, is related to time integrability issues associated with the norm for which we estimate the error. It only appears in the \textit{pure} stable case due to the constraint $\bbp_0>\boldsymbol \alpha $ whereas considering $\bbp_0=\mathbf 1$ in the Brownian case $ \alpha=2$ spares this contribution. The division by 2 follows from the quadratic nature of the nonlinearity in \eqref{S1:FPE} which also appears for the error analysis when $\bbp_0\neq\mathbf 1$ (see the proof of Theorem \textcolor{MColor}{\ref{main01}} p. \pageref{THE_PROOF_OF_THM_MAIN}).\\
\hspace*{.2cm}\\
On the other hand, when considering the pathwise approach, a lower bound is also needed in order to guarantee the sufficient smoothness on the underlying PDE to apply a Zvonkin type argument (see \eqref{S6:PDE},\eqref{S60}).
\item[-] For the parameter $\mstable_\alpha $ the range mainly corresponds to that of a limit Theorem. It actually appears from the control of a stochastic integral that appears in the SPDE satisfied by the error (see as well Theorem \ref{main01} and \eqref{Duh}, \eqref{MART_FIELD_I}). If $\alpha=2 $ then $ \mstable_\alpha=\frac 12$, which is the usual rate in a Gaussian central limit theorem. For $\alpha\in (1,2) $, since $\bbp_0>\boldsymbol \alpha $, $\mstable_\alpha=1/((p_{x,0}\wedge p_{v,0}\wedge 2)$. If $\bbp_0$ is somehow \textit{close} to $\boldsymbol \alpha$ then $\mstable_\alpha$ is close to $\frac 1\alpha $ yielding a \textit{principal rate} which is close to the stable limit theorem, i.e. in that case $1-\mstable_\alpha\simeq 1-\frac 1\alpha$. On the other hand, for \textit{large} values of $\bbp_0 $, the Gaussian regime again prevails.\\
In any case, the previous rates need to be deflated (for integrability reasons and because of the underlying singular term).  This naturally again makes the parameter $\zeta $ appear (see Lemma \ref{lem:MN}, \eqref{Duh}, \eqref{MART_FIELD_I}). Eventually, the parameter $\theta_\alpha $ is constrained by the choice of the function space in which we analyze the error that actually allows to exploit suitable controls on the underlying Fokker-Planck equation \eqref{S1:FPE}. 
\item[-] Balancing the two previous errors terms: in order that the two errors have the same range we have to set
$$N^{-\beta\zeta}=N^{\mstable_\alpha-1+\zeta\theta_\alpha+\eps} \iff \zeta(\theta_\alpha+\beta)=1-\mstable_\alpha-\eps\iff \zeta=\frac{1-\mstable_\alpha-\eps}{(\theta_\alpha+\beta)},$$
yielding a final bound in 
\begin{equation}\label{GeneralOptimalRate}
N^{-\frac{\beta}{\theta_\alpha+\beta}(1-\mstable_\alpha-\eps)} .
\end{equation}
\item[$\bullet $] For $\alpha=2 $, {\color{black}let us consider the case} $\bbp_0={\color{black}\bbb1}$ {\color{black}so that} $\Lambda=-\beta_0+\bba\cdot\tfrac{d}{\bbp_b}-\beta_b $. In that case $\Lambda$ can be arbitrarily small provided $ \beta_0,\beta_b$ and $\bbp_{\color{black}b} $ are respectively \textit{small} and \textit{large} enough{\color{black}. This makes $b$ close to a bounded interaction kernel and, as $\bbp_b=\boldsymbol{\infty}$, this yields, for $\theta_\alpha$,
	$$
	\cA_{\bbb1,\bbp_0\vee{\boldsymbol\alpha}}=\cA_{\bbb1,\bb2}= 2d,\quad \cA_{\bbp_b,\boldsymbol{\infty}}-(1+\alpha)\beta_b\simeq 0 \Rightarrow \theta_\alpha\simeq 2d.
	$$
	} In this setting, $\beta$ could be taken arbitrarily close to 1 and  the corresponding  rate would read $N^{-\frac{1}{2d+1}(\frac 12-\varepsilon)}=N^{-\frac{1}{4d+2}(1-2\varepsilon)}=N^{-\frac{1}{\bba\cdot {\boldsymbol d}+2}(1-2\varepsilon)} $. Since $\beta $ is close to $1$ in that case, the exponent in the previous convergence rate also corresponds   to the typical magnitude of the parameter
$\zeta $. For $N$ particles in dimension $2d$ this has to be compared  to the expected interaction range, which  from the multi-scale effect due to the kinetic dynamics would actually read\footnote{\textcolor{black}{From now on, for $a\in \R$, we denote be $a^-$ (resp. $a^+) $ any real number of the form $a-\varepsilon $ (resp. $a+\varepsilon $), $\varepsilon >0 $.}} $(1/{\color{black}(}d+3d{\color{black})})^-=(1/4d)^- $.
\item[$\bullet $] For $\alpha\in (1,2)$, if $\bbp_0$, $ \bbp_b'$ are sufficiently close and $\beta_0,\beta_b $ are small enough, then $\Lambda $ is again \textit{small}. Since $(\alpha-1-\Lambda)<(\textcolor{Dgreen}{\alpha}-\Lambda)/2\iff \alpha-\Lambda<2$, for $\beta_0 $ small enough, one can take $\beta $ arbitrarily close to $\alpha-1 $. If now $\bbp_0>\boldsymbol \alpha$ but close to $ \boldsymbol \alpha$ the rate writes:
$$N^{-\frac{\alpha-1}{\theta_\alpha+\alpha-1}(1-\frac 1\alpha-\varepsilon)}.$$
Eventually for $\bbp_0,\bbp_{b}' $ close and $\beta_b $ small $\theta_\alpha\simeq \mathcal A_{\boldsymbol 1,\bbp_0  }=\bba\cdot \frac{d}{\bbp_0'}\simeq d(2+\alpha)\frac{\alpha^+-1}{\alpha^+}$ and
the rate rewrites as
\begin{equation*} 
N^{-\frac{\alpha-1}{d(2+\alpha)+\alpha}(1-\frac{\alpha}{\alpha-1}\varepsilon)}=N^{-\frac{\alpha-1}{\bba\cdot {\boldsymbol d}+\alpha}(1-\frac{\alpha}{\alpha-1}\varepsilon)}\textcolor{black}{.}
\end{equation*}
\textcolor{black}{There is} then a continuity in the stability parameter for the final convergence rate. The previous ones are somehow the \textit{best} achievable rates. The more singular the parameter the worse the rate. 
\\
\hspace*{.2cm}\\
{\color{black}The above will be further illustrated in Section \ref{sec:4} in the case $b$ models a Riesz-type interaction kernel. Deriving optimal convergence rates enable to highlight the contribution of each parameter, $\mstable_\alpha$, $\beta$ and $\theta_\alpha$, (and their dependency with the singularity of $b$) and enable to draw some comparison with similar rates for (first-order) non-degenerate dynamics previously established in \cite{OlRiTo-21} and \cite{CGH-25}.}
\end{trivlist}
}
\textcolor{Dgreen}{Let us finally point out that the weak propagation of chaos result \eqref{AA200} shall be understood as a yet-to-be improved threshold, the rate being directly, as for \eqref{AA300}, derived from the moderate propagation of chaos stated in Theorem \ref{main01}. This leaves open the possibility of extended more refined semigroup techniques (e.g. \cite{CD18a}) to our present singular framework open. This will be addressed in future works.}
\er

\br[\textcolor{orange}{About the sampling of the initial condition}]

\textcolor{Dgreen}{Our results can be extended} to the case where the initial data $\xi_i$ are not \textcolor{Dgreen}{necessarily} i.i.d. variables. Specifically, \textcolor{Dgreen}{assuming only exchangeability of the initial states $(X^{N,i}_0,V^{N,i}_0)$}, the convergence \textcolor{Dgreen}{rates in Theorem \ref{S1:main01} (and Theorem \ref{main01})} still hol\textcolor{Dgreen}{d} \textcolor{Dgreen}{if}
$$
\lim_{N\to\infty} N^{c_0\zeta}\|\phi_N*(\mu^N_0-\mu_0)\|_{L^n(\Omega;\bB^{\beta_0-\beta}_{\bbp_0;\bba})}^n=0
$$
with some constants $c_0 > 0$ and $n \in \mathbb{N}$. \textcolor{Dgreen}{This assertion can be explicitly draw\textcolor{Dgreen}{n} from the proof of Theorem \ref{main01} below (see the estimates \eqref{SF3} to \eqref{SF5} in \textbf{Step 3}) from which Theorem \ref{S1:main01} is derived.} Since the representation involving $c_0$ and $n$ can be intricate, for simplicity, we \textcolor{Dgreen}{have} assume\textcolor{Dgreen}{d} independence and identical distribution of the initial data. \textcolor{Dgreen}{The} i.i.d. assumption \textcolor{Dgreen}{is made} \textcolor{black}{for simplicity}.
\er 
\br[\textcolor{Dgreen}{About the  integrability parameters $\bbp_b$ and $\bbp_0$}]\label{REM_DRIFT} 
\textcolor{Dgreen}{We would like to specify some points concerning the assumptions on the integrability parameter $\bbp_b\in [1,+\infty{\color{black}]}^2$ of the singular drift $b$ and the particular thresholds set on $\bbp_0$ in Theorem \ref{S1:main01} (and Theorem \ref{main01}). The restriction {\color{black}$\bbp_0\le \bbp_b'$ (equivalent to $\bbp_b\le \bbp_0'$) stated in {\rm \textbf{(H)}} is inherent} 
to the wellposedness of the McKean-Vlasov \eqref{MV1}, and might potentially be overcome in future works (see e.g. \cite{CdRJM-23} for the non-degenerate case).  As it will be actually seen through the propagation of chaos analysis below, the more particular condition $\bbp_0>\boldsymbol{\alpha}$ arises from the analysis of the stochastic fluctuations in the mollified empirical measure $\transOp_{t}\phi_N*\mu^N_t$ driving \eqref{S1:00} (see precisely \eqref{MART_FIELD_I}, \eqref{S1:P2} and \eqref{MMN} below). To control these fluctuations we have to rely on martingale type inequalities for 
stochastic integrals valued in Banach spaces satisfying some appropriate functional properties, notably, as in \cite{OlRiTo-21, OlRiTo-22}, the UMD property for the Brownian case $\alpha=2$, and the further M-type Banach space property for the jump case $\alpha\in(1,2)$. (Precise definitions of these spaces and properties related to our setting are presented in Appendix \ref{sec:Type}.) These functional properties are primarily to be satisfied by the subspace $\mathbb L^{\bbp_0}\cap \mathbb L^{\bbp_b'}$. Importantly, $\mathbb L^{\bbp}$ spaces have the  UMD property  for $\bbp=(p_x,p_v)$ when neither $p_x$ nor $p_v$ belonged to $\{1,\infty \} $.}
\textcolor{Dgreen}{Whenever $p_{x,b}=1$ or $p_{v,b}=1$ (or $\bbp=\bbb1$) we are led to consider $\textcolor{black}{p_{x,b'}=\infty}$
or $p_{v,b'}=\infty$ and $\mathbb L^{\bbp_b'} $ does not fulfill the aforementioned UMD property. This difficulty can be circumvented using the embedding \eqref{Sob1} between Besov spaces which ensures $b$ lies in the larger $\bB^{\tilde \beta_b}_{\tilde \bbp_b;\bba}$-space with $\bbp_b< \tilde\bbp_{b}$ ($\Rightarrow$ $\bbp'_b> \tilde\bbp'_{b}$) and $\tilde \beta_b=\beta_b-\cA_{\bbp_b,\tilde\bbp_b}$. One can so decrease the integrability exponent to a finite one, up to a slight decrease of the regularity of $b$. Since the differential index $\tilde \beta_b-\bba\cdot \tfrac d{\tilde \bbp_b}=\beta_b-\bba\cdot \tfrac d{\bbp_b}$ remains unchanged, the rates \eqref{AA200} and \eqref{AA300} in Theorem \ref{S1:main01} and the rate \eqref{S3:main01} in Theorem \ref{main01} are preserved. This approach applies to both the diffusive case $\alpha=2 $ and the pure jump case $\alpha\in (1,2) $.
}
\textcolor{Dgreen}{	
The particular case $\bbp_0=(1,1)$ features a more significant threshold. In the Brownian case, similarly to e.g. \cite{OlRiTo-21}, \cite{OlRiTo-22}, one can proceed with weighted $\mL^{\bbp}$ space with some $\bbp>\bbb1$ and the cost of introducing a weight is somehow absorbed by the Gaussian noise. Such procedure 
{\color{black}s}eemingly fails in the pure jump case, and for the case $\alpha\in(1,2)$ the condition $\bbp_0>\boldsymbol{\alpha}$ (ensuring $\mL^{\bbp_0}$ is a M-type Banach) is needed to exhibit a convergence rate.
}

The situation $\alpha\in (1,2)$ for \textit{large} integrability parameter $\bbp_b$ or integrability parameter $\bbp_0$ below $\boldsymbol{\alpha}$ is {\black a} particular open case in our paper and a natural restriction in regard to applications. Extending this present situation is at the moment beyond our proof arguments and would rather require an additional and suitable alteration (e.g. a truncation) of the L\'evy noise $L^{\alpha,i}$ in the mollified particle system \eqref{S1:00}. Importantly, observe that the constraint $p_{x,0}\wedge p_{v,0}>\alpha$ actually prevents the situations $p_{x,b}=\infty$ or $p_{v,b}=\infty$ {\black and excludes the formal consideration of fully bounded interactions} (when $\beta_b=0$) or partially bounded interactions. Namely when the kernel depends only on the velocity ($b_t(x,v)=h_1(t,v)$) or only the position component
	($b_t(x,v)=h_2(t,x)$), the integrability index $\bbp_b=(\bbp_{x,b},\bbp_{v,b})$ including so an $\infty$ component. {\color{black}While the latter will be the subject of further research, the former case can be (possibly) handled through a mild adaptation of our proof techniques - this is briefly discussed in Remark \ref{rem:NonDegenerateCase}, Section~\ref{sec:4}.}
\er
{\black\noindent
\textbf{Notation}. Throughout this paper, we use the symbol $C$ to denote a constant, whose values may vary from one line to another. The notation $:=$ is used to signify a definition. Finally, in order to lighten some computations, we will often use the shorten notation $A \lesssim B$ and $A \asymp B$ to indicate that there exists a constant $c \geq 1$ such that, respectively, $A \leq c B$ and $c^{-1} B \leq A \leq c B$.}

\section{Convergence of empirical measure via Duhamel's formula }\label{sec:ModeratedPropagation}

Throughout this section we assume {\bf (H)} \textcolor{Dgreen}{is in force} and recall \textcolor{Dgreen}{that}
\textcolor{black}{for} the time horizon $T_0>0$, 
the McKean-Vlasov SDE \eqref{MV1} admits a unique weak solution on $[0,\Tup]$.
 Moreover, by \cite[Theorem 3.11, (i)]{HRZ23}, the time marginal law $\mu_s( d  z)=u_s(z) d  z$ satisfies \textcolor{Dgreen}{the Fokker-Planck equation} \eqref{S1:FPE} and for any $\beta\geq 0$,
\begin{align}\label{S2:01}
	\sup_{s\in(0,T_0]}\left(s^{\frac{\beta-\beta_0}{\alpha}}\|u_s\|_{\bB^{\beta,1}_{\bbp_0;\bba}}+s^{\frac{\Gap+\beta}{\alpha}}\|u_s\|_{\bB^{\beta-\beta_b,1}_{\bbp_b'\textcolor{Dgreen}{;}\bba}}\right)<\infty,
\end{align}
where 
$\bbp'_b$ is the conjugate index of $\bbp_b$ \textcolor{Dgreen}{in the sense of \eqref{KineticConjugate}}.

Let $\phi_N$ be defined by \eqref{PHN}.
Consider the following  mollified approximation of the empirical measure $\mu^N_t$:
\begin{align}\label{S1:02}
    u^N_t(z):=\mu^N_t*\transOp_t{\phi_N}(z)=\frac1N\sum_{i=1}^N (\transOp_t\phi_N)(Z^{N,i}_t-z),
\end{align}
where $\transOp_t\phi_N$ corresponds to the composition of the transport operator $\transOp_t$ with $\phi_N$ namely:
$$
\transOp_t\phi_N(x,v)=N^{(2+\alpha)\modulateorder d}
\phi(N^{(1+\alpha)\modulateorder}(x-tv),N^{\modulateorder}v)\textcolor{black}{,}
$$
\textcolor{black}{
where $\phi $ is symmetric.
}
\textcolor{Dgreen}{Note that this specific choice for the mollifier involving the transport operator will precisely allow to derive \textcolor{Dgreen}{a suitable} 
	SPDE for the difference $u_t-u_t^N $ (see \eqref{DIFF_SPDE} below)}.

Let $b:\mR_+\times\mR^d\to \mR^d$\textcolor{Dgreen}{, $\beta_0$, $\bbp_0$ and $\Gap$}  be as in {\bf (H)}. For any $f\in\bB^{0,1}_{\bbp_0;\bba}\cap \bB^{-\beta_b,1}_{\bbp_b';\bba}$ and $\beta, t\geq 0$, we introduce 
\begin{equation}\label{BBT}
\|f\|_{\mS^\beta_t(b)}:=(1\wedge t)^{\frac{\beta-\beta_0}{\alpha}}\|f\|_{\bB^{0,1}_{\bbp_0;\bba}}+(1\wedge t)^{\frac{\beta+\Gap}{\alpha}}\|b_t*f\|_{\textcolor{Dgreen}{\boldsymbol\infty}
}.
\end{equation}
The main result of this section is the following convergence of \textcolor{Dgreen}{the mollified} empirical measure \textcolor{Dgreen}{\eqref{S1:02}}. 
\begin{theorem}\label{main01} 
{\black Suppose that {\bf (H)} holds, and that either $\bbp_0>\boldsymbol{\alpha}$, or $(\alpha,\bbp_0)=(2,\bbb1)$. Let $\{\xi_i\}_{i=1}^N$ be, as in \eqref{MV1}, a family of i.i.d.  
$\mR^{2d}$-valued random variables with common law $\mu_0$, satisfying}:
\begin{align}\label{SF1}
\mu_0(|\cdot|^{m}_\bba)<\infty,\ \ \forall m\in\mN.
\end{align}
Given $\mstable_\alpha$ and $\theta_\alpha$ as in \eqref{limitstable}, for any $\beta,\zeta>0$ \textcolor{Dgreen}{such that}
\begin{align}\label{BetaRange}
\beta<(\alpha-1-\Lambda)\wedge((\textcolor{Dgreen}{\alpha+\beta_0}-\Lambda)/2),\ \ \zeta<(1-\mstable_\alpha)/\theta_\alpha,
\end{align}
\textcolor{Dgreen}{and for any $\eps>0$, $m>0$,}
	there \textcolor{Dgreen}{exists} a constant $C=C(\Theta,\zeta,\beta,\eps,m)>0$ such that for all  \textcolor{Dgreen}{$N\ge 1$},
	\begin{align}\label{S3:main01}
		\sup_{t\in[0,T_0]}\|u^N_t-u_t\|_{L^m(\Omega;\mS^\beta_t(b))}\le CN^{-\beta\zeta}+CN^{\mstable_\alpha-1+\zeta\theta_\alpha+\eps}.
	\end{align} 
	\end{theorem}
The rest of the section will be dedicated to the proof of Theorem \ref{main01}. 

\subsection{Technical preliminaries on the anisotropic Besov spaces}\label{SEC_BESOV_BIS}
The following Bernstein inequality {\black for the anisotropic block operator \eqref{Ph0}} is proven in \cite{ZZ21}. 
\bl[Bernstein's inequality]\label{BI00}
For any $\bbk:=(k_1,k_2)\in{\mN^2}$ and $\bbp, \bbp_1\in[1,\infty]^2$ with $\textcolor{Dgreen}{\bbp_1}\leq \textcolor{Dgreen}{\bbp}$, there is a constant $C=C(d,\bbk,\bbp,\bbp_1)>0$ such that, for all $j\geq 0$, \textcolor{Dgreen}{$f=f(x,v)\in\cS'$}\textcolor{black}{,}
\begin{align}\label{Ber}
	\|\nabla^{k_1}_x\nabla^{k_2}_v\cR^\bba_j f\|_{\textcolor{Dgreen}{\bbp}}\le C  2^{j \bba\cdot(\bbk+\frac{d}{\textcolor{Dgreen}{\bbp_1}}-\frac{d}{\textcolor{Dgreen}{\bbp}})}\|\cR^\bba_j  f\|_{\textcolor{Dgreen}{\bbp_1}}\textcolor{Dgreen}{=2^{j \big(\textcolor{black}{(1+\alpha) k_1}+k_2+\cA_{\bbp_1,\bbp}\big)}\|\cR^\bba_j  f\|_{\bbp_1},}
\end{align}
\textcolor{Dgreen}{for $\nabla^{k_1}_{x}$ and $\nabla^{k_2}_v$ denoting respectively the $k_1^{\rm th}$ and $k_2^{\rm th}$ differential operators with respect to the variable $x$ and the variable $v$.}
\el
{\color{black}Given the above, we can derive the following bound:
\bc\label{LiftEstimate}
For any $\bbk:=(k_1,k_2)\in\mN^2$, $s\in\mR$, $q\in[1,\infty]$ and $\bbp\in[1,\infty]^2$, there is a constant $C=C(d,\bbk,\bbp)>0$ such that for all $f\in \mathbf{B}^{s+\bba\cdot \bbk,q}_{\bbp;\bba}$,
\begin{align}\label{BerCor}
	\|\nabla^{k_1}_x\nabla^{k_2}_v f\|_{\mathbf{B}^{s,q}_{\bbp;\bba}}\le C \|f\|_{\mathbf{B}^{s+\bba\cdot \bbk,q}_{\bbp;\bba}}.
\end{align}
\ec
\begin{proof}
    Recalling the definition of $\mathbf{B}^{s,q}_{\bbp;\bba}$ (Definition \ref{bs}) and using \eqref{Ber}, we have
    \begin{align*}
       \|\nabla^{k_1}_x\nabla^{k_2}_v f\|_{\mathbf{B}^{s,q}_{\bbp;\bba}}&= \left(\sum_{j\geq0}\big(2^{ js}\|\nabla^{k_1}_x\nabla^{k_2}_v\cR^\bba_{j}f\|_{\bbp}\big)^q\right)^{1/q}\\
       &\lesssim \left(\sum_{j\geq0}\big(2^{ j(s+\bba\cdot\bbk)}\|\cR^\bba_{j}f\|_{\bbp}\big)^q\right)^{1/q}=\|f\|_{\mathbf{B}^{s+\bba\cdot \bbk,q}_{\bbp;\bba}}.
    \end{align*}
\end{proof}
}

For a function $f:\mR^{2d}\to\mR$,  the first-order difference operator is defined by
\begin{align}\label{Dif1}
\delta^{(1)}_hf(z):=\delta_hf(z):=f(z+h)-f(z),\ \ z, h\in\mR^{2d}.
\end{align}
For $M\in\mN$, the $M^{\rm th}$-order difference operator  is defined recursively by
$$
\delta^{(M+1)}_hf(z)=\delta_h\circ\delta^{(M)}_hf(z).
$$
The following characterization is {\black from  \cite[Theorem 2.7]{HZZZ22} (see also \cite{ZZ21})}.
\bp\label{prop:AniBesov}
For $s>0$, $q\in[1,\infty]$ and $\bbp\in[1,\infty]^2$, an equivalent norm of $\bB^{s,q}_{\bbp;\bba}$ is given by
\begin{align}\label{CH1}
	\|f\|_{\bB^{s,q}_{\bbp;\bba}}\asymp 
	\|f\|_{\bbp}+\left(\int_{|h|_\bba\le 1}\left(\frac{\|\delta^{([s]+1)}_hf\|_{\bbp}}{|h|^s_\bba}\right)^q\frac{ d  h}{|h|^{(2+\nalpha)d}_\bba}\right)^{1/q},
\end{align}
where $[s]$ is the 
integer part of $s$. In particular, $\bC^s_{\bba}:=\mathbf{B}^{s,\infty}_{\textcolor{Dgreen}{\boldsymbol{\infty}};\bba}$ is the anisotropic 
H\"older-Zygmund space, and for $s\in(0,1)$, there is a constant $C=C(\nalpha,d,s)>0$ such that
$$
{\color{black}C^{-1}}\|f\|_{{\textcolor{Dgreen}{\boldsymbol{\infty}}}}+{\color{black}C^{-1}}\sup_{z\not= z'}|f(z)-f(z')|/|z-z'|^{s}_\bba\le \|f\|_{\bC^s_{\bba}}{\color{black}\le C}\|f\|_{{\textcolor{Dgreen}{\boldsymbol{\infty}}}}+{\color{black}C}\sup_{z\not= z'}|f(z)-f(z')|/|z-z'|^{s}_\bba.
$$
\ep
\textcolor{Dgreen}{The above characterization will be notably used, in the appendix section \ref{sec:WeightedEstimates}, for establishing  weighted versions of key estimates in this Section including the Bernstein inequality and Lemma \ref{lemB1} below.}

We recall the following lemma proved in \cite[Lemma 2.6]{HRZ23} \textcolor{Dgreen}{on embeddings and Young convolution inequalities related to $\bB^{s,q}_{\bbp;\bba}$}\comm{JF 13/04: Later in the paper, we mention \eqref{Con} as a Young inequality; a small mention of this naming here may be useful for the reader}.
\bl\label{lemB1}
\begin{enumerate}[(i)]
	\item For any $\bbp\in[1,\infty]^2$,  $s'>s$ and $q\in[1,\infty]$, it holds that
	\begin{align}\label{AB2}
		\bB^{0,1}_{\bbp;\bba}\hookrightarrow\mL^\bbp\hookrightarrow\bB^{0,\infty}_{\bbp;\bba},\ \ \bB^{s',\infty}_{\bbp;\bba}\hookrightarrow \bB^{s,1}_{\bbp;\bba}\hookrightarrow \bB^{s,q}_{\bbp;\bba}.
	\end{align}
	\item For any  $\beta,\beta_1,\beta_2\in\mR$, $q,q_1,q_2\in[1,\infty]$ and $\bbp,\bbp_1,\bbp_2\in[1,\infty]^2$ with
	$$
	\beta=\beta_1+\beta_2,\ \ 1+\tfrac{1}{\bbp}=\tfrac1{\bbp_1}+\tfrac1{\bbp_2},\ \ \tfrac{1}{q}=\tfrac1{q_1}+\tfrac1{q_2},
	$$
	it holds that, for some universal constant $C>0$,
	\begin{align}\label{Con}
		\|f*g\|_{\bB^{\beta,q}_{\bbp;\bba}}\leq C\|f\|_{\bB^{\beta_1,q_1}_{\bbp_1;\bba}}\|g\|_{\bB^{\beta_2,q_2}_{\bbp_2;\bba}}.
	\end{align}
	
	\item For $\bbb1\leq\bbp_1\leq\bbp\leq\infty$, $q\in[1,\infty]$ and $s=s_1+\bba\cdot\big(\tfrac{d}{\bbp}-\tfrac{d}{\bbp_1}\big)$, there is a $C>0$ such that
	\begin{align}\label{Sob1}
		\|f\|_{\bB^{s,q}_{\bbp;\bba}}\le C\|f\|_{\bB^{s_1,q}_{\bbp_1;\bba}}.
	\end{align}
\end{enumerate}
\el

Now we recall the basic estimates for the kinetic semi-group.
Let $p_t(z)$ be the distributional density of $Z_t:=(\int_0^t\Levy^\alpha_s d  s,\Levy^{\alpha}_t)$. 
By 
scaling, it is easy to see that
\begin{align}\label{SCL1}
	p_t(z)=t^{-(1+\frac2{\alpha})d}p_1(
	t^{-\frac{1+\alpha}{\alpha}}x, t^{-\frac{1}{\alpha}}v),\ \ t>0.
\end{align}
The kinetic semigroup of operator $\Delta^{\frac{\nalpha}{2}}_v-v\cdot\nabla_x$ is given by
\begin{align}\label{SCL11}
	P_t f(z):=\mE \textcolor{Dgreen}{\big[}f(\transOp_t z+Z_t) \textcolor{Dgreen}{\big]}=\transOp_t(p_t*f)(z)=(\transOp_tp_t*\transOp_tf)(z)=\int_{\mR^{{\color{MColor}2}d}} \transOp_t{p_t}(z-z') \transOp_t{f}(z') d  z',
\end{align}
where 
$$
\transOp_tf(z):=f(\transOp_tz)=f(x-tv,v).
$$

The following estimates are proven in \cite[Lemma 2.12]{HRZ23}.
\bl\label{SemigroupEstimate}
Let $\bbp,\bbp_1\in[1,\infty]^2$ with $\bbp_1\le \bbp$. 
For any $\beta\in\mR$ and $\gamma\geq 0$, there is a constant $C=C(\nalpha, d,\textcolor{Dgreen}{\bbp,\bbp_1,\beta,\gamma})>0$ 
such that for any $f\in \bB^{\beta}_{\bbp;\bba}$ and all  $j\ge 0$,\textcolor{Dgreen}{{ and $t>0$}},
\begin{align}
	\|\cR^\bba_jP_tf\|_{\textcolor{Dgreen}{\bbp}}
	\le C  2^{j(\cA_{\textcolor{Dgreen}{\bbp_1,\bbp}}-\beta)}
	((2^{j\alpha}t)^{-\gamma}\wedge 1)\| f\|_{\bB^{\beta}_{\textcolor{Dgreen}{\bbp_1};\bba}}.\label{AD0306}
\end{align}
\el

The following lemma provides useful estimates about the kinetic semigroup $P_t$ (see \cite[Lemma 2.16]{HRZ23})\comm{JF and S. 22/03: To not create confusion with the coefficients related to $\mu_0$, $\beta_0\rightarrow \beta$ and $\bbp_0\rightarrow \bbp$. + Harmonization of the order of the dependency in $C$ in the two lemmas}.
\bl\label{Le215}
Let $\textcolor{Dgreen}{\bbp},\bbp_1\in[1,\infty]^2$ \textcolor{Dgreen}{with $\bbp_1\le \bbp$}. 
For any $\beta,\beta_1\in\mR$, there is a  $C_0=C_0(\nalpha,d,\textcolor{Dgreen}{\beta},\beta_1,\textcolor{Dgreen}{\bbp},\bbp_1)>0$ such that for all $t>0$,
\begin{align}\label{AD0446}
	\b1_{\{\textcolor{Dgreen}{\beta}\not=\beta_1-\cA_{\bbp_1,\textcolor{Dgreen}{\bbp}}\}}\|P_tf\|_{\bB^{\textcolor{Dgreen}{\beta},1}_{\textcolor{Dgreen}{\bbp};\bba}}+\|P_tf\|_{\bB^{\textcolor{Dgreen}{\beta}}_{\textcolor{Dgreen}{\bbp};\bba}}
	{\color{black}\le C_0}(1\wedge t)^{-\frac{(\textcolor{Dgreen}{\beta}-\beta_1+\cA_{\bbp_1,\textcolor{Dgreen}{\bbp}})\vee0}\nalpha}\|f\|_{\bB^{\beta_1}_{\bbp_1;\bba}}.
\end{align} 
\el
For {\black$f\in L^1(\mR_+;\bB^{\beta}_{\bbinfty;\bba})$ (with $\beta\in\mathbb R$ arbitrary)} and $t>0$, let\comm{JF 16/03: Although we denote below the density of $\mu_t$ by $u_t(x,v)$, the notation $u_t(x,v)$ and $u(t,x,v)$ are too close - and eventually, $u(t,x,v)=\sI_t(f)(x,v)$ is local $\Rightarrow$ we do not need it.}
$$
\textcolor{Dgreen}{
\sI_t^f}(x,v):=\int^t_0 P_{t-s}f(s,x,v) d  s.
$$
{\black By \eqref{AD0446}, whenever $f\in L^1(\mR_+;\bB^{\beta}_{\bbinfty;\bba})$ for some $\beta\in\mR$, the integral $\int_0^t P_{t-s}f(s,\cdot)ds$ is well defined in $ \bB^{\beta}_{\bbinfty;\bba}$ and satisfies
$$
\left\|\int_0^tP_{t-s}f(s,\cdot)ds\right\|_{\bB^{\beta}_{\bbinfty;\bba}}\lesssim \int_0^t\|f(s,\cdot)\|_{\bB^{\beta}_{\bbinfty;\bba}}ds.
$$
}
By the definition \eqref{SCL11} of $P_t$, it is easy to see that in the distributional sense,
$$
\p_t \textcolor{Dgreen}{
\sI_t^f}=\Delta^{\frac{\nalpha}{2}}_v\textcolor{Dgreen}{
\sI_t^f}-v\cdot\nabla_x \textcolor{Dgreen}{
\sI_t^f}+f.
$$

\subsection{Proof of Theorem \ref{main01}}\label{sec:2.2}
For \textcolor{Dgreen}{notation simplicity}, we shall write
$$
    \cU^N_t(z):=u_t(z)-u^N_t(z),
$$
and, for any smooth function $f:\mR^{2d}\to \mR$ and all  $z_0=(x_0,v_0)\in\mR^d\times\mR^d$, \textcolor{Dgreen}{we shall denote respectively by $\nabla_{x}f(z_0)$ and $\nabla_{v}f(z_0)$ the gradients of $f$ w.r.t. the $x$ and $v$ variable, evaluated \textcolor{black}{at} point $z_0$.}
To prove \eqref{S3:main01}, we shall\comm{JF 20/03: We also must write explicitly the Duhamel formula for $u_t$. Updated on 27/03 - The ref. to \cite{HRZ23} is here rather evasive, but on purpose: Duhamel representation is only mentioned in remarks or proof arguments - see Remark 1.1 in \cite{HRZ23}. The link between "wellposedness of a smooth weak solution to \eqref{S1:FPE}" and its Duhamel formulation is seemingly never addressed - at least not as clearly as in \cite{CdRJM-22}.} \textcolor{Dgreen}{ exploit the Duhamel formula related to the evolution of  $\cU^N_t$. For a.e. $t\in(0,T_0]$, the density $u_t$ \textcolor{black}{satisfies}
	$$
	\partial_tu_t=\big(\triangle^{\frac{\alpha}{2}}_v-v\cdot\nabla_x\big)u_t-\text{div}_v\big((b_t*u_t)u_t\big),
	$$
	for $\text{div}_v$ denoting the divergence operator on the variable $v$. The Duhamel formula for $u_t$  (see e.g. \cite[Section 1]{HRZ23}) \textcolor{black}{then gives}:
$$
u_t(z)=P_t{\color{black}\mu_0}(z)-\int_0^t P_{t-s}(\text{div}_v{\color{black}(}(b_s*u_s)u_s){\color{black})}(z)\, d  s, \ \ t\in[0,T_0].
$$
}
We next look at {\color{black}the} evolution equation satisfied by $\cU^N_t(z)$.
For any fixed $z_0\in\mR^{2d}$, applying It\^o's formula to $t\to(\transOp_t\phi_N)(Z^{N,i}_t-z_0)$, we have
\begin{align*}
 d (\transOp_t\phi_N)(Z^{N,i}_t-z_0)
&=\big(\p_t\transOp_t\phi_N+\Delta^{\frac{\nalpha}{2}}_{v}\transOp_t\phi_N\big)(Z^{N,i}_t-z_0) d  t+V^{N,i}_t\cdot(\nabla_{x}\transOp_t\phi_N)(Z^{N,i}_t-z_0) d  t\\
&\quad+(b^N_t*\mu^N_t)(Z^{N,i}_t)\cdot(\nabla_{v}\transOp_t\phi_N)(Z^{N,i}_t-z_0) d  t+ d  M_t^{N,i}(z_0),
\end{align*}
where
\begin{equation}\label{MART_FIELD_I}
M_t^{N,i}(z_0):=\left\{
\begin{aligned}
&\textcolor{Dgreen}{\sqrt{2}}	\int_0^t(\nabla_{v}\transOp_s\phi_N)(Z^{N,i}_s-z_0) d  W^i_s, \qquad \qquad\qquad \quad &\alpha=2,\\
&\int_0^t \int_{\textcolor{Dgreen}{\mR^d\setminus\{0\}}} \delta^{(1)}_{(0,v)}(\transOp_s\phi_N)(Z^{N,i}_s-z_0){\tilde {\mathcal N}}^i( d  s, d  v),\ \quad&\alpha\in (1,2),
\end{aligned}
\right.
\end{equation}
with  $\delta^{(1)}_{(0,v)}$ being \textcolor{Dgreen}{the (forward) difference operator on the variable $v$} defined in \eqref{Dif1}, \textcolor{Dgreen}{- i.e. $\delta^{(1)}_{(0,v)}f(z)=f(z+(0,v))-f(z)$}\comm{JF 13/04: Additional detail for the sake of the reading.} -
 and $ {\tilde {\mathcal N}}^i$ standing 
for the compensated Poisson measure associated with $ L^{\alpha,i}$, \textcolor{Dgreen}{the} driving noise of the $i^{\rm th} $ particle of the system \eqref{S1:00}.

 Noting that for any smooth function $f:\mR^{2d}\to \mR$,
\begin{align*}
\p_t \Gamma_t f(z_0)=-v_0\cdot\nabla_x\Gamma_tf(z_0),\quad \forall z_0=(x_0,v_0)\in\mR^d\times\mR^d,
\end{align*}
we have
\begin{align*}
(\p_t\transOp_t\phi_N)(Z^{N,i}_t-z_0)
+V^{N,i}_t\cdot (\nabla_x\transOp_t\phi_N)(Z^{N,i}_t-z_0)&=v_0\cdot\nabla_{\textcolor{Dgreen}{x}}\transOp_t\phi_N(Z^{N,i}_t-z_0)\\
&=-v_0\cdot\nabla_{\textcolor{Dgreen}{x}}\textcolor{Dgreen}{\big(}\transOp_t\phi_N(Z^{N,i}_t-\cdot)\textcolor{Dgreen}{\big)}(z_0),
\end{align*}
\textcolor{Dgreen}{denoting by $\nabla_x(f(Z^{N,i}_t-\cdot))(z_0)$ the $x$-derivative of $z\mapsto f(Z^{N,i}_t-\cdot)$ evaluated in $z_0$, }and 
$$
\Delta^{\frac{\nalpha}{2}}_vu_t^N(z_0)=\Delta^{\frac{\alpha}{2}}_v(\transOp_t\phi_N*\mu^N_t)(z_0)=(\Delta_v^{\frac \alpha 2}\transOp_t\phi_N*\mu^N_t)(z_0).
$$
Thus\textcolor{Dgreen}{,} by the definition of $u^N_t(z)$ in \eqref{S1:02} \textcolor{Dgreen}{and for $b^N_t$ as in \eqref{MollifiedInteraction}}, we have for all $z_0=(x_0,v_0)\in\mR^{2d}$,
\begin{align}
 d  u^N_t(z_0)&=\big((\Delta^{\frac{\nalpha}{2}}_v-v_0\cdot\nabla_x)u^N_t(z_0)-\<(b^N_t*\mu^N_t)\cdot(\nabla_v\transOp_t\phi_N)(\cdot-z_0),\mu^N_t(\cdot)\>\big) d  t+ d  M^N_t(z_0)\no\\
&=\left((\Delta^{\frac{\nalpha}{2}}_v-v_0\cdot\nabla_x) u^N_t(z_0)-\div_vG^N_t(z_0)\right) d  t+ d  M^N_t(z_0),\label{S1:P2}
\end{align}
where
\begin{align}\label{MMN}
M^{N}_t(z_0):=\frac1N\sum_{i=1}^N M^{N,i}_t(z_0),
\end{align}
with $M^{N,i}_t(z_0)$ being defined in \eqref{MART_FIELD_I}. As $b^N_t*\mu^N_t=b_t*u^N_t$,  
\begin{align}\label{DEF_G}
G^N_t(z_0):=\<(b^N_t*\mu^N_t)(\cdot)(\transOp_t\phi_N)(\cdot-z_0),\mu^N_t\>=\<(b_t*u^N_t)(\cdot)(\transOp_t\phi_N)(\cdot-z_0),\mu^N_t\>.
\end{align}
Combining \eqref{S1:P2} with \eqref{S1:FPE} and changing $z_0$ to $z$, we obtain that $\cU^N_t$ solves the SPDE
\begin{equation}\label{DIFF_SPDE}
 d  \cU^N_t=\left((\Delta^{\frac{\nalpha}{2}}_v-v\cdot\nabla_x)\cU^N_t-\div_v H^N_t\right) d  t- d  M^N_t,
\end{equation}
where  
\begin{align}\label{DEF_H}
H^N_t:=(b_t*u_t)u_t-G^N_t.
\end{align}
In particular, we derive from \eqref{DIFF_SPDE} {the} Duhamel formula,
\begin{align}\label{Duh}
\begin{split}
\cU^N_t&=P_t(u_0-u^N_0)\textcolor{Dgreen}{-}\int_0^t P_{t-s}\div_v H^{N}_s d  s-\int_0^t P_{t-s} d  M^N_s=: \cI^N_t\textcolor{Dgreen}{-}\cH^{N}_t-\cM^N_t\textcolor{Dgreen}{,}
        \end{split}
\end{align}
\textcolor{Dgreen}{\textcolor{black}{where $\cM_t^{N}=\frac 1N\sum_{i=1}^N \cM_t^{N,i}$} and 
\begin{equation}\label{MART_FIELD_II}
	\cM_t^{N,i}(z_0):=\left\{
	\begin{aligned}
		&\textcolor{Dgreen}{\sqrt{2}}	\int_0^tP_{t-s}(\nabla_{v}\transOp_s\phi_N)(Z^{N,i}_s-z_0) d  W^i_s, \qquad \qquad\qquad \quad &\alpha=2,\\
		&\int_0^t \int_{\textcolor{Dgreen}{\mR^d\setminus\{0\}}} P_{t-s}\delta^{(1)}_{(0,v)}(\transOp_s\phi_N)(Z^{N,i}_s-z_0){\tilde {\mathcal N}}^i( d  s, d  v),\ \quad&\alpha\in (1,2).
	\end{aligned}
	\right.
\end{equation}
 }
Applying the norm $\|\cdot\|_{\mS^\beta_t(b)}$ \textcolor{Dgreen}{set in \eqref{BBT}} to the above, it follows that  
\begin{align}\label{PP1}
	\|\cU^N_t\|_{\mS^\beta_t(b)}\leq  \|\cI^N_t\|_{\mS^\beta_t(b)}+\|\cH^N_t\|_{\mS^\beta_t(b)}+\|\cM^N_t\|_{\mS^\beta_t(b)}.
\end{align}
\textcolor{Dgreen}{Now, w}e first look at the estimate of $\|\cI^N_t\|_{\mS^\beta_t(b)}$. 
\bl\label{098}
For any $\beta>0$, there is a  constant $C=C\textcolor{Dgreen}{(\Theta,\beta)}>0$
such that, for all \textcolor{Dgreen}{$N\ge 1$
},
\begin{equation}\label{Init-1}
	\sup_{t>0}\|\cI^N_t\|_{\mS^\beta_t(b)}{\color{black}\le} {\color{black}C}N^{-\beta\zeta}\|u_0\|_{\bB^{\beta_0}_{\bbp_0;\bba}}+{\color{black}C}
	\|\phi_N*(\mu^N_0-u_0)\|_{\bB^{\beta_0-\beta}_{\bbp_0;\bba}},
\end{equation}
where $\zeta>0$ is from the definition \textcolor{Dgreen}{\eqref{PHN}} of $\phi_N$.
\el{\ZXC
\begin{proof}
Recalling that  $\beta-\beta_0>0 $, by  \textcolor{Dgreen}{the heat kernel estimate} \eqref{AD0446} \textcolor{Dgreen}{in Lemma \ref{Le215}},
we have 	
$$
	\|\cI^{N}_t\|_{\bB^{0,1}_{\bbp_0;\bba}}\lesssim (1\wedge t)^{-\frac{\beta-\beta_0}{\alpha}}\|\phi_N*\mu^N_0- u_0\|_{\bB^{\beta_0-\beta}_{\bbp_0;\bba}}, 
	$$
	and recalling \textcolor{Dgreen}{the condition $\bbp_0\le \bbp_b'$ in ${\bf (H)}$,} the definition \textcolor{Dgreen}{\eqref{DEF_GAP} } of $\Lambda$, and \textcolor{Dgreen}{that} $\beta>0, \Lambda \textcolor{Dgreen}{>} 0$, \textcolor{Dgreen}{\eqref{AD0446} also yields}
	\begin{align*}
	\|\cI^{N}_t\|_{\bB^{-\beta_b,1}_{\bbp_b';\bba}}
	\lesssim(1\wedge t)^{-\frac{\beta+\Gap}{\alpha}}\|\phi_N*\mu^N_0- u_0\|_{\bB^{\beta_0-\beta}_{\bbp_0;\bba}}.
\end{align*}
	Hence, \textcolor{Dgreen}{as the convolution inequality} \eqref{Con} \textcolor{Dgreen}{ensures that $\|b_t*\cI^N_t\|_{\boldsymbol{\infty}}\lesssim \|b_t\|_{\bB^{\beta_b}_{\bbp_b;\bba}}\|\cI^N_t\|_{\bB^{-\beta_b,1}_{\bbp_b';\bba}}$},
\begin{align}
	\|\cI^N_t\|_{\mS^\beta_t(b)}&\lesssim  (1\wedge t)^{\frac{\beta-\beta_0}{\alpha}}
	\|\cI^{N}_t\|_{\bB^{0,1}_{\bbp_0;\bba}}+(1\wedge t)^{\frac{\beta+\Gap}{\alpha}}\|b_t\|_{ \bB^{\beta_b}_{\bbp_b\textcolor{Dgreen}{;\bba}}}	\|\cI^{N}_t\|_{\bB^{-\beta_b,1}_{\bbp'_b;\bba}}\no\\
	&\lesssim \|\phi_N*\mu^N_0- u_0\|_{\bB^{\beta_0-\beta}_{\bbp_0;\bba}}\leq \|\phi_N* u_0- u_0\|_{\bB^{\beta_0-\beta
	}_{\bbp_0;\bba}	}
	+\|\phi_N*(\mu^N_0-u_0)\|_{\bB^{\beta_0-\beta
	}_{\bbp_0;\bba}}.\label{SF9}
\end{align}
Note that \textcolor{Dgreen}{a duality argument and the Besov equivalence }\eqref{CH1} 
yields, \textcolor{Dgreen}{for all $z\in\mathbb R^d$,}   
$$
\|u_0(\cdot-z)- u_0(\cdot)\|_{\bB^{\beta_0-\beta}_{\bbp_0;\bba}}\lesssim |z|_\bba^\beta\|u_0\|_{\bB^{\beta_0}_{\bbp_0;\bba}}.
$$
\textcolor{black}{We thus have}, according to the scaling estimate \eqref{Es1} \textcolor{Dgreen}{for $\phi_N$ stated} in \textcolor{black}{Section \textcolor{Dgreen}{\ref{APP_BESOV}} of the appendix,} 
$$
	\|\phi_N* u_0- u_0\|_{\bB^{\beta_0-\beta
	}_{\bbp_0;\bba}}
	\lesssim \||\cdot|_\bba^\beta\phi_N\|_{\textcolor{Dgreen}{\bbb1}}\|u_0\|_{\bB^{\beta_0}_{\bbp_0;\bba}}=N^{-\beta\modulateorder}\||\cdot|_\bba^\beta\phi_{1}\|_{\textcolor{Dgreen}{\bbb1}}\|u_0\|_{\bB^{\beta_0}_{\bbp_0;\bba}}.
$$
\textcolor{black}{Plugging  this inequality into \eqref{SF9} we obtain \eqref{Init-1}}.
\end{proof}
}

\textcolor{black}{The estimates} for \textcolor{Dgreen}{handling the components} $\|\cH^{N}_t\|_{\mS^\beta_t(b)}$ and \textcolor{black}{the} moments of $\|\cM^N_t\|_{\mS^\beta_t(b)}$ in \eqref{PP1} are stated in the two next lemmas. 
The first one exhibits a term which goes to zero with a rate depending on the smoothness of the underlying Fokker-Planck equation and an integral contribution writing  essentially in terms of $\cU^N_t$ \textcolor{Dgreen}{and $\|u^N_t\|_{\bbp_0}$}, i.e. these terms will enter in a Gronwall or {\black absorption type} argument. The second lemma, associated with the stochastic integral, essentially  produces a convergence rate close to the one of the corresponding stable limit theorem. 
Proofs are postponed to the next Section \ref{sec:ProofMain}.

\bl\label{lem:HN} (Estimate for $\cH^{N}_t$) For any $\beta\in\big(0,{\ZXC (\alpha+\beta_0-\Lambda)\wedge 1})\big)$, there is a constant $C=C(\Theta\textcolor{Dgreen}{,\beta})>0$ such that, for all $t\in(0,T_{\textcolor{Dgreen}0}]$ and \textcolor{Dgreen}{$N\ge 1$},
\begin{align}\label{AZ11}
	\begin{split}
		\|\cH^{N}_t\|_{\mS^{\beta}_{t}(b)}
		{\color{black}\le} {\color{black}C} N^{{\ZXC -\beta}\modulateorder}+{\color{black}C}\int^t_0
		G_\beta(t,s)\Big(\|b_s*\cU^N_s\|_{\textcolor{Dgreen}{\boldsymbol{\infty}}}\|u^N_s\|_{\bbp_0} +s^{-\frac{\Gap}{\alpha}
		}\|\cU^N_s\|_{\bbp_0}\Big) d  s,
	\end{split}
\end{align}
where
\begin{align}\label{AZ12}
	G_\beta(t,s):=t^{\frac{\beta-\beta_0}{\alpha}
	}(t-s)^{-\frac1{\alpha}}+t^{\frac{\beta+\Gap}{\alpha}
	}(t-s)^{-\frac{\Gap+\beta_0+1}{\alpha}
	}.
\end{align}
\el 

\bl[Estimate for $\cM^{N}$]\label{lem:MN} 
\ZXC 
Under the assumptions of Theorem \ref{main01}, 
for any $\eps>0$, $m\in\mathbb N$ and {\color{black}$\zeta>0$} 
with the following upper bound
\begin{align}\label{Ze1}
\zeta<(1-\mstable_\alpha)/\theta_\alpha,
\end{align}
there is a constant $C=C(\textcolor{Dgreen}{\varepsilon,\text{
}}m,\beta,\zeta)>0$ such that for all $N\in\mN$,
\begin{align}\label{0207:000}
	\sup_{t\in[0,T_0]}\left\|\cM^N_t\right\|_{L^m(\Omega;\mS^{\beta}_t(b))} 
	\le C  N^{\mstable_\alpha-1+\zeta\theta_\alpha+\eps}.
\end{align}
\el
With the above \textcolor{Dgreen}{preparatory results} we are \textcolor{Dgreen}{now able} to demonstrate Theorem \ref{main01}.
\begin{proof}[Proof of Theorem \textcolor{MColor}{\ref{main01}}] \textcolor{Dgreen}{Applying}  Lemmas \ref{098} and \ref{lem:HN} \textcolor{Dgreen}{to \eqref{PP1}}, 
\label{THE_PROOF_OF_THM_MAIN}
we have
\begin{align}\label{AZ8}
\|\cU^N_t\|_{\mS^\beta_t(b)}\lesssim & \textcolor{Dgreen}{I^N_0}
+\|\cM^N_t\|_{\mS^\beta_t(b)}+\int^t_0
G_\beta(t,s)\Big(\|b_s*\cU^N_s\|_{\infty}\|u^N_s\|_{\bbp_0} +s^{-\frac{\Gap}{\alpha}
}\|\cU^N_s\|_{\bbp_0}\Big) d  s,
\end{align}
where \textcolor{Dgreen}{$\beta\in\big(0,(\alpha+\beta_0-\Gap)\wedge 1\big)$,} $G_\beta(t,s)$ is defined in \eqref{AZ12} and\comm{JF and S. 22/03: Denoting this term $\cI^N_0$ creates a strong confusion with $\cI^N_t$ defined in \eqref{Duh}.}
\begin{align}\label{AZ38}
\ZXC \textcolor{Dgreen}{I^N_0}:=N^{-\beta\zeta}+\|\phi_N*(\mu^N_0-u_0)\|_{\bB^{\beta_0-\beta}_{\bbp_0;\bba}}.
\end{align}
It now remains to handle through a suitable {\black absorption type argument }
 the terms involving $\cU^N_s$ in the right hand side of \textcolor{Dgreen}{\eqref{AZ8}}, for which we need to make the $\|\cdot\|_{\mS_s^\beta(b)} $ norm appear, {\black which is defined in \eqref{BBT}}. In particular, the technical difficulty arises from the factor $\|u^N_s\|_{\bbp_0}$, which requires to distinguish the cases $\bbp_0=(1,1)$ and $\bbp_0\neq \bbb1$. For the former, only the case $\alpha=2$ is of interest, the pure-jump case requiring the property $\alpha<p_{x,0}\wedge p_{v,0}$.
We divide the proofs into three steps.\\

\noindent
{\ZXC \bf (Step 1: Case $\alpha=2$ and $\bbp_0=\boldsymbol{1}$)\textcolor{Dgreen}{.}}  For $\beta\in\big(0,(\alpha+\beta_0-\Lambda)\wedge 1)\textcolor{Dgreen}{\big)}$,   noting that
$$
\|u^N_s\|_{\bbp_0}=\|u^N_s\|_{\bbb1}=\|\mu^N_s*\transOp_s\phi_N\|_{\bbb1}=1,
$$
 we have, by the embeddings \eqref{AB2}{\black,} 
\begin{align}
	&\int^t_0
	G_\beta(t,s)\Big(\|b_s*\cU^N_s\|_{\textcolor{Dgreen}{{\boldsymbol{\infty}}}}\|u^N_s\|_{\bbp_0} +s^{-\frac{\Gap}{\alpha}}\|\cU^N_s\|_{\bbp_0}\Big) d  s\label{PP0}\\
	&\quad\lesssim \int^t_0
	G_\beta(t,s)\Big(\|b_s*\cU^N_s\|_{\textcolor{Dgreen}{{\boldsymbol{\infty}}}}+s^{-\frac{\Gap}{\alpha}}\|\cU^N_s\|_{\bB^{0,1}_{\bbp_0;\bba}}\Big) d  s\no\\
	&\quad\lesssim  \int^t_0
	G_\beta(t,s) \Big((1\wedge s)^{-\frac{\beta+\Gap}{\alpha}}+(1\wedge s)^{\frac{\beta_0-\beta-\Gap}{\alpha}}\Big)\|\cU^N_s\|_{\mS^\beta_s(b)} d  s,\no
\end{align}
recalling \eqref{BBT} for the last inequality.

		Hence, for any \textcolor{Dgreen}{$t\in[0,T_0]$ }
		and $m\in\mN$, by \eqref{AZ8} and $\beta_0\textcolor{Dgreen}{<} 0$,\comm{S. and JF 22/03: Leaving aside the change of notation, the former expression $\|\sI^N_0\|_{L^m(\Omega;\bB^{\beta_0-\beta}_{\bbp_0;\bba})}$ was confusing at the reading, since the norm $\bB^{\beta_0-\beta}_{\bbp_0;\bba}$ is already contained in \eqref{AZ38}.}
	\begin{align*}
		\|\cU^N_t\|_{L^m(\Omega\textcolor{Dgreen}{;}\mS^\beta_t(b))}
		&\lesssim \textcolor{Dgreen}{ \|I^N_0\|_{L^m(\Omega)}}
		+\|\cM^N_t\|_{L^m(\Omega\textcolor{Dgreen}{;}\mS^\beta_t(b))}+\int^t_0
		G_\beta(t,s) s^{\frac{\beta_0-\beta-\Gap}{\alpha}} \|\cU^N_s\|_{L^m(\Omega\textcolor{Dgreen}{;}\mS^\beta_s(b))} d  s.
	\end{align*}
	Since $\beta
	<\alpha-\Gap+\beta_0$ {\black(from \eqref{BetaRange}) and $\Lambda<\alpha-1$ (from \eqref{DEF_GAP}) in {\bf (H)}}, Gronwall's inequality of Volterra type stated in Lemma \ref{0214:lem02}{\black-(iii)} {\black apply (taking $g(t)=\|I^N_0\|_{L^m(\Omega)}+\|\cM^N_t\|_{L^m(\Omega\textcolor{Dgreen}{;}\mS^\beta_t(b))}$, $h=0$ and $\gamma=-\beta_0+\beta+\Lambda$ - which lies in the interval $(0,\alpha\wedge (\alpha-1+\beta-\beta_0))$), and} we get, for any $T\in(0,T_0]$,
	 \begin{align}\label{SF12}
		\sup_{t\in(0,T]}\|\cU^N_t\|_{L^m(\Omega\textcolor{Dgreen}{;}\mS^\beta_t(b))}
		\lesssim 
		\textcolor{Dgreen}{\|I^N_0\|_{L^m(\Omega)}}
		{\ZXC
		+\sup_{t\in(0,T]}\|\cM^N_t\|_{L^m(\Omega\textcolor{Dgreen}{;}\mS^\beta_t(b))}}.
		\end{align}
{\ZXC \bf (Step 2: Case $\alpha\in(1,2]$ and $\bbp_0\not=\boldsymbol{1}$)}. 
In this situation, {\color{black}contrary to the preceding case where we used $\|\mu^N_s\|_{\bB^0_{\bbb1;\bba}}=1$, a direct estimate of $\|u^N_s\|_{\bbp_0}$ carries now a significant cost as, owing to \eqref{Con} and the estimate \eqref{Es3}, 
$$
\|u^N_s\|_{\bbp_0}\lesssim\|\Gamma_s\phi_N\|_{\bB^{0,1}_{\bbp_0;\bba}}\|\mu^N_s\|_{\bB^0_{\bbb1;\bba}}\lesssim N^{\cA_{\bbb1,\bbp_0}}.
$$ }
To avoid this cost, we again make the difference $\cU^N_s$ appear, for $s\in(0,t]$. {\ZXC Using the control 
$\|u_s\|_{\bbp_0}\textcolor{Dgreen}{\lesssim \|u_s\|_{B^{0,1}_{\bbp_0;\bba}}} \lesssim s^{\frac{\beta_0}\alpha}$ deduced from \textcolor{Dgreen}{\eqref{AB2} and }\eqref{S2:01} and recalling the definition \eqref{BBT}}, we derive
\begin{align*}
	\|b_s*\cU^N_s\|_{\textcolor{Dgreen}{{\boldsymbol{\infty}}}}\|u^N_s\|_{\bbp_0}
	&\leq\|b_s*\cU^N_s\|_{\textcolor{Dgreen}{{\boldsymbol{\infty}}}}\|\cU^N_s\|_{\bbp_0}+\|b_s*\cU^N_s\|_{\textcolor{Dgreen}{{\boldsymbol{\infty}}}}\|u_s\|_{\bbp_0}\\
	&\leq s^{\frac{\beta_0-2\beta-\Gap}{\alpha}}\|\cU^N_s\|_{\mS^\beta_s(b)}^2+s^{\frac{\beta_0-
	\beta-\Gap}{\alpha}}
	\|\cU^N_s\|_{\mS^\beta_s(b)}.
\end{align*}
Substituting this \textcolor{Dgreen}{estimate} into  \eqref{AZ8}, we get, for any $0<t\leq \Tup$,\comm{JF 15/04: There was not a single reason to introduce this formulation except to gain some space. Hopefully the referee will understand that ...} 
\begin{align*}
	\|\cU^N_t\|_{\mS^\beta_t(b)}
	&\lesssim \textcolor{Dgreen}{I^N_0}
	+\|\cM^N_t\|_{\mS^\beta_t(b)}+\sum_{i=1}^2\int_0^t G_\beta(t,s)s^{\frac{\beta_0-i\beta-\Gap}{\alpha}} \|\cU^N_s\|^i_{\mS^\beta_s(b)} d  s.
\end{align*}
Note that by \eqref{BetaRange},
$$
\beta_0-2\beta-\Gap>-\alpha,\ \ \beta_0-2\beta-\Gap>-(\alpha-1-\beta_0+\beta).
$$
\textcolor{Dgreen}{Note also that, as $\frac{\alpha}2>\alpha-1$, $\beta<\frac{\alpha}2-\Gap$ and, as $\beta_0<0$, according to ${\bf (H)}$,  $\Gap<\alpha-1-\beta_0$. These bounds ensure:
$$
\sup_{t\in(0,\Tup]}\sum_{i=1}^2\int_0^t G_\beta(t,s)s^{\frac{\beta_0-i\beta-\Gap}{\alpha}} d  s\asymp \sup_{t\in(0,\Tup]}\sum_{i=1}^2\int_0^t \big((t-s)^{-\frac 1{\alpha}}+(t-s)^{-\frac{\Gap+\beta_0+1}{\alpha}}\big)s^{\frac{\beta_0-i\beta-\Gap}{\alpha}} d  s<\infty.
$$} 
{\ZXC Thus one can use \textcolor{Dgreen}{the Gronwall inequality \eqref{0214:03} stated in} Lemma \ref{0214:lem02}-\textcolor{Dgreen}{$(iii)$} to derive, for all $s\in(0,T_0]$,}
\begin{align}\label{QuadrControlUN}
	\|\cU^N_s\|_{\mS^\beta_s(b)}
	\lesssim& \,\,\textcolor{Dgreen}{I^N_0}
	+\|\cM^N_s\|_{\mS^\beta_s(b)}+\int_0^sG_\beta(s,r)r^{\frac{\beta_0-\beta-\Gap}{{\alpha}}}\|\cM^N_r\|_{\mS^\beta_r(b)} d  r\no\\
	&+\int_0^s G_\beta(s,r)r^{\frac{\beta_0-2\beta-\Lambda}{\alpha}}\|\cU^N_r\|_{\mS^\beta_r(b)}^2 d  r.
\end{align}
Define
$$
	 F^N_{t}:=\sup_{t'\in(0,t]}\int^{t'}_0G_\beta(t',s)s^{\frac{\beta_0-2\beta-\Gap}{{\alpha}}}\|\cU^N_s\|_{\mS^\beta_s(b)}^2 d  s
$$
and
	 \begin{align}\label{GrowthTerm}
	 \cT_N&:=\textcolor{Dgreen}{I^N_0}
	 +\Bigg(\sup_{t\in(0,T_0]}\int^t_0G_\beta(t,s)
	 {s}^{\frac{\beta_0-{\black 2}\beta-\Gap}{{\alpha}}}\|\cM^N_s\|_{\mS^\beta_s(b)}^2 d  s\Bigg)^{\frac 12}.
	 \end{align}
     {\black
    We note that by Jensen's inequality, 
\begin{align*}
& \int_0^{t'}G_\beta(t',s)s^{\frac{\beta_0-2\beta-\Lambda}{\alpha}}\left(\int_0^sG_\beta(s,r)r^{\frac{\beta_0-\beta-\Lambda}{\alpha}}\|\cM^N_r\|_{\mS^\beta_r(b)}dr\right)^2ds\\
&\lesssim \int_0^{t'}G_\beta(t',s)s^{\frac{\beta_0-2\beta-\Lambda}{\alpha}}\int_0^sG_\beta(s,r)r^{\frac{\beta_0-\beta-\Lambda}{\alpha}}\|\cM^N_r\|_{\mS^\beta_r(b)}^2drds\\
&\lesssim \cT_N^2\int_0^{t'}G_\beta(t',s)s^{\frac{\beta_0-2\beta-\Lambda}{\alpha}}ds\lesssim \cT_N^2.
\end{align*}
Then} taking the square and multiplying   $G_\beta(t',s){s^{\frac{\beta_0-2\beta-\Gap}{{\alpha}}}}$  on both sides of \eqref{QuadrControlUN},
 we obtain that for some $c_0,c_1>0$,
$$
	  F^N_t 
	 	\le (c_0 \cT_N)^2+(c_1F^N_t)^2,\ \ \forall 0\le t\le \Tup.
$$
	 	 In particular,  
$$
	 \big(c_1F^N_t-(2c_1)^{-1}\big)^2\geq (2c_1)^{-2}-(c_0\cT_N)^2.
$$
	 Thus, on the event $\Omega_0:=\left\{\omega\in\Omega\,:\, c_0\cT_N<(2c_1)^{-1}\right\}$,
	 \begin{equation}\label{ControlFN_2}
	 |c_1F^N_t-(2c_1)^{-1}|\geq \sqrt{(2c_1)^{-2}-(c_0\cT_N)^2}.
	\end{equation}
	\textcolor{Dgreen}{Observe now that, using successively the control \eqref{S2:01} on $u$, the control $\sup_N\|\mu^N_t\|_{\bB^0_{\bbb1;\bba}}<\infty$, \eqref{AB2} and finally applying the estimate \eqref{Es3} on $\|\transOp_{t}\phi_N\|_{\bB^\beta_{\bbp;\bba}}$ (taking $\ell=0$, $\lambda=N^{\modulateorder}$, $\ell=0$ and successively $\bbp=\bbp_0$ and $\bbp=\bbp_b'$), we have
	}
	 \begin{align}\label{BX1}
         \|\cU^N_t\|_{\mS^\beta_t(b)}&={\black(1\wedge t)}^{\frac{\beta-\beta_0}{\alpha}}
         	\|u_t-\Gamma_t\phi_N*\mu^N_t\|_{\bB^{0,1}_{\bbp_0;\bba}}+{\black(1\wedge t)}^{\frac{\Gap+\beta}{\alpha}} \|b_t*\big(u_t-\Gamma_t\phi_N*\mu^N_t\big)\|_{\boldsymbol{\infty}}\no\\
         &\stackrel{\textcolor{Dgreen}{\eqref{Con}}}{\lesssim} t^{\frac{\beta-\beta_0}{\alpha}}
	\|u_t-\Gamma_t\phi_N*\mu^N_t\|_{\bB^{0,1}_{\bbp_0;\bba}}+t^{\frac{\Gap+\beta}{\alpha}} \|u_t-\Gamma_t\phi_N*\mu^N_t\|_{\bB^{-\beta_b,1}_{\bbp'_b;\bba}}\no\\
	&\lesssim1+ \|\Gamma_t\phi_N*\mu^N_t\|_{\bB^{0,1}_{\bbp_0;\bba}}+ \|\Gamma_t\phi_N*\mu^N_t\|_{\bB^{-\beta_b,1}_{\bbp'_b;\bba}}\no\\
	&\lesssim 1+ \|\Gamma_t\phi_N\|_{\bB^{0,1}_{\bbp_0;\bba}}+ \|\Gamma_t\phi_N\|_{\bB^{-\beta_b,1}_{\bbp'_b;\bba}}\no\\
	&\stackrel{\eqref{AB2}}{\lesssim}1+ \|\Gamma_t\phi_N\|_{\bB^{\eps}_{\bbp_0;\bba}}+ \|\Gamma_t\phi_N\|_{\bB^{\eps-\beta_b}_{\bbp'_b;\bba}}\stackrel{\eqref{Es3}}{\lesssim} N^{\theta_0\zeta},
         \end{align}
         	 where $\eps>0$ and
	 $$
	 \theta_0=[\cA_{\bbb1,\bbp_0}+(1+\alpha)\eps]\vee [\cA_{\textcolor{black}{\bbp_b,\infty}}+(1+\alpha)(\eps-\beta_b)]\textcolor{Dgreen}{.}
	 $$
\textcolor{Dgreen}{Hence}, it is easy to see that $t\mapsto F^N_t$ is continuous and $\lim_{t\downarrow 0}F^N_t=0$.
Thus, by \eqref{ControlFN_2}, 
	 $$
	 F^N_t\leq  \Big((2c_1)^{-1}-\sqrt{(2c_1)^{-2}-(c_0\cT_N)^2}\Big)/c_1\leq c_2\cT_N,\ \ \forall t\in[0,T_0],\mbox{ on $\Omega_0$,}
	 $$
for  $c_2=c_0/c_1$.
	 Plugg\textcolor{Dgreen}{ing} this into \eqref{QuadrControlUN} \textcolor{Dgreen}{and recalling \eqref{GrowthTerm}}, the above yields
	 $$
	 \|\cU^N_t\|_{\mS^\beta_t(b)}{\ZXC \ind_{\Omega_0}}\lesssim 
	 \Big(\cT_N+\|\cM^N_t\|_{\mS^\beta_t(b)}\Big){\ZXC \ind_{\Omega_0}}.
	 $$
	 Therefore, for all $m,n\ge 1$, by \eqref{BX1} and Chebyshev's inequality,  we get 
	 \begin{align}\label{0215:04}
	 \sup_{t\in(0,T_0]}\mathbb E[\|\cU^N_t\|_{\mS^\beta_t(b)}^m]&{\ZXC\leq}\sup_{t\in(0,T_0]}
	 \mathbb E[\|\cU^N_t\|_{\mS^\beta_t(b)}^m{\ZXC \ind_{\Omega^c_0}}]+\sup_{t\in(0,T_0]}\mathbb E[\|\cU^N_t\|_{\mS^\beta_t(b)}^m{\ZXC \ind_{\Omega_0}}]\no\\
	 &\lesssim N^{m  \theta_0\zeta} \mP(\Omega^c_0)+\mathbb E[ \textcolor{Dgreen}{\big(}\cT_N\textcolor{Dgreen}{\big)}^m]+\sup_{t\in (0,T_0]}\E[\|\cM^N_t\|_{\mS^\beta_t(b)}^m]\no\\
&\leq  N^{m  \theta_0\zeta}(2c_1c_0)^n\mE [\textcolor{Dgreen}{\big(}\cT_N\textcolor{Dgreen}{\big)}^n]+\mathbb E[ \textcolor{Dgreen}{\big(}\cT_N\textcolor{Dgreen}{\big)}^m]+\sup_{t\in (0,T_0]}\E[\|\cM^N_t\|_{\mS^\beta_t(b)}^m].
	 \end{align}
We now turn to the estimat\textcolor{Dgreen}{ion of} $\mE [\textcolor{Dgreen}{\big(}\cT_N\textcolor{Dgreen}{\big)}^n]$. By {\black Lemma \ref{0214:lem02} (i)}, there is a $p$ large enough so that
$$
\int^t_0G_\beta(t,r)r^{\frac{\beta_0-{\black 2}\beta-\Gap}{{\alpha}}}\|\cM^N_r\|^{2}_{\mS^\beta_r\textcolor{Dgreen}{(b)}} d  r\lesssim 
\left(\int^t_0\|\cM^N_r\|^{2p}_{\mS^\beta_r\textcolor{Dgreen}{(b)}} d  r\right)^{1/p}.
$$
Thus, if we take $n\geq 2p$, then
	 \begin{align*}
	 \mathbb E [(\cT_N)^n]
	 &\lesssim \mE\textcolor{Dgreen}{\big[}(\textcolor{Dgreen}{I^N_0}
	 )^n\textcolor{Dgreen}{\big]}+\E\Bigg[\sup_{t\in (0,T_0]}\Bigg(\int^t_0G_\beta(t,r)r^{\frac{\beta_0-\beta-\Gap}{{\alpha}}}\|\cM^N_r\|^2_{\mS^\beta_r\textcolor{Dgreen}{(b)}} d  r\Bigg)^{\frac n2}\Bigg]\\
	 &\lesssim \mE\textcolor{Dgreen}{\big[}(\textcolor{Dgreen}{I^N_0}
	 )^n\textcolor{Dgreen}{\big]}+\E\left(\int^{T_0}_0\|\cM^N_r\|^n_{\mS^\beta_r\textcolor{Dgreen}{(b)}} d  r\right)
	 \lesssim \mE\textcolor{Dgreen}{\big[}(\textcolor{Dgreen}{I^N_0})^n
	 \textcolor{Dgreen}{\big]}
	 +\sup_{t\in (0,T_0]}\E\left[\|\cM^N_t\|_{\mS^\beta_t(b)}^n\right].
	 \end{align*}
Substituting this into \eqref{0215:04}, we obtain that for any $n,m\geq 2p$,
 \begin{align}\label{SF2}
 \begin{split}
\sup_{t\in(0,T_0]}\|\cU^N_t\|_{L^m(\Omega;\mS^\beta_t(b))}
&\lesssim N^{\theta_0\zeta}\left(\|\textcolor{Dgreen}{I^N_0}\|_{L^n(\Omega)}
+\sup_{t\in (0,T_0]}\|\cM^N_t\|_{L^n(\Omega;\mS^\beta_t(b))}\right)^{n/m}\\
&\quad+\|\textcolor{Dgreen}{I^N_0}\|_{L^m(\Omega)}+\sup_{t\in (0,T_0]}\|\cM^N_t\|_{L^m(\Omega;\mS^\beta_t(b))}.
\end{split}
\end{align}
{\ZXC {\bf (Step 3)\textcolor{Dgreen}{.}} \textcolor{Dgreen}{In view of \eqref{SF2} and, as the control of $\|\cM^N_t\|_{L^m(\Omega;\mS^\beta_t(b))}$ is already set by Lemma \ref{lem:MN}, it remains to estimate $\|\textcolor{Dgreen}{I^N_0}\|_{L^n(\Omega)}$.}
Recalling the definition \eqref{AZ38} \textcolor{Dgreen}{of $I^N_0$}, we have\textcolor{Dgreen}{, for any $n\in\mN$,}
 \begin{align}\label{SF3}
\|\textcolor{Dgreen}{I^N_0}\|_{L^n(\Omega)}
\lesssim N^{-\beta\zeta}+\|\phi_N*(\mu^N_0-u_0)\|_{L^n(\Omega;\bB^{\beta_0-\beta}_{\bbp_0;\bba})}.
\end{align}
\textcolor{Dgreen}{Now let us again consider separately the cases $\bbp_0>\boldsymbol{\alpha}$ and $\bbp_0=\bbb1$, $\alpha=2$. 
For the former, let}
 $q=p_{x,0}\wedge p_{v,0}\wedge 2$. By \eqref{Con1} \textcolor{Dgreen}{in Lemma \ref{THM_CI} 
 ,} with $\bbp=\bbp_0$ and $\beta_0-\beta<0$, we have
 \begin{align}\label{SF4}
\|\phi_N*(\mu^N_0-u_0)\|_{L^n(\Omega;\bB^{\beta_0-\beta}_{\bbp_0;\bba})}\lesssim N^{\frac1 q-1+\zeta\cA_{\bbb1,\bbp_0}}\leq N^{\mstable_\alpha-1+\zeta\theta_\alpha},
\end{align}
where $\mstable_\alpha$ and $\theta_\alpha$ are defined in \eqref{limitstable}.
For \textcolor{Dgreen}{the case} $(\alpha,\bbp_0)=(2,\bbb1)$, by \eqref{Con1} with $\bbp=\bb2$ and $\beta_0-\beta<0$ and by \eqref{SF1}, we have
 \begin{align}\label{SF5}
\|\phi_N*(\mu^N_0-u_0)\|_{L^n(\Omega;\bB^{\beta_0-\beta}_{\bbp_0;\bba})}\lesssim  N^{-\frac1 2+\zeta\cA_{\bbb1,\bb2}}\leq N^{\mstable_2-1+\zeta\theta_2}.
\end{align}
\textcolor{Dgreen}{Therefore, c}ombining \eqref{SF12}, \eqref{SF2}-\eqref{SF5} \textcolor{Dgreen}{and applying Lemma \ref{lem:MN}}, we obtain that for any $n,m\geq 2p$,
\begin{align*}
\sup_{t\in(0,T_0]}\|\cU^N_t\|_{L^m(\Omega;\mS^\beta_t(b))}
&\lesssim N^{\theta_0\zeta}(N^{-\beta\zeta}+N^{\mstable_\alpha-1+\zeta\theta_\alpha{\black+\eps}})^{n/m}+N^{-\beta\zeta}+N^{\mstable_\alpha-1+\zeta\theta_\alpha{\black+\eps}}.
\end{align*}
\textcolor{Dgreen}{As $n$ can be chosen arbitrarily, given $N$ fixed, we can} choose $n$ large enough, \textcolor{Dgreen}{so that the $(\textcolor{black}{N^{-\beta\zeta}}+N^{\mstable_\alpha-1+\zeta\theta_\alpha})^{n/m-1}$ dominates $N^{\theta_0\modulateorder}$. This enables} to derive the desired estimate.
}
\end{proof}

 \subsection{Proof of Theorem \ref{S1:main01}}\label{sec:ProofMainThm}

By \eqref{Con} and \textcolor{Dgreen}{the essential control }\eqref{S2:01} \textcolor{Dgreen}{on $u_s$},  for any $\beta\geq 0$, we have 
\begin{align}\label{1009:00}
\|b_s*\mu_s\|_{\bB^{\beta}_{\textcolor{Dgreen}{\boldsymbol{\infty}};\bba}}=\|b_s*u_s\|_{\bB^{\beta}_{\textcolor{Dgreen}{\boldsymbol{\infty}};\bba}}\textcolor{Dgreen}{\lesssim \|b_s\|_{\bB^{\beta_b}_{\bbp_b;\bba}}\|u_s\|_{\bB^{\beta-\beta_b}_{\bbp_b';\bba}}}\lesssim  s^{-\frac{\Gap+\beta}{\alpha}}
\|b\|_{\textcolor{Dgreen}{L}^\infty_{T_0}\bB^{\beta_b}_{\bbp_b;\bba}},
\end{align}
\textcolor{Dgreen}{and, according to \eqref{AB2}, 
\begin{align}\label{1009:00-bis}
	\|b_s*\mu_s\|_{\mL^{\textcolor{Dgreen}{\boldsymbol\infty}}}\lesssim\|b_s*u_s\|_{\bB^{0,1}_{\boldsymbol{\infty};\bba}}\lesssim  s^{-\frac{\Gap}{\alpha}}
	\|b\|_{\textcolor{Dgreen}{L}^\infty_{T_0}\bB^{\beta_b}_{\bbp_b;\bba}}.
\end{align}
}
In particular, \textcolor{Dgreen}{\eqref{1009:00} yields: }for any $\beta\in[0,\alpha-1-\Gap)$, there exists $q>\frac{\alpha}{\alpha-1}$  
such that \textcolor{Dgreen}{the drift of \eqref{MV1}}\comm{JF 23/03: Since the symbol $B$ will be frequently used in the proof, we need to emphasize the notation $B$ more heavily. Furthermore keeping the notation $b_s*u_s$ instead $b_s*\mu_s$ may be more favourable for the reading.}
$$
B_s(x,v):=b_s*\mu_s(x,v)\textcolor{Dgreen}{=b_s*u_s(x,v)}
$$
\textcolor{Dgreen}{is in ${L}^{q}_{T_0}\bB^\beta_{\boldsymbol{\infty};\bba}$. We split the proof of Theorem \ref{S1:main01} by establishing successively the estimates \eqref{AA200} and \eqref{AA300}.}

\begin{proof}[Proof of \eqref{AA200}]
Fix $T\in(0,\Tup\textcolor{Dgreen}{]}$ and $\varphi\in C^\infty_b(\mR^{2d})$, and set $B^{T}_t:=B_{T-t}(=b_{T-t}*\mu_{T-t})$. By \cite[Theorem 4.2-(i)]{HRZ23}, there is a unique smooth\comm{JF 23/03: The application of the theorem requires that we recall that $\mathcal C^\beta_{\bba}\simeq \mB^{\beta}_{\infty;\bba}$ from Proposition \ref{prop:AniBesov}.} solution to the following PDE:
$$
\p_tw=(\Delta^{\frac{\nalpha}{2}}_v
-v\cdot\nabla_x)w+B^{T}\cdot\nabla_v w \ \text{on} \ [0,T]\times\R^{2d}, \qquad w(0)=\varphi.
$$
Its Duhamel formulation is further given by
\begin{align*}
w(t)= P_{t}\varphi+\int_0^t  P_{t-s}(B^{T}_s\cdot\nabla_v w(s)) d  s, \qquad t\in[0,T].
\end{align*}
Noting that by \eqref{AD0446},
$$
\Vert\nabla_v P_tf\Vert_{\mL^{\textcolor{Dgreen}{\boldsymbol\infty}}}\lesssim\Vert\nabla_v P_tf\Vert_{\bB^{0,1}_{\infty;\bba}}
\lesssim t^{-\frac 1{\alpha}}\|f\|_{\bB^{0}_{\infty;\bba}}\lesssim t^{-\frac 1{\alpha}}\|f\|_{\mL^{\textcolor{Dgreen}{\boldsymbol{\infty}}}},
$$ 
by \textcolor{Dgreen}{\eqref{1009:00-bis}}, we have
\begin{align*}
\|\nabla_v w(t)\|_{\mL^{\textcolor{Dgreen}{\boldsymbol\infty}}}&\lesssim t^{-\frac 1{\alpha}} 
\|\varphi\|_{\mL^{\textcolor{Dgreen}{\boldsymbol\infty}}}+\int_0^t (t-s)^{-\frac 1{\alpha}}
\|B^{T}_s\|_{\mL^{\textcolor{Dgreen}{\boldsymbol\infty}}}\|\nabla_v w(s)\|_{\mL^{\textcolor{Dgreen}{\boldsymbol\infty}}} d  s\\
&\lesssim t^{-\frac 1{\alpha}}
\|\varphi\|_{\mL^{\textcolor{Dgreen}{\boldsymbol\infty}}}+\int_0^t (t-s)^{-\frac1\alpha}s^{-\frac{\Gap
}{\alpha}}
\|\nabla_v w(s)\|_{\mL^{\textcolor{Dgreen}{\boldsymbol\infty}}} d  s.
\end{align*}
Since $\Gap<1$ 
, by \textcolor{Dgreen}{the classical} Gronwall inequality of Volterra type
{\black,} we have
\begin{align}\label{09:3000}
\|\nabla_v w(t)\|_{\mL^{\textcolor{Dgreen}{\boldsymbol\infty}}}\lesssim t^{-\frac1\alpha}
\|\varphi\|_{\mL^{\boldsymbol\infty}} {\black +\int_0^t(t-s)^{-\frac1\alpha}s^{-\frac{\Lambda+1}{\alpha}}\|\varphi\|_{\mL^{\boldsymbol\infty}} d  s \lesssim t^{-\frac1\alpha}
\|\varphi\|_{\mL^{\boldsymbol\infty}}},
\end{align}
{\black as, by {\bf (H)}, $1+\Lambda<\alpha$.}

By applying \textcolor{Dgreen}{the} generalized version of It\^o's formula to $(t,z)\mapsto u(T-t,z)$ stated in \cite[Lemma 4.3]{HRZ23}, we have
\begin{align*}
\mE\textcolor{Dgreen}{\big[}\varphi(Z_{T})\textcolor{Dgreen}{\big]}=\mE \textcolor{Dgreen}{\big[}w(0,Z_{T})\textcolor{Dgreen}{\big]}=\mE \textcolor{Dgreen}{\big[}w(T,Z_0)\textcolor{Dgreen}{\big]},
\end{align*}
 and
 \begin{align*}
\mE \textcolor{Dgreen}{\big[}\varphi(Z^{N,1}_{T})\textcolor{Dgreen}{\big]}=\mE\textcolor{Dgreen}{\big[} w(0,Z^{N,1}_{T})\textcolor{Dgreen}{\big]}=\mE\textcolor{Dgreen}{\big[} w(T,\xi_1)\textcolor{Dgreen}{\big]}+\mE\textcolor{Dgreen}{\Big[}\int_0^{T} \left(b_s*u^N_s-B_s\right)\cdot \nabla_v w(T-s,Z^{N,1}_s) d  s\textcolor{Dgreen}{\Big]}.
\end{align*}
Thus, by \eqref{Con} and \eqref{09:3000},\comm{S. and JF 22/03: The preceding bound was written in $\mE\|u^N_s-u_s\|_{\bB^{-\beta_b{\color{MColor},1}}_{\bbp'_b;\bba}}$ and so was obsolete since we have the estimate on $\|b_s*u^N_s-b_s*u_s\|_{\mL^\infty}$} 
 \begin{align*}
|\mE\textcolor{Dgreen}{\big[} \varphi(Z_{T})\textcolor{Dgreen}{\big]}-\mE\textcolor{Dgreen}{\big[} \varphi(Z^{N,1}_{T})\textcolor{Dgreen}{\big]}|&\le \mE\textcolor{Dgreen}{\Big[}\int_0^{T} \|b_s*u^N_s-b_s*u_s\|_{\mL^{\textcolor{Dgreen}{\boldsymbol{\infty}}}}\|\nabla_v w(T-s)\|_{\mL^{\textcolor{Dgreen}{\boldsymbol{\infty}}}} d  s\textcolor{Dgreen}{\Big]}\\
&\lesssim \|\varphi\|_{\mL^{\textcolor{Dgreen}{\boldsymbol{\infty}}}}\int_0^{T}
(T-s)^{-\frac1\alpha}\textcolor{Dgreen}{\|b_s*u^N_s-b_s*u_s\|_{\mL^{\boldsymbol{\infty}}}} d  s.
\end{align*}
{\black Since the right-hand side depends only on $\|\varphi\|_{\mL^\infty}$, we can extend the estimate to the total variation:
$$
\Vert\mathcal L(Z_t)-\mathcal L(Z^{N,1}_t)\Vert_{\text{var}}=\sup_{\varphi\,\text{bounded Borel func.}}|\mE\big[ \varphi(Z_{T})\big]-\mE\big[ \varphi(Z^{N,1}_{T})\big]|,
$$
 approximating arbitrary $\varphi$ with smooth functions.}

Estimate \eqref{AA200} now follows from \eqref{S3:main01} \textcolor{Dgreen}{which gives\comm{From JF to S. 01/04: The modification in the orange part is just $\infty\rightarrow\boldsymbol{\infty}$ (nonbold -> bold).} 
 \begin{align}
	& \mE\textcolor{Dgreen}{\Big[}\int_0^T(T-s)^{-\frac1\alpha} \textcolor{orange}{\|b_s*u^N_s-b_s*u_s\|_{\mL^{\textcolor{Dgreen}{\boldsymbol{\infty}}}}} d  s\textcolor{Dgreen}{\Big]}\notag\\
	&\lesssim \int_0^{T}(T-s)^{-\frac1\alpha}(1\wedge s)^{-\frac{\beta+\Gap}{\alpha}}\mE\textcolor{Dgreen}{\big[}\|u^N_s-u_s\|_{\mS^\beta_s(b)}\textcolor{Dgreen}{\big]} d  s\lesssim N^{-\beta\modulateorder}+N^{\mstable_\alpha-1+\zeta\theta_\alpha+\eps},\label{TO_EXPLAIN_FIRST_CTR_BETA}
\end{align}
the limit $\beta<\alpha-1-\Gap \Leftrightarrow \frac{\beta+\Gap}{\alpha}<1-\textcolor{orange}{\frac 1\alpha}$ ensuring the finiteness of the integral $\int_0^{T}(T-s)^{-\frac1\alpha}(1\wedge s)^{-\frac{\beta+\Gap}{\alpha}}\, d  s$. 
}
\end{proof}
\begin{proof}[Proof of \eqref{AA300}]
	\textcolor{Dgreen}{Since \textcolor{orange}{$b\in \textcolor{Dgreen}{L}^\infty_{\Tup}\bB^{1+\frac\alpha2,\frac \alpha2-1+\beta_b}_{\bbp_b;x,\bba}$}, according to Lemma 4.7-$(ii)$ in \cite{HRZ23}, for any $\beta \in\big(1-\frac{\alpha}2,\bar \beta_\alpha)$, $B$ belongs to the space $ L^q_{T_0}(\bC^\beta_\bba\cap\bC^{\frac{\alpha+\beta}{\alpha+1}}_{x})$, for $\bC^{\frac{\alpha+\beta}{\alpha+1}}_x$ denoting the (isotropic) H\"older-Zygmund space on the $x$-variable. 
	}


 Then, \textcolor{Dgreen}{ for any fixed $\lambda>0$, by \cite[Theorem 4.2-(i)]{HRZ23}} {\black (inverting the forward form (4.4) therein with the time reversal $t\to T-t$)}
 , there is a unique solution $u$ to the following (Zvonkin type) backward PDE for $T\in(0,\Tup)$:
\begin{align}\label{S6:PDE}
\p_t w+(\Delta_v^{\frac \alpha 2}-v\cdot\nabla_x-\lambda )w +B\cdot\nabla_v w=B, \quad w(T)=0,
\end{align}
such that, setting $\nabla u=(\nabla_xw,\nabla_v w)$, by \cite[Theorem 4.2-(iii)]{HRZ23}, for $\lambda$ large enough
: 
\begin{align}\label{S60}
		\|\nabla w\|_{\mL_{T}^{\infty}}:=\|\nabla w\|_{L^\infty((0,T);\mL^{\textcolor{Dgreen}{\boldsymbol\infty}})}\le \tfrac 12,\qquad  \textcolor{Dgreen}{g}(s):=\textcolor{Dgreen}{\|\textcolor{black}{\nabla w(s)} \|_{C^{\delta}_v}}\in L^q([0,T_0]),
\end{align}
\textcolor{Dgreen}{where, if $\alpha=2$,  \textcolor{black}{$\delta=1$ and}  $\|\textcolor{black}{\nabla w(s)} \|_{C^{1}_v}\textcolor{black}{:=}\|\nabla_x\nabla_v w(s) \|_{\mL^{\boldsymbol{\infty}}}+\|\nabla^2_v w(s) \|_{\mL^{\boldsymbol{\infty}}}$, and, if $\alpha\in(1,2)$,  $\delta\ge \frac{\alpha}2+\varepsilon_0$ for some $ \varepsilon_0>0$,}	 \textcolor{black}{and}
\begin{align*}
\|\textcolor{Dgreen}{\textcolor{black}{\nabla w(s)}}\|_{C_v^{\delta}}:=\|\textcolor{black}{\nabla w(s)}\|_{\mL^{\textcolor{Dgreen}{\boldsymbol\infty}}}+\sup_{v\in\mR^d\backslash\{0\}}\frac{\|\delta^{(1)}_{(0,v)}\textcolor{black}{\nabla w(s)}\|_{\mL^{\textcolor{Dgreen}{\boldsymbol{\infty}}}}}{|v|^\delta},
\end{align*}
\textcolor{Dgreen}{$\delta^{(1)}_{(0,v)}$ being the difference operator  defined in \eqref{Dif1}.}
In particular, for each $t\in[0,T]$, $z\mapsto\Phi_t(z):=z+(0,w(t,z))$
 forms a $C^1$-diffeomorphism on $\mathbb R^{2d}$.
By It\^o's formula, we have, for any $t\in[0,T]$,
\begin{align*}
\Phi_t(Z_t^{N,1})=&\Phi_0(\xi_1)+\int_0^t(V^{N,1}_s,\lambda w 
(s,Z^{N,1}_s)) d  s+M_t^{N,1}
+{\ZXC \int_0^t[b_s*(u_s^N-\textcolor{Dgreen}{u}_s)\cdot \nabla_v\Phi_s](Z^{N,1}_s)  d  s},
\end{align*}
where
\begin{equation*}
M_t^{N,1}=\left\{
\begin{aligned}
&\sqrt{2}\int_0^t \nabla_v\Phi_s(Z^{N,1}_s) d  W_s^1,\ \qquad{\rm if}\ \alpha=2,\\
&\int_0^t\!\!\!  \int_{\R^d }\delta^{(1)}_{(0,v)}\Phi_s(Z^{N,1}_s) \tilde {\mathcal N}^1( d  s, d  v),\ {\rm if }\ \alpha\in (1,2).
\end{aligned}
\right.
\end{equation*}
Similarly, for $Z^1_t=(X^1_t,V^1_t)$, \textcolor{Dgreen}{the solution to \eqref{MV1} driven by $L^{\alpha,1}$ and starting at the initial $Z^1_0=\xi_1$, we have}
\begin{align*}
\Phi_t(Z^1_t)=&\Phi_0(\xi_1)+\int_0^t(V^1_s,\lambda w(s,Z^1_s)) d  s+M_t,
\end{align*}
where 
\begin{equation*}
M_t=\left\{
\begin{aligned} 
&\sqrt{2}\int_0^t \nabla_v\Phi_s(Z^1_s) d  W_s^1,\ \qquad {\rm if }\ \alpha=2,\\
&\int_0^t \!\!\! \int_{\R^d }\delta^{(1)}_{(0,v)}\Phi_s(Z^1_s) \tilde {\mathcal N}^1( d  s, d  v),\ {\rm if }\ \alpha\in (1,2).
\end{aligned}
\right.
\end{equation*}
Thus, \textcolor{Dgreen}{according to \eqref{S60},} 
 \begin{align*}
|\Phi_t(Z^{N,1}_t)-\Phi_t(Z^{\textcolor{Dgreen}1}_t)|\lesssim&(1+\lambda\|\nabla w 
\|_{\mL^\infty_{T}})\int_0^t |Z^{N,1}_s-Z^1_s| d  s+|M_t^{N,1}-M_t|\\
&+\|\nabla_v\Phi\|_{\mL^\infty_{T}}\int_0^t |b_s*(u^N_s-u_s)(Z^{N,1}_s)| d  s,
\end{align*}
where
\begin{equation*}
M_t^{N,1}-M_t=\left\{
\begin{aligned}
&\textcolor{Dgreen}{\sqrt{2}}\int_0^t\left(\nabla_v w(s,Z^{N,1}_s)-\nabla_v w(s,Z^1_s)\right) d  W^1_s,\ \qquad {\rm if}\ \alpha=2,\\
&\int_0^t\!\!\! \int_{\R^d}\Big\{\delta^{(1)}_{(0,v)}\Phi_s(Z^{N,1}_s)-\delta^{(1)}_{(0,v)}\Phi_s(Z^1_s)\Big\} \tilde {\mathcal N}^1( d  s, d  v),\ {\rm if }\ \alpha\in (1,2).
\end{aligned}
\right.
\end{equation*}
Observe in this latter case that:
\begin{align*}
	&|\delta^{(1)}_{(0,v)}\Phi_s(Z^{N,1}_s)-\delta^{(1)}_{(0,v)}\Phi_s(Z^1_s)|\lesssim |Z_s^{N,1}-Z^{\textcolor{Dgreen}1}_s| \Big(\|\nabla w(s,\cdot)\|_{\mL^{\textcolor{Dgreen}{\boldsymbol\infty}}}\ind_{\{|v|> 1\}}+\|\nabla w(s,\cdot)\|_{C_v^\delta}|v|^\delta\ind_{\{|v|\le 1\}}\Big).
\end{align*}
Combining BDG's inequality \textcolor{Dgreen}{and} \eqref{S60}, we have 
\begin{align*}
\mathbb E\textcolor{Dgreen}{\big[}|M_t^{N,1}-M_t|^{2}\textcolor{Dgreen}{\big]}&\lesssim\begin{cases}{\displaystyle \textcolor{Dgreen}{\sqrt{2}}\int_0^t\mE\textcolor{Dgreen}{\big[}|\nabla_v w(s,Z^{N,1}_s)-\nabla_v w(s,Z^{\textcolor{Dgreen}1}_s)|^2\textcolor{Dgreen}{\big]} d  s},\ \qquad {\rm if}\ \alpha=2,\\
		{\displaystyle\int_0^t \!\!\! \int_{\R^d}\textcolor{Dgreen}{\mathbb E\big[}\big|\delta^{(1)}_{(0,v)}\Phi_s(Z^{N,1}_s)-\delta^{(1)}_{(0,v)}\Phi_s(Z^1_s)\big|^2 \textcolor{Dgreen}{\big]}\nu( d  v) d  s},\ {\rm if }\ \alpha\in (1,2).
	\end{cases}\\
	&\lesssim \int_0^t  |g(s)|^2\mE|Z^{N,1}_s-Z^{\textcolor{Dgreen}1}_s|^2 d  s,
\end{align*}
observing in particular that, in the case $\alpha\in(1,2)$, since $2\delta>\alpha$,
$$
\int |v|^{2\delta}\ind_{\{|v|\le 1\}}\nu( d  v)<\infty.
$$ 
 We note that $\textcolor{Dgreen}{|}g\textcolor{Dgreen}{|}^2\in L^{q/2}([0,T_0])$ with $q>\alpha/(\alpha-1)\textcolor{black}{\ge}2$.
Therefore, {\black since $\{\Phi_s\}_{s\in[0,T]}$ is a family of $C^1$-diffeomorphisms with uniformly bounded Lipschitz norms (including those of the inverse $\{\Phi^{-1}_s\}_{s\in[0,T]}$),}
 \begin{align*}
\mE\left[\sup_{s\in[0,t]}|Z^{N,1}_s-Z^{\textcolor{Dgreen}1}_s|^2\right]&\lesssim\mE\left[\sup_{s\in[0,t]}|\Phi_s(Z^{N,1}_s)-\Phi_s(Z^{\textcolor{Dgreen}1}_s)|^2\right]\\
&\lesssim \int_0^t |g(s)|^2\mE|Z^{N,1}_s-Z^{\textcolor{Dgreen}1}_s|^2 d  s+\mE\textcolor{orange}{\left[\left(\int_0^{T} \|b_s*(u^N_s-u_s)\|_{\mL^{\textcolor{Dgreen}{\boldsymbol\infty}}} d  s\right)^2\right]},
\end{align*}
which implies, by Gronwall's inequality, that\comm{S. and JF 23/03: The preceding upper-bound was $\left(\int_0^{T} \mE\|u^N_s-u_s\|_{\bB^{-\beta_b,1}_{\bbp'_{b};\bba}} d  s\right)^2$. JF: 24/ 03: Moreover the square has been too lightly handled in the computations: The square imposes that $\beta<\frac \alpha2-\Gap$ and this restriction has not been considered in the main theorem ! Update of S. and JF 26/03: Everything is fine as $\alpha\le 2\Leftrightarrow \alpha -1\le \alpha/2$ so the bound $\beta<\alpha-1-\Lambda$ implies $\beta<\alpha/2-\Lambda$. No need for an additional restriction on $\beta$.} 
\begin{align*}
\mE\left[\sup_{s\in[0,T]}|Z^{N,1}_s-Z^{\textcolor{Dgreen}1}_s|^ 2\right]&\lesssim\mE\textcolor{Dgreen}{\Bigg[}\int_0^{T} \|b_s*(u^N_s-u_s)\|_{\mL^\infty}^{\textcolor{orange}{2}} d  s\textcolor{Dgreen}{\Bigg]}\\
&\textcolor{orange}{\lesssim \left(\int_0^{T}(1\wedge s)^{-2(\frac{\beta+\Gap}{\alpha})}\mE\textcolor{Dgreen}{\Big[}\|u^N_s-u_s\|_{{\mS^\beta_s(b)}}^2\textcolor{Dgreen}{\Big]} d  s\right).}
\end{align*}
\textcolor{orange}{Observe that $\beta<\alpha-1-\Gap $ implies $\textcolor{orange}{\beta}<\frac\alpha2-\Gap$.  Recall indeed that since $\alpha\in (1,2] $ then $\alpha-1\le \frac \alpha 2 $. Hence,
$\int_0^{T}(1\wedge s)^{-2(\frac{\beta+\Gap}{\alpha})}  d  s$ is finite.} Estimate \eqref{AA300} \textcolor{Dgreen}{again} follows by \eqref{S3:main01}.
\end{proof}

\section{Proof of the main technical Lemmas} \label{sec:ProofMain}
\subsection{Proof of Lemma \ref{lem:HN}}
First of all, applying \eqref{AD0446} (for $\textcolor{Dgreen}{\beta}=0$, $\beta_1=\textcolor{Dgreen}{-1}$, $\bbp_1=\textcolor{Dgreen}{\bbp}\textcolor{Dgreen}{=\bbp_0}$), we have
\begin{align}\label{B0}
\|\cH^{N}_t\|_{\bB^{0,1}_{\bbp_0;\bba}}
  & {\ZXC\lesssim \int^t_0\|P_{t-s}\div_v H^{N}_s\|_{\bB^{0,1}_{\bbp_0;\bba}} d  s\lesssim \int_0^t (t-s)^{-\frac1{\alpha}}\|\div_vH^N_s\|_{\bB^{-1}_{\bbp_0;\bba}} d  s}\no\\
   &{\ZXC\lesssim \int_0^t (t-s)^{-\frac1{\alpha}}\|H^N_s\|_{\bB^{0}_{\bbp_0;\bba}} d  s\lesssim \int_0^t (t-s)^{-\frac1{\alpha}}\|H^N_s\|_{\bbp_0} d  s,}
\end{align}
using as well \eqref{AB2} for the \textcolor{Dgreen}{two} last inequalit\textcolor{Dgreen}{ies}.
In the same way, \textcolor{Dgreen}{using again} \eqref{AD0446} (for $\beta=-\beta_b$, $\beta_1=\textcolor{Dgreen}{-1}$, $\textcolor{Dgreen}{\bbp}=\bbp_{b}', \bbp_1=\bbp_0$), recalling, from {\bf (H)} that, since $\frac 1{\bbp_0}+\frac 1{\bbp_b} \ge 1$, $\bbp_0\le \bbp_{b}' $ and, from \eqref{DEF_GAP}, that 
$\ZXC\Gap+\beta_0+1=\cA_{\bbp_0,\bbp'_b}-\beta_b+1>0$, \textcolor{Dgreen}{we get}
\begin{align}\label{B1}
\|\cH^{N}_t\|_{\bB^{-\beta_b,1}_{\bbp_b';\bba}}\lesssim {\ZXC\int^t_0\|P_{t-s} \div_vH^{N}_s\|_{\bB^{-\beta_b,1}_{\bbp_b';\bba}} d  s}
\lesssim \int_0^t (t-s)^{-\frac{\Gap+\beta_0+1}{\alpha}}\|H^N_s\|_{\bbp_0} d  s.
\end{align}
In order to  estimate  $\|H^N_s\|_{\bbp_0}$,  by adding and subtracting the elements $(b_s*u_s)u^N_s$ and $\<(b_s*u_s)(\transOp_s\phi_N)(\cdot-z),\mu^N_s\>$, we obtain from \eqref{DEF_H} \textcolor{Dgreen}{and} \eqref{DEF_G} the following decomposition for $H^N_s(z)$\comm{\textcolor{Dgreen}{S. and JF 22/03: The modifications are here made to make clearer the decomposition of $H^N$.}}:
\begin{align}\label{DecomposeHN}
H^N_s(z)&=\<(b_s*\cU^N_s)(\transOp_s\phi_N)(\cdot-z),\mu^N_s\>+((b_s*u_s)\cU^N_s)(z)\notag\\
&\quad-\textcolor{Dgreen}{\Big(}\<(b_s*u_s)(\transOp_s\phi_N)(\cdot-z),\mu^N_s\>\textcolor{Dgreen}{-}((b_s*u_s)u^N_s)(z)\textcolor{Dgreen}{\Big)}\notag\\
&=:H^{1,N}_s(z)+H^{2,N}_s(z)\textcolor{Dgreen}{-}H^{3,N}_s(z).
\end{align}
For $H^{1,N}_s(z)$, noting that for any $z\in\mR^{2d}$,\comm{\textcolor{Dgreen}{S. and JF 22/03: Same reason as invoked in the preceding footnote.}}
$$
|H^{1,N}_s(z)|\textcolor{Dgreen}{=\Big|\<(b_s*\cU^N_s)(\transOp_s\phi_N)(\cdot-z),\mu^N_s\>\Big|}\leq\|b_s*\cU^N_s\|_{\textcolor{Dgreen}{\boldsymbol\infty}}\<(\transOp_s\phi_N)(\cdot-z),\mu^N_s\>=\|b_s*\cU^N_s\|_{\textcolor{Dgreen}{\boldsymbol\infty}} u^N_s(z),
$$
we have
$$
\|H^{1,N}_s\|_{\bbp_0}\leq\|b_s*\cU^N_s\|_{\textcolor{Dgreen}{\boldsymbol\infty}}\|u^N_s\|_{\bbp_0}.
$$
For $H^{2,N}_s(z)$, we directly have
$$
\|H^{2,N}_s\|_{\bbp_0}\leq\|b_s*u_s\|_{\textcolor{Dgreen}{\boldsymbol\infty}}\|\cU^N_s\|_{\bbp_0}
\stackrel{\eqref{Con}}{\lesssim} \|b_s\|_{\bB^{\beta_b}_{\bbp_b;\bba}}\|u_s\|_{\bB^{-\beta_b,1}_{\bbp_b';\bba}}\|\cU^N_s\|_{\bbp_0}
\stackrel{\eqref{S2:01}}{\lesssim}s^{-\frac\Gap{{\alpha}}}
\|\cU^N_s\|_{\bbp_0}.
$$
For $H^{3,N}_s(z)$,  by definition, we can write
\begin{align}\label{B3}
H^{3,N}_s(z)=\textcolor{Dgreen}{-}\int_{\mR^{2d}}(b_s*u_s(z)-b_s*u_s(\bar{z}))(\transOp_s\phi_N)(z-\bar{z})\mu^N_s( d  \bar{z}).
\end{align}
Note that, for any $\kappa\in(0,1)$, \textcolor{Dgreen}{and, for $\bC^{\kappa}_\bba=\bB^{\kappa,\infty}_{\boldsymbol{\infty};\bba}$ as in Proposition \ref{prop:AniBesov},}
\begin{align*}
|b_s*u_s(z)-b_s*u_s(\bar{z})|&\leq |(b_s*u_s)(x,v)-(b_s*u_s)(x-s(v-\bar v),v)|\\
&\quad+|(b_s*u_s)(x-s(v-\bar v),v)-(b_s*u_s)(\bar x,\bar v)|\\
&\le s^{\kappa}|v-\bar v|^{\kappa}\|b_s*u_s\|_{\bC^{(1+\alpha)\kappa}_\bba}
+|\transOp_s(z-\bar z)|_\bba^{\kappa}\|b_s*u_s\|_{\bC^{\kappa}_\bba}.
\end{align*}
By \eqref{Con} and \eqref{S2:01}, we have, for any $\beta\geq 0$,
$$
\|b_s*u_s\|_{\bC^{\beta}_\bba}\lesssim 
\|b_s\|_{\bB^{\beta_b}_{\bbp_b;\bba}}\|u_s\|_{\bB^{\beta-\beta_b,1}_{\bbp'_b;\bba}}\lesssim s^{-\frac{\Gap+\beta}{\alpha}}.
$$
Hence,
\begin{align*}
|b_s*u_s(z)-b_s*u_s(\bar{z})|&\lesssim 
(s^{-\frac{\Gap+(1+\alpha)\kappa}{\alpha}+\kappa}|v-\bar v|^{\kappa}+s^{-\frac{\Gap+\kappa}{\alpha}}|\transOp_s(z-\bar z)|_\bba^{\kappa})\\
&\lesssim s^{-\frac{\Gap+\kappa}{\alpha}} (|v-\bar v|^{\kappa}+|\transOp_s(z-\bar z)|_\bba^{\kappa}) .
\end{align*}
Note that for any $z=(x,v)\in \mR^{2d}$, thanks to supp$\phi_N\subset\{(x,v): |N^{(1+\alpha)\modulateorder} x|^{\frac1{1+\alpha}}+|N^{\modulateorder}v|\le C\}= \{z: |z|_\bba\le CN^{-\modulateorder}\}$
$$
(|v|+|\transOp_sz|_\bba)^\kappa\transOp_s\phi_N(z)\lesssim N^{-\modulateorder\kappa}\transOp_s\phi_N(z).
$$
Substituting these into \eqref{B3}, we get for any $\kappa\in\textcolor{Dgreen}{[}0,1)$,
 \begin{align*}
\|H^{3,N}_s\|_{\bbp_0}&\lesssim s^{-\frac{\kappa+\Gap}{\alpha}}
\left\|\int_{\mR^{2d}}(|\cdot-\bar v|^{\kappa}+|\transOp_s(\cdot-\bar{z})|_\bba^\kappa)(\transOp_s\phi_N)(\cdot-\bar{z})\mu^N_s( d  \bar{z})\right\|_{\bbp_0}\\
&\lesssim s^{-\frac{\kappa+\Gap}{\alpha}}
N^{-\modulateorder\kappa}\left\|\int_{\mR^{2d}}\transOp_s\phi_N(\cdot-\bar{z})\mu^N_s( d  \bar{z})\right\|_{\bbp_0}= 
s^{-\frac{\kappa+\Gap}{\alpha}}  
N^{-\kappa\modulateorder}\|u^N_s\|_{\bbp_0}.
\end{align*}
This implies that
\begin{align*}
\|H^{3,N}_s\|_{\bbp_0}\lesssim \left[(s^{-\frac{\kappa+\Gap}{\alpha}} N^{-\kappa\modulateorder})\wedge s^{-\frac{\Gap}{\alpha}}\right]\|u^N_s\|_{\bbp_0}&\leq s^{-\frac{\Gap}{\alpha}}\|\cU^N_s\|_{\bbp_0}+s^{-\frac{\kappa+\Gap}{\alpha}}
N^{-\kappa\modulateorder}\|u_s\|_{\bbp_0}\\
&\stackrel{\eqref{S2:01}}{\lesssim} s^{-\frac{\Gap}{\alpha}}
\|\cU^N_s\|_{\bbp_0}+s^{\frac{\beta_0-\kappa-\Gap}{\alpha}} 
N^{-\kappa\modulateorder}.
\end{align*}
Combining the above calculations, we obtain that: for any $\kappa\in(0,1)$,
\begin{align}\label{08:21}
\begin{split}
  \|H^{N}_s\|_{\bbp_0}\lesssim\|b_s*\cU^N_s\|_{\textcolor{Dgreen}{\boldsymbol{\infty}}}\|u^N_s\|_{\bbp_0}+
s^{-\frac{\Gap}{\alpha}} \|\cU^N_s\|_{\bbp_0}+s^{\frac{\beta_0-\kappa-\Gap}{\alpha}} N^{-\kappa\modulateorder}.
  \end{split}
\end{align}
By \eqref{B0}, \eqref{B1} and \eqref{08:21}, we get, for $\kappa<{\ZXC (\alpha+\beta_0-\Gap)\wedge 1}$,
$$
    \|\cH^{N}_t\|_{\bB^{0,1}_{\bbp_0;\bba}}
    \lesssim \int_0^t (t-s)^{-\frac1{\alpha}}\Big(\|b_s*\cU^N_s\|_{\textcolor{Dgreen}{\boldsymbol{\infty}}}\|u^N_s\|_{p_0} +s^{-\frac{\Gap}{\alpha}}\|\cU^N_s\|_{p_0}\Big) d  s
    +t^{\frac{\alpha-1+\beta_0-\kappa-\Gap}{\alpha}
    }N^{-\kappa\modulateorder},
$$
and (remembering that, from \eqref{DEF_GAP} and $\beta_0\le 0 $, $\frac{\Gap+\beta_0+1}{\alpha}<1$)
$$
\|\cH^{N}_t\|_{\bB^{-\beta_b,1}_{\bbp_b';\bba}}\lesssim \int_0^t (t-s)^{-\frac{\Gap+\beta_0+1}{\alpha}}\Big(\|b_s*\cU^N_s\|_{\textcolor{Dgreen}{\boldsymbol{\infty}}}\|u^N_s\|_{\bbp_0} 
+s^{-\frac{\Gap}{\alpha}}\|\cU^N_s\|_{p_0}\Big) d  s+t^{\frac{\alpha-1-\kappa-2\Gap}{\alpha}} N^{-\kappa\modulateorder}.
$$
In particular, if we take {\ZXC $\kappa=\beta$ with $\beta\in(0,(\alpha+\beta_0-\Gap)\wedge 1)$}, then it is easy to see that
\begin{align*}
\begin{split}
\|\cH^{N}_t\|_{\mS^{\beta}_{t}(b)}&\stackrel{\eqref{Con}}{\lesssim} t^{\frac{\beta-\beta_0}{\alpha}}
\|\cH^{N}_t\|_{\bB^{0,1}_{\bbp_0;\bba}}+t^{\frac{\beta+\Gap}{\alpha}}{\ZXC\|b_t\|_{\bB^{\beta_b}_{\bbp_b;\bba}}}\|\cH^{N}_t\|_{\bB^{-\beta_b,1}_{\bbp_b';\bba}}\\
&\lesssim N^{{\ZXC -\beta}\modulateorder}+\int^t_0
G_\beta(t,s)\Big(\|b_s*\cU^N_s\|_{\textcolor{Dgreen}{\boldsymbol{\infty}}}\|u^N_s\|_{\bbp_0} +s^{-\frac{\Gap}{\alpha}}\|\cU^N_s\|_{\bbp_0}\Big) d  s,
\end{split}
\end{align*}
where $G_\beta(t,s)$ is given as in \eqref{AZ12}.
\subsection{Proof of Lemma \ref{lem:MN}} 
\label{sec:341}
\textcolor{Dgreen}{Our proof arguments are first focused on establishing \textcolor{black}{general} estimates on $\cM^N_t$ successively in the Brownian case $\alpha=2$ (Theorem \ref{08:thm01} and Lemma \ref{08:29} below) and next the pure-jump case $\alpha\in(1,2)$ (Theorem \ref{thm:MartPart-Stable}). Combining these results, the proof of Lemma \ref{lem:MN} is achieved at the end of the section.}

We first show the following estimate.
\bl\label{S3:2/1:lem1}
Let $\alpha=2$, $\bbp\in[2,\infty]^2$ and $\bbp_1\in[1,\infty)^2$ with $\bbp_1\leq\bbp$ and $\beta, \beta_1\in\mR$.
For any $m\ge 1$ and $\beta_2>(\beta-\beta_1)\vee 0$, 
there is a constant $C=C(m,\bbp,\bbp_1,\beta,\beta_1,\beta_2,d)>0$ such that, for any $N\ge 1$ and $f\in \bB^{\beta_1}_{\bbp_1;\bba}$,
\begin{align}\label{S4:08:06}
\sup_{t>0}\|f*\cM^N_t\|_{L^m(\Omega;\bB^{\beta,1}_{\bbp;\bba})} \le C   N^{-\frac12+\modulateorder(3\beta_2+\cA_{\bbp_1,\bbp})}\|f\|_{\bB^{\beta_1}_{\bbp_1;\bba}}.
\end{align}
In particular, for any $\bbp\in[2,\infty]^2$, $\beta\in\mR$, $\beta_2>\beta\vee 0$ and $m\ge 1$, 
\begin{align}\label{S4:08:46}
\|\cM^N_t\|_{L^m(\Omega;\bB^{\beta,1}_{\bbp;\bba})} \le C   N^{-\frac12+\modulateorder(3\beta_2+\cA_{\bbb1,\bbp})}.
\end{align}
\el

\begin{proof}
{\ZXC By \textcolor{Dgreen}{the\comm{JF 16/04: We never specified this mention, so why starting now ...}} embedding \eqref{Sob1}, without loss of generality we may assume $\bbp\in[2,\infty)^2$.}
By the definition of \textcolor{Dgreen}{the anisotropic} Besov norm and Minkowski's inequality, we have\comm{JF 02/04: For the sums start at $j\ge -1$ or $j\ge 0$ ?}
\begin{align*}
 \|f*\cM^N_t\|_{L^m(\Omega;\bB^{\beta,1}_{\bbp;\bba})}
\leq\sum_{j\geq {\color{black}0}}
2^{\beta j} \|f*\cR^\bba_j\cM^{N}_t\|_{L^m(\Omega;\mL^\bbp)}.
\end{align*}
Recalling the definition \textcolor{Dgreen}{\eqref{MART_FIELD_II}} of $\cM^N_t$, we have
\begin{align*}
f*\cR^\bba_j\cM^{N}_t(z)&=\frac1N\sum_{i=1}^N\int^t_0f*\cR^\bba_jP_{t-s}((\nabla_v\transOp_s\phi_N)(Z^{N,i}_s-\textcolor{black}{\cdot}))(z) d  W^i_s\\
&=\frac1N\sum_{i=1}^N\int^t_0(f*\cR^\bba_jP_{t-s}\nabla_v\transOp_s\phi_N)(\transOp_{t-s}Z^{N,i}_s-z) d  W^i_s,
\end{align*}
where we have used the symmetry property $\phi_N(z)=\phi_N(-z)$ and, for a function $g:\mR^{2d}\to\mR$ and $z'\in\mR^{2d}$,
\begin{align}\label{AQ1}
P_{t-s}(g(z'-\textcolor{black}{\cdot}))(z)=(P_{t-s}g)(\transOp_{t-s}z'-z).
\end{align}
Next let us define, for $t$  fixed, the stopped process 
$$	f*\cR^\bba_j\cM^{N}_{u,t}(z)=\frac1N\sum_{i=1}^N\int^{\textcolor{black}u}_0(f*\cR^\bba_jP_{t-s}\nabla_v\transOp_s\phi_N)(\transOp_{t-s}Z^{N,i}_s-z) d  W^i_s,\quad u\in[0,t],
$$
which defines an $\mL^{\bbp}$-valued martingale. With this, we are in position (as in \cite{OlRiTo-21}) to apply a functional BDG (Burkh\"older-Davis-Gundy) inequality to derive the control \eqref{S4:08:06}.  Following e.g. \cite{VerYar-19} (see Theorem 16.1.1 therein), this BDG inequality states that, given $(E,|\cdot|_E)$ a UMD (unconditional martingale difference) Banach space, for any $E$-valued martingale $\{\mathcal m_t\}_{t\ge 0}$ with quadratic variation $\{[\mathcal m]_t\}_{t\ge 0}$ and any $p\in (1,\infty)$, we have, for all $t\ge 0$,  
	\begin{equation}\label{BDG}
	\|\sup_{0\le u\le t}|\mathcal m_u|_E\|_{L^p(\Omega)}{\color{black}\le c_{E,p}}\sup_{0\le u\le t}\||\mathcal m_u|_E\|_{L^p(\Omega)}{\color{black}\le C_{E,p}}  \||[\mathcal m]^{1/2}_t|_E\|_{L^{p}(\Omega)},
\end{equation}
	for some constants $C_{E,p}=C(E,p)$ and $c_{E,p}=c(E,p)$. 
The property that \textcolor{Dgreen}{the mixed space} $\mL^{\bbp}$ is indeed a UMD space for $\bbp\in(1,\infty)^2$ - along a recall of the notion of UMD spaces - is established in the appendix section (see Corollary \ref{cor:UMD}). 
Consequently, applying \eqref{BDG} {\color{black}with $m>1$ (otherwise we use H\"older's inequality to increase the exponent)}, we have 
\begin{align}
\mE \textcolor{Dgreen}{\big[}\|f*\cR^\bba_j\cM^{N}_t\|_{\bbp}^m\textcolor{Dgreen}{\big]}&\lesssim \sup_{0\le u\le t}\mE\textcolor{Dgreen}{\big[} \|f*\cR^\bba_j\cM^{N}_{u,t}\|_{\bbp}^m\textcolor{Dgreen}{\big]}\lesssim \mE\textcolor{Dgreen}{\big[} \|[f*\cR^\bba_j\cM^{N}_{u,t}]_{u=t}\|_{\bbp}^m\textcolor{Dgreen}{\big]}
\no\\
&\lesssim\mE \left\|\left(\frac 1{N^2}\sum_{i=1}^N\int^t_0|(f*\cR^\bba_jP_{t-s}\nabla_v\transOp_s\phi_N)(\transOp_{t-s}Z^{N,i}_s-\cdot)|^2 d  s\right)^{1/2}\right\|_{\bbp}^m\no\\
&\lesssim N^{-m}
\mE\left\|\left(\sum_{i=1}^N\int^t_0|(f*\cR^\bba_jP_{t-s}\nabla_v\transOp_s\phi_N)(\transOp_{t-s}Z^{N,i}_s-\cdot)|^2 d  s\right)^{1/2}\right\|^m_{\bbp}\no\\
&=N^{-m}\mE\left\|\sum_{i=1}^N\int^t_0|(f*\cR^\bba_jP^\kappa_{t-s}\nabla_v\transOp_s\phi_N)(\transOp_{t-s}Z^{N,i}_s-\cdot)|^2 d  s\right\|^{m/2}_{\bbp/2}\no\\
&\leq N^{-m}\mE\left(\sum_{i=1}^N\int^t_0\|(f*\cR^\bba_jP_{t-s}\nabla_v\transOp_s\phi_N)(\transOp_{t-s}Z^{N,i}_s-\cdot)\|^2_{\bbp} d  s\right)^{m/2}\no\\
&= N^{-m/2}\left(\int^t_0\|f*\cR^\bba_jP_{t-s}\nabla_v\transOp_s\phi_N\|^2_{\bbp} d  s\right)^{m/2}.\label{AC:pre}
\end{align}
Let $\bbp_2\in[1,\infty]^2$ be defined \textcolor{Dgreen}{by} $\textcolor{Dgreen}{\bbb1}+\frac1{\bbp}=\frac1{\bbp_1}+\frac1{\bbp_2}$.
By Young's inequality, \textcolor{black}{the D}efinition \textcolor{Dgreen}{\ref{bs} of the anisotropic Besov norm, } together with \eqref{Ph0}, and \textcolor{Dgreen}{the heat kernel estimate} \eqref{AD0306}, we have
\begin{align*}
\|f*\cR^\bba_jP_{t-s}\nabla_v\transOp_s\phi_N\|_{\bbp}
&\lesssim 2^{-\beta_1 j}\|f\|_{\bB^{\beta_1}_{\bbp_1;\bba}}\|\cR^\bba_jP_{t-s}\nabla_v\transOp_s\phi_N\|_{\bbp_2}\\
&\lesssim2^{j(1-\beta_2-\beta_1)}\|f\|_{\bB^{\beta_1}_{\bbp_1;\bba}}((2^{2j}(t-s))^{-1}\wedge 1)\|\nabla_v\transOp_s\phi_N\|_{\bB^{\beta_2-1}_{\bbp_2;\bba}}.
\end{align*}
Observing that
 \begin{equation}\label{INT_DYAD_COMP}
 \int^t_0((2^{2j} s)^{-1}\wedge 1)^2 d  s\lesssim 2^{-2j},
 \end{equation}
  we further have, by {\black  \eqref{BerCor} in Corollary \ref{LiftEstimate} and \eqref{Es3_Temp}
  }\textcolor{Dgreen}{in Appendix \ref{APP_BESOV}}, 
\begin{align*}
    \left(\int^t_0\|f*\cR^\bba_jP_{t-s}\nabla_v\transOp_s\phi_N\|_{\bbp}^2 d  s\right)^{1/2}&{\black \lesssim 2^{j(1-\beta_2-\beta_1)} \|f\|_{\bB^{\beta_1}_{\bbp_1;\bba}}\left(\int_0^t((2^{2j}(t-s))^{-1}\wedge 1)^2\|\nabla_v\transOp_s\phi_N\|^2_{\bB^{\beta_2-1}_{\bbp_2;\bba}}\,ds\right)^{1/2}}\\
    &{\black \lesssim 2^{-j(\beta_2+\beta_1)}\|f\|_{\bB^{\beta_1}_{\bbp_1;\bba}}\sup_{s\in[0,t]}\|\transOp_s\phi_N\|_{\bB^{\beta_2}_{\bbp_2;\bba}}}\\
&\lesssim 2^{-j(\beta_2+\beta_1)}\|f\|_{\bB^{\beta_1}_{\bbp_1;\bba}}N^{\modulateorder(3\beta_2+{\ZXC \cA_{\bbp_1,\bbp}})}.
\end{align*}
Combining the above calculations, we get
$$
\|f*\cM^N_t\|_{L^m(\Omega;\bB^{{\beta},1}_{\bbp;\bba})}
\lesssim  \sum_{j\geq \textcolor{Dgreen}{0}}2^{\beta j}2^{-j(\beta_2+\beta_1)}\|f\|_{\bB^{\beta_1}_{\bbp_1;\bba}}N^{\modulateorder(3\beta_2+{\ZXC \cA_{\bbp_1,\bbp}})}.
$$
The result \eqref{S4:08:46} now follows by $\beta<\beta_2+\beta_1$. The estimate \eqref{S4:08:46} is deduced next by taking $f$ the Dirac measure on $0$, which lies in $\bB^{\beta_1}_{\bbp_1;\bba}$ for $\beta_1=0$ and $\bbp_1=\bbb1$.
\end{proof}
  
The previous lemma is \textcolor{Dgreen}{naturally} not enough for \textcolor{Dgreen}{the proof of Lemma \ref{lem:MN}} since it requires $\bbp\geq 2$. To drop this restriction, we use, as in \cite{OlRiTo-21}, weight function techniques  to show the following stronger estimate. The price we have to pay for this procedure is that we need uniform moment estimates on {\ZXC the process $Z^{N,1}_t$}.
\bt\label{08:thm01}
Let $\alpha=2$, $\beta\geq 0$\textcolor{Dgreen}{,} $\ZXC \bbp\in[1,\infty]^2$ with $p_{x}\wedge p_v<2$ and $\ell>\cA_{\bbp,\bbp\vee \textcolor{Dgreen}{\bb2}}$. For any {\ZXC $\beta_1>\beta$ and $m\ge2$}, there is a constant 
$C=C(\Theta,\textcolor{Dgreen}{\bbp,\ell,}\beta,\beta_1,m)>0$ such that, for any $N\ge 1$,
\begin{align}\label{08:06}
\sup_{t\in[0,T]}\|\cM^N_t\|_{L^m(\Omega;\bB^{\beta,1}_{\bbp;\bba})}
\le C  \left[1+{\ZXC \sup_{s\in[0,T]}(\mE\textcolor{Dgreen}{\big[}|Z^{N,1}_s|^{\ell m}_\bba\textcolor{Dgreen}{\big]})^{1/m}}\right]
N^{\modulateorder(3\beta_1+{\ZXC  \cA_{\textcolor{Dgreen}{\bbb1},\bbp\vee{\bf 2}}})-\frac12}.
\end{align}
\et
\begin{proof}
 {\ZXC By \textcolor{Dgreen}{the} embedding \eqref{Sob1}, without loss of generality, we may \textcolor{Dgreen}{(again)} assume $\bbp\in[1,\infty)^2$.}
 Note again from the Minkowski inequality that
\begin{align}\label{CB1}
\ZXC \|\cM^N_t\|_{L^m(\Omega;\bB^{\beta,1}_{\bbp;\bba})}
\le \sum_{j\textcolor{Dgreen}{\ge 0}}2^{\beta j} \|\cR^\bba_j\cM^{N}_t\|_{L^m(\Omega; \mL^\bbp)}.
\end{align}
Whenever $\bbp<\bb2$, one cannot directly make an estimate by BDG's inequality since $\mL^{\bbp/2}$
is not a Banach space. For simplicity of notation, we write $\bbp_1:=\bbp\vee \textcolor{Dgreen}{\bb2}$ and let $\bbp_2\in[1,\infty]^2$ be defined by $\frac{1}{\bbp}=\frac{1}{\bbp_1}+\frac{1}{\bbp_2}$.
To overcome the difficulty, we \textcolor{Dgreen}{now specifically} use a weight function: 
\begin{align}\label{Om1}
{\ZXC \omega_\ell(z):=\textcolor{Dgreen}{\big(}1+(1+|x|^2)^{\frac1{1+\alpha}}+|v|^2\textcolor{Dgreen}{\big)}^{\ell/2},\ \ z=(x,v)\in\mR^{2d}.}
\end{align}
\textcolor{Dgreen}{for $\ell>\bba\cdot\frac d{\bbp_2}=\cA_{\bbp,\bbp_1}$}. \textcolor{black}{The} \textcolor{Dgreen}{main properties and related key estimates on $\omega_\ell$ are stated in Appendix \ref{sec:WeightedEstimates}.}
Clearly, we have
$$
 \omega_\ell(z+z')\lesssim \omega_\ell(z)+\omega_\ell(z'), \ \ \textcolor{Dgreen}{z,z'\in\mathbb R^{2d},}
$$
\textcolor{Dgreen}{and t}he choice of $\ell$ also precisely guarantees the following integrability of the inverse weight,
$$\|\omega^{-1}_\ell\|_{\bbp_2}<\infty.$$
Since $\mL^{\bbp_1}$ is \textcolor{Dgreen}{(now)} a UMD space, by H\"older's inequality and BDG's inequality \textcolor{Dgreen}{\eqref{BDG} applied} {\ZXC for $\mL^{\bbp_1}$-valued martingale} (and introducing, for the application of the inequality, an appropriate stopped version of $\cR^\bba_j\cM^{N}_t$ as in the proof of Lemma \ref{S3:2/1:lem1}), we have
\begin{align}
\mE \textcolor{Dgreen}{\big[}\|\cR^\bba_j\cM^{N}_t\|_{\bbp}^m \textcolor{Dgreen}{\big]}
&\leq \mE\textcolor{Dgreen}{\big[}\big \|\cR^\bba_j\cM^{N}_t\omega_\ell\big\|_{\bbp_1}^m\|\omega^{-1}_\ell\|^m_{\bbp_2}\textcolor{Dgreen}{\big]}\no\\
&\lesssim N^{-m}
\mE\left\|\left(\sum_{i=1}^N\int^t_0|(\cR^\bba_jP_{t-s}\nabla_v\transOp_s\phi_N)(\transOp_{t-s}Z^{N,i}_s-\cdot)\omega_\ell|^2 d  s\right)^{1/2}\right\|^m_{\bbp_1}\no\\
&=N^{-m}\mE\left\|\sum_{i=1}^N\int^t_0|(\cR^\bba_jP_{t-s}\nabla_v\transOp_s\phi_N)(\transOp_{t-s}Z^{N,i}_s-\cdot)\omega_\ell|^2 d  s\right\|^{m/2}_{\bbp_1/2}\no\\
&\leq N^{-m}\mE\left(\sum_{i=1}^N\int^t_0\|(\cR^\bba_jP_{t-s}\nabla_v\transOp_s\phi_N)(\transOp_{t-s}Z^{N,i}_s-\cdot)\omega_\ell\|^2_{\bbp_1} d  s\right)^{m/2}.\label{LJ1}
\end{align}
Noting that for a nonnegative function $g:\mR^{2d}\to[0,\infty)$,
\begin{align*}
g(\transOp_{t-s}Z^{N,i}_s-\cdot)\omega_\ell
&\lesssim (g\omega_\ell)(\transOp_{t-s}Z^{N,i}_s-\cdot)+g(\transOp_{t-s}Z^{N,i}_s-\cdot)\omega_\ell(\transOp_{t-s}Z^{N,i}_s)\\
&\lesssim (g\omega_\ell)(\transOp_{t-s}Z^{N,i}_s-\cdot)(1+|Z^{N,i}_s|^\ell),
\end{align*}
by the translation  invariance of $\mL^\bbp$-norm, we have
\begin{align*}
\ZXC\|(\cR^\bba_jP_{t-s}\nabla_v\transOp_s\phi_N)(\transOp_{t-s}Z^{N,i}_s-\cdot)\omega_\ell\|_{\bbp_1}
\lesssim\|\textcolor{Dgreen}{\big(}\cR^\bba_jP_{t-s}\nabla_v\transOp_s\phi_N\textcolor{Dgreen}{\big)}\omega_\ell\|_{\bbp_1}(1+|Z^{N,i}_s|_{\textcolor{Dgreen}{\bba}}^\ell).
\end{align*}
Substituting this into \eqref{LJ1} and by $Z^{N,i}_s\overset{({\rm law})}{=} Z^{N,1}_s$, we further have\comm{From JF to S 01/04: Just the addition of the brackets for the expectation and our convention $m\ell \rightarrow \ell m$. (Same reasons apply below.)}
\begin{align}\label{CB2}
\begin{split}
\mE \textcolor{Dgreen}{\big[}\|\cR^\bba_j\cM^{N}_t\|_{\bbp}^m\textcolor{Dgreen}{\big]}
&\lesssim N^{-m/2}
\left(1+\sup_{s\in[0,t]}\mE\textcolor{Dgreen}{\big[}|Z^{N,1}_s|_\bba^{\textcolor{Dgreen}{\ell m}}\textcolor{Dgreen}{\big]}\right)
\left(\int^t_0\big\| \textcolor{Dgreen}{\big(}\cR^\bba_jP_{t-s}\nabla_v\transOp_s\phi_N\textcolor{Dgreen}{\big)} \omega_\ell\big\|^2_{\bbp_1} d  s\right)^{m/2}.
\end{split}
\end{align}

\textcolor{Dgreen}{
At this stage, we can apply Lemma \ref{WeightedHeatEstimate} from Appendix \ref{sec:WeightedEstimates} - with $\beta=\beta_1$, $\bbp=\bbp_1$ and $\gamma=1$ - to get, for any $j\ge 0$,
\begin{align*}
\|\big(\cR^\bba_jP_{t-s}\nabla_v\transOp_{s}\phi_N\big)\omega_\ell\|^2_{\bbp_1} 
\lesssim 2^{2j(1-\beta_1)}\big((2^{2j}(t-s))^{-1}\wedge 1\big)^2\|(\transOp_{s}\phi_N)\omega_\ell\|^2_{\bB^{\beta_1}_{\bbp_1;\bba}}.
\end{align*}
}
\textcolor{Dgreen}{Hence,} applying Lemma \ref{08:18}, \textcolor{Dgreen}{it follows that}
\begin{align*}
&\int^t_0\|\textcolor{Dgreen}{\big(}\cR^\bba_jP_{t-s}\nabla_v\transOp_s\phi_N\textcolor{Dgreen}{\big)}\omega_\ell\|^2_{\bbp_1} d  s\\
&\quad\lesssim 2^{2j(1-\beta_1)}  \int^t_0((2^{2j} (t-s))^{-1}\wedge 1)^2\|\omega_\ell\transOp_s\phi_N\|^2_{\bB^{\beta_1}_{\bbp_1;\bba}} d  s\\
&\quad\lesssim (1+(tN^{2\modulateorder})^{2\beta_1})N^{2\modulateorder(\beta_1+{\ZXC  \cA_{\textcolor{Dgreen}{\bbb1},\bbp_1}})}2^{2j(1-\beta_1)}
\int^t_0((2^{2j} s)^{-1}\wedge 1)^2 d  s\\
&\quad\lesssim N^{2\modulateorder(3\beta_1+{\ZXC  \cA_{\textcolor{Dgreen}{\bbb1},\bbp_1}})}2^{-2j\beta_1},
\end{align*}
using again \eqref{INT_DYAD_COMP} for the last inequality.
Substituting this into \eqref{CB2} and by \eqref{CB1}, we obtain
\begin{align*}
\|\cM^N_t\|_{L^m(\Omega;\bB^{\beta,1}_{\bbp;\bba})}
&\lesssim N^{-1/2}\left(1+{\ZXC\sup_{s\in[0,t]}(\mE\textcolor{Dgreen}{\big[}|Z^{N,1}_s|_\bba\textcolor{Dgreen}{\big]}^{\textcolor{Dgreen}{\ell m}}})^{1/m}\right)
\sum_j2^{\beta j}N^{\modulateorder(3\beta_1+{\ZXC  \cA_{\textcolor{Dgreen}{\bbb1},\bbp_1}})}2^{-j\beta_1},
\end{align*}
{\ZXC which immediately implies the estimate \eqref{08:06} by $\beta_1>\beta$.}
\end{proof}

To conclude the proof of Lemma \ref{lem:MN} for $\alpha=2$, it remains to establish the uniform moment estimates for the solution to the interacting particle system \eqref{S1:00} when $\bbp_0=\bbb1$. These estimates are given by the following lemma.
\bl
\label{08:29}
Let $\alpha=2$ and $\bbp_0=\bbb 1 $ and assume the condition {\bf (H)} holds. Then, for any  $m\in\mN$, 
there is a constant $C=C(\textcolor{Dgreen}{\Theta,m})>0$\comm{JF 01/04: Just to respect the convention of order: $\Theta$ first, the rest after.} such that for all $t\in(0,\Tup\textcolor{Dgreen}{\big]}$ and $N\ge 1$,
\begin{align}\label{AQ2}
\|b_t*u^N_t\|_{L^m(\Omega;\mL^{\textcolor{Dgreen}{\boldsymbol\infty}})}\le Ct^{-\frac{\Gap}{2}}+C\sup_{t\in[0,T_0]}\|b_t*\cM^N_t\|_{L^m(\Omega;\mL^{\textcolor{Dgreen}{\boldsymbol\infty}})}.
\end{align}
In particular,  
\begin{align}\label{AQ3}
\sup_{t\in[0,T_0]}\mE\textcolor{Dgreen}{\big[}|Z^{N,1}_t|_\bba^m\textcolor{Dgreen}{\big]}\lesssim 1+\mE\textcolor{Dgreen}{\big[}|Z^{N,1}_0|_\bba^m\textcolor{Dgreen}{\big]}+\sup_{t\in[0,T_0]}\|b_t*\cM^N_t\|^{\textcolor{Dgreen}{m}}_{L^m(\Omega;\mL^{\textcolor{Dgreen}{\boldsymbol\infty}})}.
\end{align}
\el
\begin{proof}
By \eqref{S1:P2} and  Duhamel's formula (with the notations of \eqref{Duh}), we have
\begin{align*}
u^N_t(z)=P_tu^N_0(z)+\int_0^t P_{t-s}\div_v G_s^{N}(z) d  s+\cM^N_t(z).
\end{align*}
By \eqref{Con} and \eqref{AD0446}, we thus get
\begin{align*}
\|b_t*u^N_t\|_{\mL^{\textcolor{Dgreen}{\boldsymbol\infty}}}
&\leq
\|b_t*P_tu^N_0\|_{\mL^{\textcolor{Dgreen}{\boldsymbol\infty}}}
+\int_0^t \|b_t*P_{t-s}\div_v G_s^{N}\|_{\mL^{\textcolor{Dgreen}{\boldsymbol\infty}}}
 d  s+\|b_t*\cM^N_t\|_{\mL^{\textcolor{Dgreen}{\boldsymbol\infty}}}\\
&\lesssim
\|b_t\|_{\bB^{\beta_b}_{\bbp_b;\bba}}\|P_tu^N_0\|_{\bB^{-\beta_b,1}_{\bbp'_b;\bba}}+\int_0^t \|b_t\|_{\bB^{\beta_b}_{\bbp_b;\bba}}\|P_{t-s}\div_v G_s^{N}\|_{\bB^{-\beta_b,1}_{\bbp'_b;\bba}} d  s
+\|b_t*\cM^N_t\|_ {\mL^{\textcolor{Dgreen}{\boldsymbol\infty}}}\\
&\lesssim
t^{\ZXC (\beta_b-\cA_{\bbb1,\bbp'_b})/2}\|u^N_0\|_{\bB^{0}_{\bbb1;\bba}}+\int_0^t (t-s)^{\ZXC (\beta_b-\cA_{\bbb1,\bbp'_b}-1)/2}\|G_s^{N}\|_{\bB^{0}_{\bbb1;\bba}} d  s+\|b_t*\cM^N_t\|_ {\mL^{\textcolor{Dgreen}{\boldsymbol\infty}}}.
\end{align*}
Observing as well  that 
$$
\|u^N_0\|_{\bB^0_{\bbb1;\bba}}\lesssim\|u^N_0\|_{{\bbb1}}=\|\mu^N_t*\transOp_t\phi_N\|_{{\bbb1}}=1,
$$
and, from \eqref{DEF_G},
\begin{align*}
|G_s^{N}(z)|\le \|b_s*u^N_s\|_{\mL^{\textcolor{Dgreen}{\boldsymbol\infty}}} 
u^N_s(z)\Rightarrow\|G_s^{N}\|_{\bbb1}\leq\|b_s*u^N_s\|_{\mL^{\textcolor{Dgreen}{\boldsymbol\infty}}}.
\end{align*}
From  the above computations, using again \eqref{AB2} and from \eqref{DEF_GAP}, 
$$
\cA_{\bbb1,\bbp'_b}-\beta_b\le \cA_{\bbb1,\bbp'_b}-\beta_b-\beta_0=\Lambda<1,
$$
 we have, {\ZXC for all $t\in(0,T_0]$,}
\begin{align*}
\|b_t*u^N_t\|_{\mL^{\textcolor{Dgreen}{\boldsymbol\infty}}}
\lesssim t^{-\frac{\Gap}{2}} +\int_0^t (t-s)^{-\frac{1+\Gap}{2}}\|b_s*u^N_s\|_{\mL^{\textcolor{Dgreen}{\boldsymbol\infty}}}
 d  s+\|b_t*\cM^N_t\|_{\mL^{\textcolor{Dgreen}{\boldsymbol\infty}}}.
\end{align*}
Therefore,
$$
\|b_t*u^N_t\|_{L^m(\Omega;\mL^{\textcolor{Dgreen}{\boldsymbol\infty}})}\lesssim t^{-\frac{\Gap}{2}}+ \|b_t*\cM^N_t\|_{L^m(\Omega;\mL^{\textcolor{Dgreen}{\boldsymbol\infty}})}
+\int_0^t (t-s)^{-\frac{1+\Gap}{2}}\|b_s*u^N_s\|_{L^m(\Omega;\mL^{\textcolor{Dgreen}{\boldsymbol\infty}})} d  s.
$$
Since $\Gap<1$, by Gronwall's inequality of Volterra-type (see Lemma \ref{0214:lem00}
), we obtain \eqref{AQ2}.
 
In particular, by \eqref{S1:00}, we have, for any {\color{black}$m\ge 0$},
\begin{align*}
\sup_{t\in[0,T_0]}\mE\textcolor{Dgreen}{\big[}|Z^{N,1}_t|_\bba^m\textcolor{Dgreen}{\big]}\lesssim 1+\mE\textcolor{Dgreen}{\big[}|Z^{N,1}_0|_\bba^m\textcolor{Dgreen}{\big]}+\int_0^{T_0}\mE\textcolor{Dgreen}{\big[}\|b_s*u^N_s\|^m_{\mL^{\textcolor{Dgreen}{
			\boldsymbol{\infty}}}}\textcolor{Dgreen}{\big]} d  s.
\end{align*}
Estimate \eqref{AQ3} now follows by \eqref{AQ2}.
\end{proof}

 \br
 \ZXC
For $\bbp_0\neq \bbb1$, \textcolor{Dgreen}{let us point out that }we have a direct treatment for $\|b_t*u^N_t\|_{\mL^{\textcolor{Dgreen}{\boldsymbol{\infty}}}}$, but with \textcolor{black}{a worse} convergence rate. Indeed, by \eqref{Con} and \eqref{S1:02}, 
\begin{align*}
\|b_t*u^N_t\|_{\mL^{\textcolor{Dgreen}{
		\boldsymbol{\infty}}}}&\lesssim  \|b_t\|_{\bB^{\beta_b}_{\textcolor{Dgreen}{\bbp_b};\bba}}
	\|u^N_t\|_{\bB^{-\beta_b,1}_{\textcolor{Dgreen}{\bbp_b'};\bba}}=\|b_t\|_{\bB^{\beta_b}_{\textcolor{Dgreen}{
			\bbp_b};\bba}}\|\mu^N_t*\transOp_t{\phi_N}\|_{\bB^{-\beta_b,1}_{\textcolor{Dgreen}{\bbp_b'};\bba}}\\
&\le \|b_t\|_{\bB^{\beta_b}_{\textcolor{Dgreen}{\bbp_b};\bba}}\|\mu^N_t\|_{\bB^{0}_{
	\bbb1;\bba}}\|\transOp_t{\phi_N}\|_{\bB^{-\beta_b,1}_{\textcolor{Dgreen}{\bbp_b'};\bba}}\lesssim \|\transOp_t{\phi_N}\|_{\bB^{-\beta_b+\frac \varepsilon3}_{\textcolor{Dgreen}{\bbp_b'};\bba}}\lesssim N^{\zeta(-3\beta_b+\varepsilon+\bba \cdot \frac d{\bbp_b})}
\end{align*}
using \eqref{AB2} for the last but one inequality and Lemma \ref{08:18} for the last one.
 \er
 
 
{\ZXC Next we turn to the treatment of $\alpha\in(1,2)$. In this case, we need to use a \textcolor{Dgreen}{specific martingale} inequality established in \cite{Hausenblas-11}
.}
We first \textcolor{Dgreen}{state the following estimate}.
\bl\label{Le304}
For any $\beta\in\mR$, $\bbp\in[1,\infty]^2$ and $0\leq \eps\leq\theta<1$, there is a constant $C>0$ such that for all $t>0$, $v\in\mR^d$ and
$f\in \bB^{\beta+2\eps}_{\bbp;\bba}$,
$$
\|P_t\delta^{(1)}_{(0,v)}f\|_{\bB^{\beta,1}_{\bbp;\bba}}\le C  \big[\big(|v|^\theta t^{-\frac {\theta-\eps}\alpha}\big)\wedge 1\big]
\|f\|_{\bB^{\beta+2\eps}_{\bbp;\bba}}.
$$
\el
\begin{proof}
Let $\bbp\in[1,\infty]^2$. By \eqref{AB2} and \textcolor{Dgreen}{the heat kernel estimate} \eqref{AD0446}, we have, for any $\gamma\geq 0$,
\begin{align*}
\big\|P_t\delta^{(1)}_{(0,v)}f\big\|_{\bB^{\beta,1}_{\bbp;\bba}}\lesssim\big\|P_t\delta^{(1)}_{(0,v)}f\big\|_{\bB^{\beta+\eps}_{\bbp;\bba}}
\lesssim  t^{-\frac \gamma\alpha}\|\delta^{(1)}_{(0,v)}f\|_{\bB^{\beta-\gamma+\eps}_{\bbp;\bba}}.
\end{align*}
On the other hand, by \textcolor{Dgreen}{the equivalent form of }\eqref{CH1}, we have, for any $\beta\in\mR$ and $\gamma\in[0,1)$,
\begin{align*}
\|\delta^{(1)}_{(0,v)}f\|_{\bB^{\beta}_{\bbp;\bba}}&=\sup_{j\geq 0}2^{j\beta}\|\cR^\bba_j\delta^{(1)}_{(0,v)}f\|_{\bbp}
=\sup_{j\geq 0}2^{j\beta}\|\delta^{(1)}_{(0,v)}\cR^\bba_jf\|_{\bbp}\\
&\lesssim |v|^\gamma\sup_{j\geq 0}2^{j\beta}\|\cR^\bba_jf\|_{\bB^{\gamma}_{\bbp;\bba}}\lesssim|v|^\gamma\|f\|_{\bB^{\beta+\gamma}_{\bbp;\bba}}.
\end{align*}
Combining the above two estimates, we obtain
$$
\big\|P_t\delta^{(1)}_{(0,v)}f\big\|_{\bB^{\beta,1}_{\bbp;\bba}}\lesssim\|f\|_{\bB^{\beta+\eps}_{\bbp;\bba}}
$$
and for any $0\leq \eps\leq\theta<1$,
$$
\big\|P_t\delta^{(1)}_{(0,v)}f\big\|_{\bB^{\beta,1}_{\bbp;\bba}}\lesssim t^{-\frac {\theta-\eps}\alpha}\|\delta^{(1)}_{(0,v)}f\|_{\bB^{\beta-\theta+2\eps}_{\bbp;\bba}}
\lesssim |v|^\theta t^{-\frac{\theta-\eps}\alpha}\|f\|_{\bB^{\beta+2\eps}_{\bbp;\bba}}.
$$
The proof is complete.
\end{proof}
Now we can show the following estimate.
 	\bt\label{thm:MartPart-Stable} Let $T>0$, $1<\alpha<2$ and $\bbp =(p_x,p_v)\in(1,\infty]^2$ with
 	$$
 	\alpha<q:=2\wedge p_x\wedge p_v.
 	$$
 	Then, for any $\beta\ge 0$, $m\ge 1$ and $\epsilon>0$,
 	there exists $C=C(\textcolor{Dgreen}{T,}\alpha,\bbp,\beta,m\textcolor{Dgreen}{,\eps})>0$ such that for all $N\textcolor{Dgreen}{\ge 1}$,
\begin{align}\label{ZH1}
 	\sup_{t\in [0,T]}\|\cM^N_t\|_{L^m(\Omega;\bB^{\beta,1}_{\bbp;\bba})} \le C   N^{\frac1q-1+\zeta((\alpha+1)\beta+\cA_{\bbb1,\bbp}+\eps)}.
\end{align}
    	\et
    		As a preliminary to the proof of Theorem \ref{thm:MartPart-Stable}, let us highlight that applying the functional BDG inequality \eqref{BDG} previously used for the case Brownian case $\alpha=2$ would naturally bring rather intricate moment issues. \textcolor{Dgreen}{It is more natural to rely on a specific martingale inequality which fits the case of integrals with compensated Poisson measures. As we refer the interested reader to \cite[Section 5]{MarRoc-14} for \textcolor{black}{a} historical account of maximal inequalities for such integrals, we will make use below of the particular estimate \textcolor{black}{established in}  \cite{Hausenblas-11}.}
    		\bt[\cite{Hausenblas-11}, Theorem 2.13 and Corollary 2.14]\label{thm:Hausenblas-BDG} Let $(E,|\cdot|_E)$ be a separable Banach space of martingale type $p$ with $1<p\le 2$, $(Z,\mathcal Z)$ be a measurable space and $\nu$ be a positive $\sigma$-finite measure on $Z$. Let $(\Omega,\mathcal F,(\mathcal F_t)_{t\ge 0},\mathbb P)$ be a complete filtered probability space with right continuous filtration. Assume that $\tilde \eta$ is a compensated time homogeneous Poisson random measure on $Z$ over  $(\Omega,\mathcal F,(\mathcal F_t)_{t\ge 0},\mathbb P)$ with intensity $\nu$. Then, for any progressively measurable process $\xi:\R_+\times Z\times\Omega\rightarrow E$ such that, for some $n\ge 1$,
    		$$
    		\mathbb E\Big[\int_0^\infty\int_{Z} |\xi(t,z)|_E^p\nu(dz)\, d  t+\int_0^\infty\int_{Z} |\xi(t,z)|_E^{p^n}\nu(dz)\, d  t\Big]<\infty
    		$$	
    		we have
    		\begin{equation}\label{BDG-Jump}
    			\mathbb E\Big[\sup_{0\le t\le T}\Big{|}\int_0^t\int_{Z}\xi(s,z)\widetilde{\eta}(ds,dz)
    			\Big{|}^{p^n}_E\Big]\le C  \sum_{k=1}^n\mathbb E\Bigg[\Big(\int_0^T\int_Z|\xi(s,z)|_E^{p^k}\nu(dz)\, d  s\Big)^{p^{n-k}}\Bigg],
    		\end{equation}
    		for some positive constant $C=C(\textcolor{Dgreen}{E,p,n})$.
    		\et
    		As for the notion of UMD spaces, the definition of Banach space of martingale type $p$ is recalled in \textcolor{Dgreen}{Appendix} \ref{sec:Type}. Therein, we also establish that, under the assumption $\alpha<1\wedge p_x\wedge p_v$, the inequality  \eqref{BDG-Jump} applies to the case $E=\mL^{\bbp}$ (see \textcolor{Dgreen}{Corollary} \ref{MixedBanachType}).
\begin{proof}[Proof of Theorem \ref{thm:MartPart-Stable}]
\textcolor{Dgreen}{As in the previous proof of this section, a}ccording to the embedding \eqref{Sob1}, we may assume $\bbp\in(1,\infty)^2$ without loss of generality.
As $\alpha\in(1,2)$, by \eqref{MART_FIELD_I}, \eqref{MMN} and \eqref{Duh}, we can write
	\begin{align*}
 		\cM^{N}_{t}(z)=\frac1{N}\sum_{i=1}^N\int_0^t\!\!\int_{\mR^d\setminus\{0\}} \xi^i_t(s,v_{\textcolor{Dgreen}i})(z)
		\tilde{\mathcal{N}}^{i}( d  s, d  v_{\textcolor{Dgreen}i}),
 	\end{align*}
	where (using again the \textcolor{Dgreen}{difference} operator \eqref{Dif1}),
	$$
	\xi^i_t(s,v)(z):=P_{t-s}\Big(\delta^{(1)}_{(0,v)}\transOp_{s}\Phi_N\Big)(Z^{N,i}_{s-}-\cdot)(z).
	$$
To use appropriately the inequality \eqref{BDG-Jump}, we first need to lift $\cM^{N}$ on the product space $\mR^{Nd}_0:=\mR^{Nd}\setminus\{0\}$ as follow\textcolor{Dgreen}{s}: For
	\begin{align*}
\boldsymbol{L}^N_t:=(L^{\alpha,1}_t,\cdot\cdot\cdot, L^{\alpha,N}_t), 
\end{align*}
the overall noise driving the particle system \eqref{S1:00}, let ${\boldsymbol\cN}^N((0,t],U)$ denote the jump measure of  $\boldsymbol{L}^N$ and $\widetilde{\boldsymbol\cN}^N$ the related compensated measure, respectively defined as: for all $t\in(0,T]$ and  $U  \in\sB(\mR^{Nd}_0)$, 
$$
{\boldsymbol\cN}^N((0,t],U):=\sum_{0<s\le t}\ind_{U}(\Delta \boldsymbol{L}^N_s),\ \ 
{\widetilde{\boldsymbol\cN}^N}( d  s, d  \boldsymbol{v}):={\boldsymbol\cN}^N( d  s, d  \boldsymbol{v})-\boldsymbol{\nu}( d \boldsymbol{v})\textcolor{Dgreen}{ d  s},
$$
where $\boldsymbol{\nu}( d \boldsymbol{v})$ is the L\'evy measure of $\boldsymbol{L}^N$. Since the $L^{\alpha,i}$ are independent, their jumps $\Delta L^{\alpha,i}\neq 0$ never occur at the same time and 
$\boldsymbol{\nu}$ and $\boldsymbol{\cN}^N$ admit the following representations:
\begin{align}\label{JumpDecomposition}
&\boldsymbol{\nu}( d \boldsymbol{v})=\sum_{i=1}^N\delta_{\{0\}}( d  v_1)\cdots\delta_{\{0\}}( d  v_{i-1})\nu( d  v_i)\delta_{\{0\}}( d  v_{i+1})\cdots\delta_{\{0\}}( d  v_N),\\
&\boldsymbol{\cN}^N(ds, d \boldsymbol{v})=\sum_{i=1}^N\delta_{\{0\}}( d  v_1)\cdots\delta_{\{0\}}( d  v_{i-1}) \cN^i(ds, d  v_i)\delta_{\{0\}}( d  v_{i+1})\cdots\delta_{\{0\}}( d  v_N),\nonumber
\end{align}
for $\delta_{\{0\}}$ the Dirac measure in $0$. In particular, \textcolor{Dgreen}{for any $i$, }since $\xi^i_t(s,0)(z)=0$ 
\textcolor{Dgreen}{ and since the measure ${\widetilde{\boldsymbol\cN}^N}$ only supports one jump at \textcolor{black}{a given}  time, we have:
\begin{align*}
&\int_{\textcolor{Dgreen}{\mR^{Nd}_0}} {\xi}^i_t(s,v_i)(z){\widetilde{\boldsymbol\cN}^N}( d  s, d  \boldsymbol{v})\\
&=\sum_{j=1,j\neq i}^N
\int_{\mR^{d}_0} {\xi}^i_t(s,0)(z)\ind_{\{v_j\neq 0\}}{\widetilde{\cN}^j}( d  s, d  v_{\textcolor{Dgreen}j})+\int_{\mR^{d}_0} {\xi}^i_t(s,v_i)(z){\widetilde{\cN}^i}( d  s, d  v_{\textcolor{Dgreen}i})\\
&=\int_{\mR^{d}_0} {\xi}^i_t(s,v_i)(z){\widetilde{\cN}^i}( d  s, d  v_{\textcolor{Dgreen}i}).
\end{align*}
}
As such, if we next introduce the predictable process
$$
\boldsymbol{\xi}^N_t(s, \boldsymbol{v})(z):=\frac1{N}\sum_{i=1}^N\xi^i_t(s,v_i)(z), \ \ 0\le s\le t\le T, z\in\mathbb R^d,
$$
then $\cM^{N}_{t}(z)$ can be written as
$$
\cM^{N}_{t}(z)
=\int^t_0\!\!\int_{\mR^{Nd}_0}\boldsymbol{\xi}^N_t(s, \boldsymbol{v})(z){\widetilde{\boldsymbol\cN}^N}( d  s, d  \boldsymbol{v}).
$$
Since $\mL^{\bbp}$ is a space of martingale $q$-type (recall that $\textcolor{Dgreen}{q\in (\alpha,2\wedge p_x\wedge p_v]}$),  
 applying Theorem \ref{thm:Hausenblas-BDG} to the stopped martingale
 $$
 \cM^{N}_{u,t}(z)=\int^u_0\!\!\int_{\mR^{Nd}_0}\boldsymbol{\xi}^N_t(s, \boldsymbol{v})(z){\widetilde{\boldsymbol\cN}^N}( d  s, d  \boldsymbol{v}),\quad u\in[0,t],
 $$
 we have: for any  $n\in\mN$,
\begin{align}\label{MartTemp}
\Big(\E\Big[\|\cM^N_t\|^{q^n}_{\bB^{\beta,1}_{\bbp;\bba}}\Big]\Big)^{1/q^n} 
	&\leq\sum_{j\geq 0}2^{\beta j}\Bigg(\E\Big[ \|\cR^\bba_j\cM^{N}_t\|^{q^n}_{\mL^\bbp}\Big]\Bigg)^{1/q^n}\le \sum_{j\geq 0}2^{\beta j}\Bigg(\E\Big[\sup_{0\le u\le t} \|\cR^\bba_j\cM^{N}_{u,t}\|^{q^n}_{\mL^\bbp}\Big]\Bigg)^{1/q^n}\nonumber\\
	&\quad\quad= \sum_{j\geq 0}2^{\beta j}\Bigg(\E\Big[ \sup_{0\le u\le t}\Big{\|}\int^u_0\!\!\int_{\mR^{Nd}_0}\cR^\bba_j\boldsymbol{\xi}^N_t(s, \boldsymbol{v})(z){\widetilde{\boldsymbol\cN}^N}( d  s, d  \boldsymbol{v})\Big{\|}^{q^n}_{\mL^\bbp}\Big]\Bigg)^{1/q^n}\nonumber\\
	&\lesssim\sum_{j\geq 0}2^{\beta j} 
	\Bigg(\sum_{k=1}^n\mathbb E\Bigg[\left(\int_0^t\!\!\!\int_{\mR^{Nd}_0}\|\cR^\bba_j\boldsymbol{\xi}^N_t(s, \boldsymbol{v})\|_{\mL^{\bbp}}^{q^k}\boldsymbol{\nu}( d \boldsymbol{v})\,  d  s\right)^{q^{n-k}}\Bigg]\Bigg)^{1/q^n}.
\end{align}
According to \eqref{JumpDecomposition}\textcolor{Dgreen}{,
$$
\int_{\mR^{Nd}_0}\|\cR^\bba_j\boldsymbol{\xi}^N_t(s, \boldsymbol{v})\|_{\mL^{\bbp}}^{q^k}\boldsymbol{\nu}( d \boldsymbol{v})
=\frac 1{N^{q^k}}\sum_{i=1}^N\int_{\textcolor{Dgreen}{\mR^{d}_0}}\|\cR^\bba_j \xi^i_t(s, v_i)\|_{\mL^{\bbp}}^{q^k}\nu( d  v_i),
$$
}
\textcolor{Dgreen}{and so}
\begin{align*}
\mathbb E\Bigg[\left(\int_0^t\!\!\!\int_{\mR^{Nd}_0}\|\cR^\bba_j\boldsymbol{\xi}^N_t(s, \boldsymbol{v})\|_{\mL^{\bbp}}^{q^k}\boldsymbol{\nu}( d \boldsymbol{v})\, d  s\right)^{q^{n-k}}\Bigg]
=\frac1{N^{q^n}}\mathbb E\Bigg[\left(\sum_{i=1}^N\int_0^t\!\!\!\int_{\mR^{d}_0}\|\cR^\bba_j\xi^i_t(s,v)\|_{\mL^{\bbp}}^{q^k}\nu( d  v)\, d  s\right)^{q^{n-k}}\Bigg].
\end{align*}
\textcolor{Dgreen}{Using successively the invariance by translation of $\|\cdot\|_{\mL^{\bbp}}$ and \textcolor{black}{Lemmas \ref{SemigroupEstimate}} and  \ref{Le304} further gives
\begin{align*} \|\cR^\bba_j\xi^i_t(s,v)\|_{\mL^{\bbp}}&=\left\|\textcolor{Dgreen}{\cR^\bba_j}P_{t-s}\Big(\delta^{(1)}_{(0,v)}\transOp_{s}\Phi_N\Big)(Z^{N,i}_{s-}-\cdot)\right\|_{\mL^{\bbp}}=\left\|\textcolor{Dgreen}{\cR^\bba_j}P_{t-s}\Big(\delta^{(1)}_{(0,v)}\transOp_{s}\Phi_N\Big)\right\|_{\mL^{\bbp}}\\
&\lesssim 2^{-(\beta+\epsilon)j}\|P_{t-s}\delta^{(1)}_{(0,v)}\transOp_{s}\Phi_N\|_{\bB^{\beta+\varepsilon,1}_{\bbp;\bba}}\lesssim 2^{-(\beta+\epsilon)j}
\big[\big(|v|^\theta (t-s)^{-\frac \theta\alpha}\big)\wedge 1\big]
\|\Gamma_s\phi_N\|_{\bB^{\beta+3\eps}_{\bbp;\bba}}.
\end{align*}}
Hence, {\black recalling (see e.g. \cite{HRZ23}) the isotropic $\alpha$-stable L\'evy measure  $\nu( d  v)=c_{d,\alpha}|v|^{-d-\alpha} d  v$,} \textcolor{Dgreen}{plugging the above into \eqref{MartTemp} yields}
\begin{align*}
	&\Big(\mE\left[\|\cM^{N}_{t}\|_{\bB^{\beta,1}_{\bbp;\bba}}^{q^n}\right]\Big)^{1/q^n}\\
	&\lesssim\frac 1{N}\sup_{s\in[0,t]}\|\Gamma_s\phi_N\|_{\bB^{\beta+3\eps}_{\bbp;\bba}}\sum_{j\ge 0}2^{-\eps j}\Bigg(\sum_{k=1}^n 
	\Big(N\int_0^t\!\!\!\int_{\mR^{d}_0}\frac{\big(|v|^\theta (t-s)^{-\frac {\theta-\eps}\alpha}\big)^{q^k}\wedge 1}{|v|^{d+\alpha}} d  v  d  s\Big)^{q^{n-k}}\Bigg)^{1/q^n}\\
	&\lesssim\frac {N^{1/q}}{N} \sup_{s\in[0,t]}\|\Gamma_s\phi_N\|_{\bB^{\beta+3\eps}_{\bbp;\bba}}\Bigg(\sum_{k=1}^n
	\left(\int_0^t(t-s)^{\frac{\eps}\theta-1}\int_{\mR^{d}_0}\frac{|v|^{\theta q^k}\wedge 1}{|v|^{d+\alpha}} d  v  d  s\right)^{q^{n-k}}\Bigg)^{1/q^n}.
\end{align*}
Since $q\in(\alpha,2]$, one can choose $\theta\in(\frac\alpha q,1)$ so that the remaining time-integrals on the right-hand side are finite for each $k=1,\cdots,n$. Thus, we obtain
\begin{align}\label{AC:pre02}
    \Big(\mE\left[\|\cM^{N}_{t}\|_{\bB^{\beta,1}_{\bbp;\bba}}^{q^n}\right]\Big)^{1/q^n}\lesssim  N^{1/q-1}\sup_{s\in[0,t]}\|\Gamma_s\phi_N\|_{\bB^{\beta+3\eps}_{\bbp;\bba}}.
\end{align}
Now for any $m\in\mN$, one can choose $n$ large enough so that $m\leq q^n$. Thus, by H\"older's inequality and \eqref{Es3},
$$
\|\cM^{N}_{t}\|_{L^m(\Omega;\bB^{\beta,1}_{\bbp;\bba})}\lesssim N^{1/q-1}\sup_{s\in[0,t]}\|\Gamma_s\phi_N\|_{\bB^{\beta+\textcolor{black}{3}\eps}_{\bbp;\bba}}
\lesssim N^{1/q-1+\zeta((\alpha+1)(\beta\textcolor{Dgreen}{+3\varepsilon})+
	\cA_{\bbb1,\textcolor{Dgreen}{\bbp}})}.
$$
The proof is complete{\black .}
\end{proof}

{\ZXC
\begin{proof}[Proof of Lemma \ref{lem:MN}]
We separately consider \textcolor{Dgreen}{the cases} $\alpha=2$ and $\alpha\in(1,2)${\color{black}, and, through H\"older's inequality, assume $m\ge 1$}.\\
\noindent
{\bf (Case $\alpha=2$).}
\textcolor{Dgreen}{Fix} $\eps>0$. By Lemma \ref{S3:2/1:lem1} \textcolor{Dgreen}{(with $\beta=0$, $\beta_1=\beta_b$ , $\bbp=\boldsymbol\infty$ and $\bbp_2=\bbp_b$)}, we have, 
\begin{align}\label{AM1}
\|b_t*\cM^N_t\|_{L^m(\Omega;\mL^{{\boldsymbol{\infty}}})}\lesssim N^{-\frac12+\modulateorder(\cA_{\bbp_b,\boldsymbol{\infty}}-3\beta_b+\eps)}\|b_t\|_{\bB^{\beta_b}_{\bbp_b;\bba}}.
\end{align}
In particular, if $\zeta\textcolor{Dgreen}{\le}1/(2(\cA_{\bbp_b,\infty}-3\beta_b\textcolor{Dgreen}{+\eps}))$
, then\comm{JF 25/03: The point of this estimate is only for the control of moments in the case $\bbp_0=\bbb1$; it will be relevant to mention it only at this place.}
\begin{align}\label{AM2}
\sup_{t\in[0,T_0]}\sup_N\|b_t*\cM^N_t\|_{L^m(\Omega;\mL^{\textcolor{Dgreen}{\boldsymbol{\infty}}})}\lesssim1.
\end{align}
If $\bbp_0\textcolor{Dgreen}{> \bb2}$\comm{S. and JF 23/03: The case $\bbp_0=\bb2$ is not anymore relevant.}, then by \eqref{S4:08:46} with $\beta=0$\textcolor{Dgreen}{, $\beta_2=\textcolor{black}{\eps/3}$} and $\bbp=\bbp_0$, we have
$$
\|\cM^N_t\|_{L^m(\Omega;\bB^{0,1}_{\bbp_0;\bba})} \lesssim N^{-\frac12+\modulateorder(\textcolor{black}{\eps}+\cA_{\bbb1,\bbp_0})}.
$$
Hence,
by definition of \textcolor{Dgreen}{the norm $\|\cdot\|_{\mS^{\beta}_t(b)}$} and \eqref{AM1}, we have
\begin{align}
\left\|\cM^N_t\right\|_{L^m(\Omega;\mS^{\beta}_t(b))} 
&\textcolor{Dgreen}{=} t^{\frac{\beta-\beta_0}{\alpha}}\|\cM^N_t\|_{L^m(\Omega;\bB^{0,1}_{\bbp_0;\bba})}+t^{\frac{\beta+\Gap}{\alpha}}\|b_t*\cM^N_t\|_{L^m(\Omega;\mL^{\textcolor{Dgreen}{\boldsymbol{\infty}}})}\label{AM5}\\
 &\lesssim N^{-\frac12+\modulateorder(\textcolor{black}{\eps}+\cA_{\bbb1,\bbp_0})}+N^{-\frac12+\modulateorder(\cA_{\bbp_b,\textcolor{Dgreen}{\boldsymbol{\infty}}}-3\beta_b+\eps)}\no\\
& \lesssim N^{\textcolor{Dgreen}{-\frac12}+\modulateorder((\cA_{\bbp_b,{\textcolor{Dgreen}{\boldsymbol{\infty}}}}-3\beta_b)\vee\cA_{\bbb1,\bbp_0}+\eps)}.\no
\end{align}
If $\bbp_0=\textcolor{Dgreen}{\bbb1}$, then, by \eqref{08:06} \textcolor{Dgreen}{(for $\beta=0$, \textcolor{black}{$\beta_1=\varepsilon/3$}, $\bbp=\bbb1$)}, 
\begin{align*}
\|\cM^N_t\|_{L^m(\Omega;\bB^{0,1}_{\bbb1;\bba})}
&\lesssim\left[1+\sup_{s\in[0,T]}(\mE|Z^{N,1}_s|_\bba^{\textcolor{Dgreen}{\ell m}})^{1/m}\right]N^{\modulateorder(\eps+{\ZXC  \cA_{\bbb1, {\bf 2}}})-\frac12}\\
&\!\!\!\!\!\!\!\!\!\stackrel{\eqref{AQ3}, \eqref{AM2}}\lesssim\left[1+(\mE|Z^{N,1}_0|_\bba^{\textcolor{Dgreen}{\ell m}})^{1/m}\right]N^{\modulateorder(\eps+{\ZXC  \cA_{\bbb1, {\bf 2}}})-\frac12}.
\end{align*}
Substituting this into \eqref{AM5} and by \eqref{AM1}, we obtain 
\begin{align*} 
\|\cM^N_t\|_{L^m(\Omega;{\color{black}\mS^{\beta}_t(b)})}
&\lesssim \left[1+(\mE|Z^{N,1}_0|_\bba^{\textcolor{Dgreen}{\ell m}})^{1/m}\right]N^{\modulateorder(\eps+{\ZXC  \cA_{\bbb1, \bb2}})-\frac12}
+N^{-\frac12+\modulateorder(\cA_{\bbp_b,\boldsymbol{\infty}}-3\beta_b+\eps)}\\
&\lesssim \left[1+(\mE|Z^{N,1}_0|_\bba^{{\ell m}})^{1/m}\right]N^{\modulateorder((\cA_{\bbp_b,\boldsymbol{\infty}}-3\beta_b)\vee \cA_{\bbb1, {\bf 2}}+\eps)-\frac12}.
\end{align*}
{\color{black}Combining the cases $\bbp_0>\boldsymbol{2}$ and $\bbp_0=\bbb1$ yields, for $\alpha=2$,
\begin{equation}\label{FinalMartingaleEstim-1}
\|\cM^N_t\|_{L^m(\Omega;\mS^{\beta}_t(b))}\lesssim\left[1+(\mE|Z^{N,1}_0|_\bba^{{\ell m}})^{1/m}\right]
N^{-\frac12+\modulateorder((\cA_{\bbb1, \bbp_0\vee {\bf 2}}\vee(\cA_{\bbp_b,\boldsymbol{\infty}}-3\beta_b)+\eps)}.
\end{equation}
}
{\bf (Case $\alpha\in(\textcolor{black}{1},2)$).} \textcolor{black}{By} \eqref{Con}, we have
\begin{align*}
\left\|\cM^N_t\right\|_{L^m(\Omega;\mS^{\beta}_t(b))}&\textcolor{Dgreen}{=t^{\frac{\beta-\beta_0}{\alpha}}\|\cM^N_t\|_{L^m(\Omega;\bB^{0,1}_{\bbp_0;\bba})}+t^{\frac{\beta+\Gap}{\alpha}}\|b_t*\cM^N_t\|_{L^m(\Omega;\mL^{\textcolor{Dgreen}{\boldsymbol{\infty}}})}} \\
&\lesssim t^{\frac{\beta-\beta_0}{\alpha}}\|\cM^N_t\|_{L^m(\Omega;\bB^{0,1}_{\bbp_0;\bba})}+t^{\frac{\beta+\Gap}{\alpha}}\|\cM^N_t\|_{L^m(\Omega;\bB^{-\beta_b,1}_{\bbp_b';\bba})}.
\end{align*}
Since $p_{x,0}\wedge p_{v,0}>\alpha$ and $\bbp_0\leq\bbp_b'$, one can choose $(\beta,\bbp)=(0,\bbp_0)$ and $(-\beta_b,\bbp_b')$ separately in \textcolor{Dgreen}{Estimate} \eqref{ZH1} so that \textcolor{Dgreen}{ the exponent $q$ dominating the rate in $N$ is given by} $q=p_{x,0}\wedge p_{v,0}\wedge 2$\textcolor{Dgreen}{, and}
\begin{align}\label{FinalMartingaleEstim-2}
\left\|\cM^N_t\right\|_{L^m(\Omega;\mS^{\beta}_t(b))} 
&\lesssim N^{\frac1q-1+\zeta(\cA_{\textcolor{Dgreen}{\bbb1},\bbp_0}+\eps)}+N^{\frac1q-1+\zeta(\cA_{\textcolor{Dgreen}{\bbb1},\bbp'_b}-(1+\alpha)\beta_b+\eps)}\\
&\lesssim N^{\frac1q-1+\zeta(\cA_{\textcolor{Dgreen}{\bbb1},\bbp_0}\vee(\cA_{\bbp_b,\textcolor{Dgreen}{\boldsymbol{\infty}}}-(1+\alpha)\beta_b)+\eps)}.
\end{align}
\textcolor{Dgreen}{This} complete\textcolor{Dgreen}{s} the proof.
 \end{proof}
 
 }

{\color{black}
\section{Applications to Coulomb and Riesz type interaction kernel (and a remark on the non-degenerate case)}
\label{sec:4}
In this section, we apply Theorems \ref{S1:main01} and \ref{main01} - and derive the best achievable convergence rates - to the case of the following Riesz-type singular interaction kernel:
$$
b_t(x,v)=-\gamma K_M(x,v)=-\gamma\nabla_x |x|^{s-d}\chi(M^{-1}v),\quad s\in(0,d), \ d\ge 3.
$$
The function $\chi$ defines here a non-negative compactly support $\mathcal C^\infty(\mathbb R^d)$ cut-off function satisfying  $\chi(v) = 1$ for all $|v| \le 1$ and $M\in(0,\infty)$ a re-scaling factor (which at the formal limit $M\rightarrow \infty$ reduces the cut-off to the unity). The force field $-\gamma\nabla_x|x|^{s-d}$ models a Riesz-type (or Coulombic-type) interaction kernel. The factor $\gamma$ models an intensity or re-normalizing parameter whose sign defines the attractive or repulsive behavior of the kernel (we indistinguishably consider both cases, assuming only $\gamma\neq 0$ hereafter). The exponent $s$ measures the strength of the interaction; larger is $s$, milder is the singularity of the kernel. The case $s=2$ specifically corresponds to a Coulomb force,  the range $s\in(2,d)$ refers to as a sub-critical regime  and $s\in(0,2)$ as a super-critical regime (see e.g. \cite{Serfaty-20}, \cite{Serfaty-24} for a more detailed discussion on the underlying physics and interests related to these regimes). The corresponding McKean-Vlasov SDE
\begin{align}\label{VPFP-McKV}
\begin{cases}
		\begin{aligned}
			X_t &= \xi^1 + \int_0^t V_s\, ds, \\
			V_t &= \xi^2 - \gamma\int_0^t \big(K_M * \mu_s(Z_s))\mathrm{d}s + L_t^{\alpha},\quad Z_t=(X_t,V_t), \  \mu_t=\cL(Z_t),
		\end{aligned}
\end{cases}
\end{align}
provides a probabilistic interpretation of the fractional Vlasov–Poisson–Fokker–Planck (VPFP) like 
equation:
\begin{align}\label{VPFP}
	\partial_t f = \Delta^{\alpha/2}_v f - v \cdot \nabla_x f + \gamma\mathrm{div}_v\big((K_M * f) f\big).
\end{align}
The case $s=2$ without cut-off (i.e. at the formal limit $M\rightarrow \infty$) particularly corresponds to the fractional VPFP equation ($K_\infty* f$ describing then an electric force field with charge $\int f_t(x,v)\,dv$ defining the non-linearity in \eqref{VPFP} through $-\div_x K_\infty*f(x)=\int f(x,v)\,dv$). 

The moderately interacting particle system for \eqref{VPFP-McKV} is given by
\begin{align}\label{VPFP-Particle}
	\begin{cases}
		\begin{aligned}
			X_t^{N,i} &= \xi^1_i + \int_0^t V_s^{N,i} \,ds, \quad i = 1,2,\dots,N, \\
			V_t^{N,i} &= \xi^2_i - \gamma\int_0^t \big(K_M * \Gamma_t \phi_N * \mu_s^N\big)(Z_s^{N,i})\,ds + L_t^{\alpha,i},\\
            Z_t^{N,i} &= (X_t^{N,i}, V_t^{N,i}),\, \mu^N_t=\frac1{N}\sum_{j=1}^N\delta_{\{Z^{N,j}_t\}}.
		\end{aligned}
	\end{cases}
\end{align}
The initial states are hereafter assumed to be independent copies of $(X_0,V_0)$ and to satisfy the moments condition $\mathbb E[|(X_0,V_0)|_\bba^m]<\infty$, $\forall m\in\mathbb N$.

For the sake of clarity, we also assume, from now on, that $\mu_0\in\mathcal P(\mathbb R^{2d})\cap \bB^{0^{-},1}_{\bbp_0;\bba}$ (i.e. $\beta_0\simeq 0$)\footnote{Recalling the embedding \eqref{AB2}, up to a slight shift in $\bbp_0$ this amounts to assuming $\mu_0$ admits a density $u_0\in \mL^{\bbp_0}$.} and the regime $\bbp_0>\boldsymbol{\alpha}$ (as pointed out in Remark \ref{RieszRestriction} below, the case $(\bbp_0,\alpha)=(\bbb1,2)$, is too restrictive to ensure well-posedness of \eqref{VPFP-McKV} for a singularity range $s\in(0,d)$). In this setting, by extending the proof arguments of \cite{HRZ23}, Section 5, the weak wellposedness of \eqref{VPFP-McKV} holds, for any $s\in(1,d\wedge \frac{d}{\alpha}+1)$ and appropriate bounds for $\bbp_0$, see ${\bf (H_{VFPw})}$ below.\footnote{We shall highlight that the presence of the cut-off $\chi$ avoids that the condition $1/\bbp_{v,0}+1/\bbp_{v,b}\ge 1$ in {\bf (H)} restrict the integrability range in the velocity component.} The weak and moderate propagation of chaos in Theorems \ref{S1:main01} and \ref{main01} then hold and a common optimal convergence rate (see \eqref{RieszStableOptimalRate}) can be achieved for particular thresholds for $\bbp_0$. This rate notably expands, to stable-driven kinetic dynamics, the ones established in the seminal paper \cite{OlRiTo-21} and the recent paper \cite{CGH-25} for first-order (non-degenerate) moderately interacting particle systems.

\br For the sake of brevity, we purposely left aside the question of the strong wellposedness of \eqref{VPFP-McKV} and the best achievable rate for the pathwise propagation of chaos in Theorem \ref{S1:main01}. These particular applications of our results can nevertheless be obtained by identifying the characteristic mixed anisotropic Besov space related to $K_M$ (following again \cite{HRZ23}) and under mild additional constraints on $\bbp_0$ and $\alpha$, the optimal rate \eqref{RieszBrownianOptimalRate} is expected to remain valid. 
\er
While propagation of chaos for second-order particle system with singular interaction kernels has been steadily growing (see e.g. \cite{JW16}, \cite{Serfaty-20}, \cite{BJS-25} and references therein), mean-field for first-order (non-degenerate) particle systems with singular kernels remains a center of attention and has made tremendous advances over the past years. Propagation of chaos for particle systems with Riesz-type interaction kernels (including Biot-Savart and Keller-Segel kernels) have been remarkably established, through modulate energy controls in \cite{Serfaty-20} (in a very broad deterministic setting) or through entropy methods in \cite{JW18}, \cite{BJW-23}. Closer to our {\it moderately interacting} setting, quantitative {\it moderate} propagation of chaos for (non-degenerate) moderately interacting particle systems for Riesz kernels in the range $s\in(1,d)$ have been broadly established in \cite{OlRiTo-21}. In the sub-critical case $s\in(2,d)$, the subsequent paper \cite{CHJ-26} addressed a refined fluctuation analysis of the dynamics and \cite{CGH-25} established weak (entropic-based) propagation of chaos result. Although further comparison of our results with recent weak propagation of chaos for non-mollified particle approximation related VFKP equations may be drawn - we can cite again to \cite{BJS-25} and reference therein for such results -, this would, as discussed in Remark \ref{RK_RATES}, emphasized our own result \eqref{AA200} as yet-to-be improved. Rather comparison with the papers \cite{OlRiTo-21} and \cite{CGH-25} enable to highlight the impact particular kinetic and stable scalings at the core of this paper.

		\paragraph{{\bf Weak wellposedness of \eqref{VPFP-McKV}}.} As established in Lemma~A.3 and Proposition~A.4 in \cite{HRZ23} (see also 
		Lemma~5.1 therein for the special case $s=2$), Riesz-type interaction kernels identify as anisotropic Besov functions with, provided $s\in(1,d)$, 
		\begin{equation}\label{RieszAnistropicBesov}
		\nabla|x|^{s-d} \in \bB^{(\alpha+1)(\frac{d}{p_{x,b}} - d +s- 1)}_{(p_{x,b},\infty);\bba}\quad \text{for any $p_{x,b} \in(\tfrac{d}{d-s+1}, \infty]$.} 
		\end{equation}

		Since $v \mapsto \chi(M^{-1}v)$ is smooth and compactly supported, the cut-off in the velocity argument gives $L^p(\mathbb R^d)$- integrability at all orders, and, so
		for any $p_{v,b} \in [1,\infty]$ - and $p_{x,b}$ as above - 
		\begin{equation}\label{RieszAnistropicBesov-bis}
		-\gamma K_M\in \bB^{(\alpha+1)(\frac{d}{p_{x,b}} - d +s- 1)}_{(p_{x,b},p_{v,b});\bba}.
		\end{equation}
		The (negative) regularity index of $K_M$ being given by 
		\begin{equation}\label{BesovIndexRiesz}
		\beta_b=(\alpha+1)(\tfrac d{p_{x,b}}-d+s-1),
		\end{equation}
		the coefficient $\Gap$ in \eqref{DEF_GAP} (in the case $\beta_0\simeq 0$) writes as
		\begin{align}\label{GapKineticApplied}
		\Lambda&=\bba\cdot\tfrac{d}{\bbp_0}+\tfrac{d}{p_{v,b}}+\tfrac{(1+\alpha)d}{p_{x,b}}-\beta_b-(2+\alpha)d=\bba\cdot\tfrac{d}{\bbp_0}+\tfrac{d}{p_{v,b}}-(\alpha+1)(s-1)-d.  
		\end{align}
 where $\tfrac{d}{p_{v,b}}-d=-\tfrac{d}{p'_{v,b}}$ and $p'_{v,b}$ can be arbitrary chosen in $[p_{v,0},\infty]$.
		According to {\bf (H)}, the weak wellposedness of \eqref{VPFP-McKV} - on a finite time interval $[0,T_0]$ - is hence ensured under the conditions 
		\begin{equation*}
		{\bf (H_{VFPw})}\qquad 
		\begin{aligned}
		&\Lambda\in \big(0,\alpha-1) \Leftrightarrow \tfrac{d}{p_{x,0}}+\tfrac{1}{\alpha+1}\Big(\tfrac{d}{p_{v,0}}-\tfrac{d}{p'_{v,b}}\Big)+1-s\in \Big(0,\tfrac{\alpha-1}{\alpha+1}\Big),\\
		&  \tfrac{1}{\bbp_0}+\tfrac{1}{\bbp_b}\ge \bbb1 \Rightarrow 1\le p_{x,0}\le p'_{x,b}<(\tfrac{d}{d-s+1})'=\tfrac{d}{s-1},
		\end{aligned}
		\end{equation*}
		and, in the regime, $\bbp_0> \boldsymbol{\alpha}$, under the additional restriction for $s$: 
		\begin{equation}\label{SingularityRange}
			\tfrac{d}{s-1}>\alpha\Rightarrow 1<s<d\wedge \big(\tfrac{d}{\alpha}+1\big).
		\end{equation}
			This setting covers sub-critical regimes (up to $d\wedge (\frac{d}\alpha+1)$), the special Coulomb case $s=2$ (for e.g. $p_{v,0}=$ that $p_{x,0}$ and $p'_{v,b}$ satisfies $\frac d{p_{x,0}}-\frac{d}{p'_{v,b}}\in \big(1,1+\frac{\alpha-1}{\alpha+1}\big)$
		and the supercritical regime $s\in(1,2)$ (e.g. for $p'_{v,b}=p_{v,0}$ and $p_{x,0}$ large enough).  
		\begin{remark}\label{RieszRestriction} For $(\bbp_0,\alpha)=(\bbb1,2)$, ${\bf (H_{VFPw})}$ is not enough to ensure the wellposedness of \eqref{VPFP-McKV} as this condition impose to restrict the exponent $s$ to
			$$
			d+\tfrac{d}{(\alpha+1)p_{v,b}}-s+1\in(0,\tfrac{1}3)\Leftrightarrow s\in\big(d+\tfrac{d}{(\alpha+1)p_{v,b}}+\tfrac{2}3,d+\tfrac{d}{(\alpha+1)p_{v,b}}+1\big)
			$$
			which sits outside the range $(1,d)$. 
		\end{remark}
	 \paragraph{{\bf Weak and moderate propagation of chaos for \eqref{VPFP-Particle}.}} Under the setting ${\bf (H_{VFPw})}$, the moderate propagation of chaos in Theorem \ref{main01} and the weak propagation of chaos stated in Theorem \ref{S1:main01} apply. Notably, maximizing along the mollification parameter $\zeta$ (as in \eqref{GeneralOptimalRate} in Remark \ref{RK_RATES}), the quantities $\sup_{t\in[0,T_0]}\|u^N_t-u_t\|_{L^m(\Omega;\mS^\beta_t(b))}$ and
	 $\|\cL(Z_t)
	 -\cL(Z^{N,1}_t)\|_{\rm var}$ are bounded by the principal rate
	 \begin{equation}\label{ModerateRateRiesz}
	 N^{-\beta \zeta} + N^{\mstable_\alpha-1 +\theta_\alpha\zeta + \eps}= N^{-\frac{\beta}{\beta+\theta_\alpha}(1-\mstable_\alpha-\epsilon)},
	 \end{equation}
		where (recalling $\Gap$ in \eqref{GapKineticApplied})
        \begin{align*} 
			\mstable_\alpha=\tfrac1{2\wedge p_{x,0}\wedge p_{v,0}},
		\end{align*}
        $\beta\in(0,\bar{\beta}_\alpha)$ with
		\begin{equation*}
			\bar{\beta}_\alpha=
				1-(1+\alpha)\big(\tfrac{d}{p_{x,0}}-s+1\big)-\big(\tfrac{d}{p_{v,0}}-\tfrac{d}{p'_{v,b}}\big),\ \text{for}\quad\alpha=2,
		\end{equation*}
        and for $\alpha\in(1,2)$,
		\begin{equation*}
			\bar{\beta}_\alpha=\big[\alpha-1-(\alpha+1)\Big(\tfrac{d}{p_{x,0}}-s+1\big)-\big(\tfrac{d}{p_{v,0}}-\tfrac{d}{p'_{v,b}}\Big)\big]\wedge \Big[\tfrac1{2}\Big(\alpha-(\alpha+1)\big(\tfrac{d}{p_{x,0}}-s+1\big)-\big(\tfrac{d}{p_{v,0}}-\tfrac{d}{p'_{v,b}}\big)\Big)\Big],
		\end{equation*}
		 and eventually, recalling $\beta_b$ in \eqref{BesovIndexRiesz}, and, that from ${\bf (H_{VFPw})}$ $1/p'_{v,0}\le 1/p_{v,b}$ and $1/p'_{x,0}\le 1/p_{x,b}$, where
		\begin{align*} 
		\theta_\alpha&=\big[\tfrac{(1+\alpha)d}{p'_{x,0}}+\tfrac{d}{p'_{v,0}}\big]\vee\big[\tfrac{d}{p_{v,b}}+(1+\alpha)\tfrac{d}{p_{x,b}}-(\alpha+1)\beta_b\big]=\tfrac{d}{p_{v,b}}+(1+\alpha)\tfrac{d}{p_{x,b}}-(\alpha+1)\beta_b\\
		&=
				\big[\tfrac{(1+\alpha)d}{p'_{x,0}}+\tfrac{d}{p'_{v,0}}\big]\vee
				\big[\tfrac{d}{p_{v,b}}-\alpha(1+\alpha)\tfrac{d}{p_{x,b}}+(\alpha+1)^2(d-s+1)\big]\\
				&=\tfrac{d}{p_{v,b}}-\alpha(1+\alpha)\tfrac{d}{p_{x,b}}+(\alpha+1)^2(d-s+1).
		\end{align*}
The optimal rate in \eqref{ModerateRateRiesz} is obtained by taking  $\beta$ up to its admissible limit $\bar{\beta}_\alpha$ (i.e. $\beta=(\bar\beta_\alpha)^-$), maximizing the latter while minimizing $\mstable_\alpha$ and $\theta_\alpha$. This is achieved by taking $p'_{v,b}=p_{v,0}$, choosing $p_{x,b}$ close to its minimal value $\frac{d}{d-s+1}$ (see \eqref{RieszAnistropicBesov}), 
and $p_{x,0}$ close to its largest admissible value ($p_{x,0}\simeq \frac{d}{s-1}$). This yields, for any $\alpha\in(1,2]$,
$$
\bar{\beta}_\alpha\simeq (\alpha-1)\wedge\tfrac{\alpha}2= \alpha-1,\quad \mstable_\alpha\simeq \tfrac{1}{2\wedge (\tfrac d{s-1})\wedge p_{v,0}},\quad 
\theta_\alpha\simeq\tfrac{d}{p_{v,b}}+(\alpha+1)(d-s+1).
$$
		Finally, optimizing the remaining indices $p_{v,0}$ and $p_{v,b}$ to achieve the final balance between $\mstable_\alpha$ and $\theta_\alpha$ follows by taking $p_{v,0}\simeq \alpha$, which implies $ p_{v,b}\simeq \frac{\alpha}{\alpha-1}$ (and recalling from \eqref{SingularityRange} that $s<d\wedge(\frac{d}\alpha+1)$), we then have
		$$
		\bar{\beta}_\alpha\simeq \alpha-1,\quad 
		\mstable_\alpha\simeq \tfrac{1}{(\tfrac{d}{s-1})\wedge \alpha}=\tfrac1{\alpha},\quad \theta_\alpha\simeq \tfrac{(\alpha-1)d}{\alpha}+(\alpha+1)(d-s+1),
		$$
		For all $s\in\big(1,d\wedge(\frac{d}\alpha+1)\big)$,
		this yields to the optimal convergence rate for \eqref{ModerateRateRiesz}: $N^{-\varrho}$ where  
		\begin{equation}\label{RieszStableOptimalRate}
		\varrho=\Big(\frac{(\alpha-1)}{(\alpha-1)+ (\alpha-1)d/\alpha+(\alpha+1)(d-s+1)}(1-\frac{1}{\alpha})\Big)^-,
		\end{equation}
	The optimal exponent \eqref{RieszStableOptimalRate} carries the maximal value possible for $1-\mstable_\alpha$ ($1/2$ for the Gaussian regime, $1-1/\alpha$ for the pure jump regime) and  the maximal value possible for the parameter $\beta$, which, recalling Remark \ref{RK_RATES} measures the smoothness of $\mathcal L(X_t,V_t)$. The expression $d-s+1$ reflects the interaction range of the Riesz force $\nabla |x|^{s-d}=x|x|^{s-d-2}$ and the sensibility to the exponent $s$. The factor $(\alpha+1)$ particularly reflects the characteristic kinetic scaling of the model (and the fact that the interaction kernel acts on the degenerate position argument). As will be discussed in a few lines, this provides a kinetic analogue of the rate established in \cite{OlRiTo-21} and \cite{CGH-25}. Eventually, the remaining element $(\alpha-1)d/\alpha$ may be viewed as the impact of the singular kernel in the velocity direction and also present a particular trade-off between $\theta_\alpha$ and $\mstable_\alpha$; minimizing $\frac{d}{p_{v,b}}$ in $\theta_\alpha$ requires maximizing $p_{v,b}$ and as $p_{v,0}\le p'_{v,b}$, increase $\mstable_\alpha$.  
		
	\paragraph{{\bf Comparison with first-order (non-degenerate) systems \eqref{RieszBrownianOptimalRate}}} In \cite{OlRiTo-21}, the authors established a moderate propagation of chaos for the moderately interacting particle system 
	$$
	V^{N,i}_t = \xi_i + \int_0^t F\Big(K^N * \overline{\mu}^N_s(V^{N,i}_s)\Big) d  s + \sqrt{2}W^{i}_t,\quad \overline{\mu}^N_t=\frac1{N}\sum_{i=1}^N\delta_{\{V^{N,i}_t\}},
	$$
    where $K^N(v):=K*\psi_N(v)$ with some mollifier $\psi_N(v):=N^{d\zeta}\psi(N^{\zeta}v)$ and $\zeta>0$,
	towards the first order non-degenerate (autonomous) dynamic: 
		$$
		d  V_t = K * \overline{\mu}_t(V_t) d  t + \sqrt{2}W_t,\quad \overline{\mu}_t=\mathcal L(V_t), \,0\le t\le T.
		$$
	The field $F$ define here a cut-off function, flat-off outside the ball $\mathbb B(0,\sup_t\|K*\mu_t\|_{L^\infty})$ and $K=K(v)$ is a general interaction kernel satisfying mild integrability and mild assumption. For the case of the Riesz-type force $K(v)=-\gamma\nabla |v|^{s-d}$ (with $\gamma\neq 0$,  $s\in(1,d)$ and $d\ge 1$), the authors notably established quantitative moderate propagation of chaos results for  
	$$
	\|\sup_{t\in[0,T]}\|K^N*\bar{\mu}^N_t-\overline{\mu}_t\|_{L^1\cap L^r}\|_{L^m(\Omega)}
	$$
	with different optimal convergence rates (see \cite{OlRiTo-21}, Section 5.2):
	
	$\bullet$ for $s\in(2,d)$ and $r\ge \frac{d}{s-1}$,  $N^{-\bar{\varrho}}$, $\bar{\varrho}=\Big(\frac{1}{1+d/2+(d-s+2-d/2)\vee 0}\times\frac1{2}\Big)^-$
	
	$\bullet$ for $s\in(1,2]$ and $r=\infty$, $N^{-\big(\frac1{(1+d)}\times\frac1{2}\big)^-}$.
	
	In the sub-critical regime $s\in(1,d)$, using alternative entropy estimate (and for $F=1$), the authors in \cite{CGH-25} derived an alternative weak propagation of chaos result with an optimal rate  $N^{-\tilde{\varrho}}$, $\tilde{\varrho}=\big(\frac{1}{1+(d-s+1)}\times \frac1{2}\big)^-$. 
	
	For $\alpha=2$, $d\ge 3$ and $s\in(1,\frac{d}2+1)$, the rate \eqref{RieszStableOptimalRate}:
		\begin{equation}\label{RieszBrownianOptimalRate}
		N^{-\varrho},\quad\varrho= \Big( \frac{1}{1+d/2+3(d-s+1)}\times\frac{1}2\Big)^-,
		\end{equation}
	can be viewed as a consistent, up-scaled to second-order dynamics, version of the above rates. The coefficient $\frac1{2}$, characteristic of the Gaussian limit regime, appear as a common factor in all rates and the front multiplier reflects the singularity (specially the dimension-dependency) of the Riesz-type force. From first-order to second order dynamics, for the subcritical range $s\in(2,d)$, the interaction range $d-s+1$ exhibited in \cite{CGH-25} is increased by the kinetic factor $3$ ($=\alpha+1$ for $\alpha=2$) while, for the Coulombic and super-critical range $s\in(1,2)$, the bound $1+3(d-s+1)\ge 1+3d$ is consistent the component $1+d$ in \cite{OlRiTo-21}. While the rate \eqref{RieszBrownianOptimalRate} continuously spans super- to sub-critical regimes, the additional component $d/2$ in the front multiplier \eqref{RieszBrownianOptimalRate} ($d/\alpha$ in \eqref{RieszStableOptimalRate}) and the bound $s\le \frac{d}{2}+1$ are currently inherent to our current setting (notably to the conditions $\boldsymbol{\alpha}<\bbp_0$ and $p_{v,0}\le p'_{v,b}$) and should be improved in future works. 

\br\label{rem:NonDegenerateCase}
In the case where $b=b_t(v)$, the position component $X_t$ becomes auxiliary and the kinetic system \eqref{MV1} reduces to the first-order non-degenerate McKean–Vlasov SDE:
\begin{align}\label{AC:MVS}
dV_t = (b_t * \overline{\mu}_t)(V_t),dt + d\Levy^{\nalpha}_t,
\qquad \overline{\mu}_t = \mathcal L(V_t), \quad 0 \le t \le T .
\end{align}
Since $b$ is independent of $x$, the interaction kernel $b$, which in the general setting lies in the anisotropic Besov space $\bB^{\beta_b}_{(\infty,p_{v,b});\bba}$, simply belongs here to the usual isotropic Besov space $\bB^{\beta_b}_{p_b}$ (see \cite[Proposition A.4]{HRZ23}).

Taking $p_{x,b}=\infty$ and $p_{v,b}=p_b$ (so that $\bba\cdot \tfrac{d}{\bbp_b}=\tfrac{d}{p_b}$),
 the weak well-posedness condition {\bf (H)} simplifies as follows:
\begin{enumerate}[{\bf (H$_{non})$}]
	\item  \textit{$b=b_t(v)\in L^\infty(\mR_+;\bB^{\beta_b{\color{black},\infty}}_{p_b})$ and $ \overline{\mu}_0=\overline{\mu}_0(\dif v)\in \mathcal P(\R^{{\color{black}d}}) \cap \bB^{\beta_0{\color{black},\infty}}_{{\color{black}p_0}}$ with $p_0\in (\alpha,\infty]$, 
		$p_b\in[1,\infty{\color{black}]}$ satisfying $\frac1{p_0}+\frac1{p_b}\ge1$, and $\beta_b\le0$ and $\beta_0\in(-1,0)$ {are such that}
		\begin{align}
			0<{\Gap}:=\tfrac{d}{p_0}+\tfrac{d}{p_b}-d-\beta_0
			-\beta_b
			<\alpha-1.\label{DEF_GAP_auto}
	\end{align}}
\end{enumerate}
Under these assumptions, the weak well-posedness criterion of \cite{HRZ23} applies (with $p_{x,b}=\infty$). In particular, there exists a finite time horizon $T_0=T_0(\Theta)$ such that \eqref{AC:MVS} admits a unique weak solution on $[0,T_0]$ and, for all $t\in(0,T_0]$, the marginal law $\bar{\mu}_t,$ admits a smooth density $\overline{u}_t$ satisfying regularity estimates analogous to \eqref{S2:01},  
now formulated in isotropic Besov norms (see \cite[Theorem 1.6-(i)]{HRZ23}).
Furthermore, strong uniqueness holds provided that ${\Gap}$ in \eqref{DEF_GAP_auto} sits in the interval $(0,\frac3{2}\alpha-1)$. 

Likewise, the corresponding interacting particle system \eqref{S1:00} rewrites as
\begin{align}\label{AC:IPS}
	V^{N,i}_t = \xi_i + \int_0^t (b^N_s * \overline{\mu}^N_s)(V^{N,i}_s) d  s + \Levy^{\nalpha,i}_t,\quad \overline{\mu}^N_s=\frac1{N}\sum_{i=1}^N\delta_{\{V^{N,i}_t\}},
\end{align}
where the related mollified interaction simplifies into
\begin{equation*}
	b^N_t(v) := b_t * \transOp_t \phi_N(v) = b_t * \psi_N(v),
\end{equation*}
where $\psi_N(v) = \int_{\mR^d} \transOp_t \phi_N(x,v) d  x
	= N^{\zeta d} \int_{\mR^d} \phi(x, N^{\zeta} v) d  x
	=: N^{\zeta d} \psi(N^\zeta v)$ defines a smooth, compactly supported probability density function on $\mathbb R^d$.

In this framework, the moderate propagation of chaos between \eqref{AC:IPS} and \eqref{AC:MVS} is measured by the convergence of $\Psi_N*\bar\mu^N$ towards the limiting law $\overline{\mu}_t$ under the “isotropic and $\mathbb R^d$-reduced" version of the norm $\mS^\beta_t(b)$ (defined in \eqref{BBT}):
$$
\|f\|_{S^\beta_t(b))}:=(1\wedge t)^{\frac{\beta-\beta_0}{\alpha}}\|f\|_{\bB^{\beta_0,\infty}_{p_0}}+(1\wedge t)^{\frac{\beta+\Gap}{\alpha}}\|b_t*f\|_{L^\infty},\quad f=f(v)\in\bB^{0,1}_{p_0}\cap \bB^{-\beta_b,1}_{p'_b}.
$$
Replicating the starting proof arguments of Theorem \ref{main01}, as for \eqref{S1:P2}, one obtains a corresponding non-degenerate SPDE for $\overline{\cU}^N_t(v):=\Psi_N*\overline{\mu}^N_t(v)-\overline{\mu}_t(v)=\int \cU^N_t(x,v)\,dx$ and, as for \eqref{DIFF_SPDE}, one obtains the Duhamel formula
$$
\bar{\cU}^N_t(v)=\int P_t(u_0-u^N_0)(x,v) \,dx-\int_0^t \int P_{t-s}\div_v H^{N}_s(x,v)\,dx d s-\int \int_0^t P_{t-s} d  M^N_s(x,v)\,dx.
$$
Due to the integration in $x$, the action of the semigroup $P_t$ reduces to the action of the semigroup of $L^\alpha$ and the dependency in the transported mollifier $\Gamma_t\phi_N(x,v)$ simplifies into $\Psi_N(v)$, especially in the martingale part \eqref{MART_FIELD_I}. While the two first elements of the above Duhamel decomposition can be estimated as in Lemmas \ref{098} and \ref{lem:HN}, the key improvement in the non-degenerate case emerge with Lemma \ref{lem:MN}: in contrast to the scaling in Lemma~\ref{08:18}, the absence of the factor $\Gamma_t$ yields
$\|\phi_\lambda\|_{\bB^{\beta}_{\bbp;\bba}} \lesssim \lambda^{\beta + \cA_{\bbb1,\bbp}}$ and significantly reduce the characteristic bounds in Theorems \ref{08:thm01} and \ref{thm:MartPart-Stable}, respectively from $3\beta_1$ to $\beta_1$ and $(1+\alpha)\beta$ to $\beta$.
   As a result, in the non-degenerate case, instead of \eqref{limitstable}, we have the parameters $m_\alpha$ and $\overline{\theta}_\alpha$ - corresponding to $\mstable_\alpha$ and $\theta_\alpha$ for $p_{x,b}=\infty$, $p_{x,0}=1$ and $(1+\alpha)\beta\rightarrow \beta_b$ in \eqref{limitstable}:
   \begin{align*}
	m_\alpha=\tfrac{1}{(p_0\wedge 2)\vee \alpha},
	\quad \overline{\theta}_\alpha={\cA_{\bbb1,(1,p_0\vee \alpha)}\vee \big[\cA_{(\infty,p_b),\boldsymbol{\infty}}-\beta_b\big]}=(d - \tfrac{d}{p_0}) \vee (-\beta_b + \tfrac{d}{p_b}),
\end{align*}
where the factor $(1+\alpha)\beta_{b}$ in $\theta_\alpha$ from \eqref{limitstable} is replaced by $\beta_b$.  

    Finally, we can expect the following result:
\begin{theorem}\label{AC:main01}
Suppose that {\bf (H$_{non}$)} holds and that either $p_0>\alpha$ or $(\alpha,p_0)=(1,2)$. Let $\{\xi_i\}_{i=1}^N$ be as in \eqref{AC:IPS}, a family of i.i.d. $\mR^{2d}$-valued random variables with common law $\overline{\mu}_0$ satisfying $\mathbb E[|V_0|^m]<\infty$ for all $m\in\mathbb N$. 
Then, for any $\beta,\zeta>0$ such that
\begin{align}\label{AC:BetaRange}
\beta < (\alpha - 1 - \Lambda) \wedge \tfrac{\alpha + \beta_0 - \Lambda}{2}, 
\qquad \zeta < (1-m_\alpha)/\overline{\theta}_\alpha, 
\end{align}
and for any $\eps>0$, $m>0$,
there exists a constant $C = C(\Theta,\zeta,\beta,\eps,m) > 0$ such that, for all $N \ge 1$,
\begin{align*}
\sup_{t \in [0,T_0]} \|\psi_N*\overline{\mu}^N_t - \overline{\mu}_t\|_{L^m(\Omega;S^\beta_t(b))}
\le C N^{-\beta\zeta} + C N^{-1+m_\alpha + \zeta \overline{\theta}_\alpha} + \eps.
\end{align*} 
\end{theorem}
Related weak and pathwise propagation of chaos between $V$ and $V^{N,1}$ can be derived, adapting the proof arguments of Theorem \ref{S1:main01}. This result applies to a variety of non-degenerate models, including the fractional two-dimensional vortex Navier–Stokes equations, the surface quasi-geostrophic equation, and the Keller–Segel model initially addressed in \cite{HRZ23}.

Due to space limitations, we will provide the full details in a forthcoming paper.
\er
}

 \appendix
\renewcommand{\thesection}{\Alph{section}.\arabic{section}}
\renewcommand{\theequation}{\thesection.\arabic{equation}}

\setcounter{section}{0}

\begin{appendices}

		\section{Weighted estimates}\label{sec:WeightedEstimates}
		\textcolor{Dgreen}{\textcolor{black}{In this section, we state and prove} some essential properties on the weighted $\mL^\bbp$-controls relative to the block operators $\cR^\bba_j$ defined in \eqref{Ph0}, for a particular family of weight functions $\omega_\ell$, specified in \eqref{Om1-bis} below. We also establish controls for the kinetic semigroup $P_t$ defined in \eqref{SCL11}. These estimates are notably meaningful to handle the particular setting $\bbp_0=\bbb1$ in the main Theorems \ref{S1:main01} and \ref{main01} (notably for Theorem \ref{08:thm01}), and some intermediate results in the subsequent appendices \ref{APP_BESOV} and \ref{CONV_EMP_SAMP_MU0}.  \\
	  The results presented in this section consist of slight extensions of previous statements from \cite{HZZZ22} and \cite{HRZ23} (and also references therein), but, for the sake of \textcolor{black}{completeness, we present here the full proofs}. }
		
		\textcolor{Dgreen}{	
			Consider the class of weight functions, given by: For $\ell\ge 0$,  
			\begin{align}\label{Om1-bis}
			 \omega_\ell(z):=\textcolor{Dgreen}{\big(}1+(1+|x|^2)^{\frac1{1+\alpha}}+|v|^2\textcolor{Dgreen}{\big)}^{\ell/2},\ \ z=(x,v)\in\mR^{2d}.
			\end{align}
			\textcolor{Dgreen}{This weight function is $C^\infty$ and satisfies the following three properties relative to derivatives, growth and integrability control and to finite difference:\comm{JF 27/03: There is a slight difference with the growth property stated after \eqref{Om1} - here we precise the bound of the growth, but not after \eqref{Om1} ... Recheck if this is important in the core of the article.}} 
			\begin{equation}\label{ClassWeight}
			|\nabla^k \omega_\ell(z)|=|(\nabla_x^k \omega_\ell(z),\nabla_v^k \omega_\ell(z))|\lesssim \omega_\ell(z),\ \ k\in\mathbb N, z\in\mathbb R^{2d}, 
			\end{equation}
			\begin{equation}\label{WeightGrowth}
				\omega_\ell(z+z')\leq (2^{\ell-1}\vee 1)(\omega_\ell(z)+\omega_\ell(z'))\textcolor{Dgreen}{\le 2(2^{\ell-1}\vee 1)\omega_\ell(z)\omega_\ell(z')},\ z,z'\in\mathbb R^{2d},
			\end{equation}
			and, for all $\bbp\in[1,\infty]^2$, such that $\ell>\mathcal \bba\cdot \tfrac d{\bbp}$,
				\begin{equation}\label{WeightIntegrability}
				\|(\omega_\ell)^{-1}\|_{\bbp}<\infty.
				\end{equation}
				}		
		\textcolor{Dgreen}{
		Given $\omega_\ell$ as above, we shall consider the quantities 
		$$
		\|(\cR^\bba_jf)\omega_\ell\|_{\bbp}, \ \ j\ge 0, \bbp \in[1,\infty]^2, f\in\cS'(\mathbb R^{2d}).
		$$
		These are naturally related to the weighted anisotropic Besov space:
		$$
		B^{s,q}_{\bbp;\bba}(\omega_\ell):=\left\{f\in \cS': \|f\|_{\mathbf{B}^{s,q}_{\bbp;\bba}(\omega_\ell)}
		:= \left(\sum_{j\geq0}\big(2^{ js}\| (\cR^\bba_{j}f)\omega_\ell\|_{\bbp}\big)^q\right)^{1/q}<\infty\right\},
		$$
		for $s\ge 0$, $\bbp\in[1,\infty]^2$ and $q\in[1,\infty]$. Theorem 2.7 in \cite{HZZZ22} previously established, for any $f\in\cS'$, the equivalence between the norms  $\| f\|_{\mathbf{B}^{s,q}_{\bbp;\bba}(\omega_\ell)}$ and $\|\omega_\ell f\|_{\mathbf{B}^{s,q}_{\bbp;\bba}}$ in the special case where $\bbp=(p,p)\in[1,\infty]^2$. This equivalence is stated for a broad class of weight functions
		, $\omega_\ell$ defined in \eqref{Om1-bis} being a particular instance of this class. }
		\textcolor{Dgreen}{The original proof of Theorem 2.7 in \cite{HZZZ22} 
			 relies on 
			 establishing, whenever $s>0$, that 
		 \begin{equation}\label{IntermediateWeightedNorm}
		 	\|f\|_{\bB^{s,q}_{\bbp;\bba}(\omega_\ell)} \asymp \|f\|_{\widetilde\bB^{s,q}_{\bbp;\bba}(\omega_\ell)}\asymp \|\omega_\ell f\|_{\bB^{s,q}_{\bbp;\bba}},
		 \end{equation}
		 where 
		 $$
		 \|f\|_{\widetilde\bB^{s,q}_{\bbp;\bba}(\omega_\ell)}:=\|\omega_\ell f\|_{\bbp}+\left(\int_{|h|_\bba\le 1}\Bigg(\frac{\|\omega_\ell \delta^{(s] +1)}_h f\|_{\bbp}}{|h|_\bba^s}\Bigg)^q\frac{ d  h}{|h|^{(2+\alpha)d}} \right)^{1/q},
		 $$
		 for $\delta^{(M)}_h$  the $M^{\text{\rm th}}$ difference operator defined in Section \ref{SEC_BESOV_BIS}. Essentially, the equivalence \eqref{IntermediateWeightedNorm} is obtained from the composition properties of $\delta^{(M)}_h$, suitable commutations 
		 with the block operators $\cR^\bba_j$ (see \cite{HZZZ22}, pp. 647), the properties \eqref{ClassWeight} and \eqref{WeightGrowth}, 
		 and finally the weighted Young inequality:  
		 \begin{equation}\label{WeightedYoung}
		 	\|(g*h)\omega_\ell\|_{\bbp}\lesssim \|g\omega_\ell\|_{\bbp_1}\|h\omega_\ell\|_{\bbp_2}, \ \ \frac 1{\bbp'}=\frac 1{\bbp_1}+\frac 1{\bbp_2}, \bbp,\bbp_1,\bbp_2\in[1,\infty]^2
		 \end{equation}
		 (for $\bbp'$ the conjugate \eqref{KineticConjugate} of $\bbp$).
		 The latter can be simply derived from 
		 Young's inequality, noticing that 
		 $$
		 g*h(z)\omega_\ell(z)\\
		 =
		 \int\frac{\omega_\ell(z)}{\omega_\ell(z-\overline{z})\omega_\ell(\overline{z})} \big(g\omega_\ell\big)(z-\overline{z})(h\omega_\ell)(\overline{z})\, d  \overline{z},
		 $$
		 and \textcolor{black}{observing} that, by \eqref{WeightGrowth}, $\frac{\omega_\ell(z)}{\omega_\ell(z-\overline{z})\omega_\ell(\overline{z})}\lesssim 1$, for all $z,\overline{z}$.
		}
		
		\textcolor{Dgreen}{
		The restriction $\bbp=(p,p)$ in \cite{HZZZ22} 
		was only made for simplicity and replicating the initial proof steps in \cite{HZZZ22}, one may extend \eqref{IntermediateWeightedNorm} from a common integrability index to a mixed one. For the 
		present paper, we shall restrict this extension on establishing the following estimate.
		\bl\label{lem:WeightedNorm} For any $f\in\cS'(\mR^{2d})$ such that $\omega_\ell f\in \bB^{\beta}_{\bbp;\bba}$ with $\beta\ge 0$, $\bbp\in[1,\infty]^2$, we have: for any $j\ge 0$,
		\begin{equation}\label{WeightedBlock}
			\|\big(\cR^\bba_jf\big)\omega_\ell\|_{\bbp}\lesssim 2^{-\beta j}\Big(\|\omega_\ell f\|_{\bbp}+\|\omega_\ell f\|_{\bB^{\beta}_{\bbp;\bba}}\Big).
		\end{equation}
		In particular, for $\beta=0$, $\|\big(\cR^\bba_jf\big)\omega_\ell\|_{\bbp}\lesssim \|\omega_\ell f\|_{\bbp}$.
		\el}
     \begin{proof}\textcolor{Dgreen}{As in \cite{HZZZ22}, the proof essentially consists in establishing that
     		\begin{equation}\label{WeightedBlock-Intermediate}
     			\|(\cR^\bba_jf)\omega_\ell\|_{\bbp}\lesssim 2^{-\beta j}\Bigg(\|\omega_\ell f\|_{\bbp}+\sup_{|h|_\bba\le 1}\frac{\|(\delta^{([ \beta]+1)}_hf)\omega_\ell\|_{\bbp}}{|h|^\beta_\bba}\Bigg).
     		\end{equation}     		
     		Indeed, this estimate gives naturally the claim when $\beta=0$. For $\beta>0$, recalling the Leibniz rule: for any integer $m\ge 1$,
     			$$
     			\delta^{(m)}_h(f\omega_\ell)=\sum_{n=0}^m\frac{m!}{n!(m-n)!}\big(\delta^{(n)}_hf\big)\big(\delta^{(m-n)}_h\omega_\ell(\cdot+nh)\big) 
     		$$
     		the quantity $\|(\delta^{[ \beta] +1}_hf)\omega_\ell\|_{\bbp}$ is readily  bounded by $\|\delta^{([ \beta] +1)}_h(f\omega_\ell)\|_{\bbp}$ and 
     		$\|(\delta^{(k)}_hf)(\delta^{(n)}_h\omega_\ell(\cdot+nh))\|_{\bbp}$ for $0\le k<m$ and $0\le n\le m$ {\black with $m=[\beta]+1$}. Further, 
     		due to the polynomial form \eqref{Om1-bis}, by Taylor expansions, one can check that:  for all $h\in\mathbb R^d$ such that $|h|_\bba\le 1$,
     		$$
     		|\delta^{(n)}_h\omega_\ell(z+nh)|\lesssim |h|^n_\bba\omega_\ell(z)\omega_\ell(nh),
     		$$
     		and further
     		$$ \|(\delta^{(k)}_hf)(\delta^{(n)}_h\omega_\ell(\cdot+nh))\|_{\bbp}\lesssim |h|^n_\bba\|(\delta^{(k)}_hf)\omega_\ell\|_{\bbp}.
     		$$
     		By iteration, it thus follows that 
     		$$
     		\sup_{|h|_\bba\le 1}\frac{\|(\delta^{([ \beta]+1)}_hf)\omega_\ell\|_{\bbp}}{|h|^\beta_\bba}\lesssim \|\omega_\ell f\|_{\bbp}+\sup_{|h|_\bba\le 1}\frac{\|\delta^{([ \beta]+1)}_h(f\omega_\ell)\|_{\bbp}}{|h|^\beta_\bba}.
     		$$
     		The claim \eqref{WeightedBlock} is finally derived by the equivalent norm \eqref{CH1} with $q=\infty$ and $s=\beta>0$.\\
     		For the proof of \eqref{WeightedBlock-Intermediate}, recall that, for any integer $m\ge 1$ and $h\in\mathbb R^{2d}$, the $m^{\text{\rm th}}$ difference operator writes as 
    $$
    \delta^{(m)}_hf(z)=\sum_{n=0}^m(-1)^{m-n}C^n_mf(z+nh).
    $$
    For fixed $m\ge 1$, we shall so define the class $\{\phi^{m,\bba}_k\}_{k\ge 0}$ as
    $$
    \phi^{m,\bba}_k(\xi)=(-1)^{m}\sum_{n=1}^m(-1)^{m-n}C^n_m\phi^\bba_k (n\xi),\ \ \xi\in\mathbb R^{2d}.
    $$
    Given that, for any $n\in\mathbb N\setminus\{0\}$, $\{\phi^\bba_j(n\cdot)\}_{j\ge 0}$ is a partition of unity, $\{\phi^{m,\bba}_k\}_{k\ge 0}$ defines itself a partition of unity with
    $$
    \sum_{k\ge 0}\phi^{m,\bba}_k(\xi)=(-1)^{m}\sum_{n=1}^m(-1)^{m-n}C^n_m=1.
    $$
    Hence, for all $j\ge 0$, $\cR^\bba_jf=	\sum_{k\ge 0}\check\phi^\bba_j*\check\phi^{m,\bba}_k*f$. Next, since $(\phi^{\bba}_k(n\cdot))\check{\ }(z)=n^{-2d}\check\phi^{\bba}_k(n^{-1} z)$ and $n^{-2d}\check\phi^\bba_j*\check\phi^\bba_k(n^{-1}\cdot)=(\phi^\bba_j\phi^\bba_k(n\cdot))\check{\ }=0$ for all $|k-j|>n+2$, from the above, we get
    \begin{align*}
    	\cR^\bba_jf(z)=	\sum_{k\ge 0}\check\phi^\bba_j*\check\phi^{m,\bba}_k*f(z)=\sum_{k\ge 0, |k-j|\le m+2}\check\phi^\bba_j*\Big((-1)^{m}\sum_{n=1}^m(-1)^{m-n}C^n_m (\phi^\bba_k(n\cdot))\check{\ }*f\Big)(z).
    \end{align*}
    Moreover by change of variables and as $\int \check\phi^\bba_k(\overline{z}) d \overline{z}=0$, 
    \begin{align*}
    &(-1)^m\sum_{n=1}^m(-1)^{m-n}C^n_m n^{-2d}\check\phi^{\bba}_k(n^{-1} \cdot)*f(z)=(-1)^m\sum_{n=1}^m(-1)^{m-n}C^n_m\int \check\phi^{\bba}_k(-\overline{z})f(z+n\overline{z}) d  \overline{z}\\ &=(-1)^m\sum_{n=0}^m(-1)^{m-n}C^n_m\int \check\phi^{\bba}_k(-\overline{z})f(z+n\overline{z}) d  \overline{z}=(-1)^m\int \check\phi^\bba_k(-\overline{z})\delta^{(m)}_{\overline{z}}f(z)d\overline{z}.
    \end{align*}
    Combining all the above, 
    $$
    \cR^\bba_jf(z)
    =(-1)^m\sum_{k\ge 0, |k-j|\le m+2}\check\phi^\bba_j*\int \check\phi^\bba_k(-\overline{z})\delta^{(m)}_{\overline{z}}f(z)d\overline{z}=:\sum_{k\ge 0, |k-j|\le m+2}\check\phi^\bba_j*\tilde\cR^\bba_kf(z).
    $$
    Applying the weighted Young inequality \eqref{WeightedYoung}, it follows that
    \begin{align*}
    \|(\cR^\bba_jf)\omega_\ell\|_{\bbp}\lesssim \sum_{k\ge 0, |k-j|\le m+2}\|\omega_\ell\check\phi^\bba_j\|_{\bbb1}\|(\tilde\cR^\bba_kf)\omega_\ell\|_{\bbp}.
    \end{align*}
    As $\phi^\bba_j$ is a Schwartz function, $\check\phi^\bba_j$ is itself a Schwartz function, and according to \eqref{SX4},  we have $\sup_j\|\omega_\ell\check\phi^\bba_j\|_{\bbb1}<\infty$. Thus,
    \begin{align*}
    \|(\tilde\cR^\bba_kf)\omega_\ell\|_{\bbp}&=\|\int \check\phi^\bba_k(-\overline{z})\big(\omega_\ell\delta^{(m)}_{\overline{z}}f\big)(\cdot)\, d \overline{z}\|_{\bbp}\le \int |\check\phi^\bba_k(-\overline{z})|\|\omega_\ell\delta^{(m)}_{\overline{z}}f\|_{\bbp}\, d  \overline{z}\\
    &\le \int_{|\overline{z}|_\bba\le 1} |\check\phi^\bba_k(-\overline{z})|\|\omega_\ell\delta^{(m)}_{\overline{z}}f\|_{\bbp}\, d  \overline{z}+\int_{|\overline{z}|_\bba> 1} |\check\phi^\bba_k(-\overline{z})|\|\omega_\ell\delta^{(m)}_{\overline{z}}f\|_{\bbp}\, d  \overline{z}.
    \end{align*}
    Taking now $m=[ \beta]+1$, and as $\|\check\phi^\bba_k|\cdot|_\bba^\beta\|_{\bbb1}\lesssim 2^{-\beta j}$
    \begin{align*}
    \int_{|z|_\bba\le 1} |\check\phi^\bba_k(-\overline{z})|\|\omega_\ell\delta^{(m)}_{\overline{z}}f\|_{\bbp}\, d  \overline{z}\lesssim 2^{-\beta j}\sup_{|\overline{z}|_\bba\le 1}\frac{\|(\delta^{([ \beta]+1)}_{\overline{z}}f) \omega_\ell\|_{\bbp}}{|\overline{z}|_\bba^\beta}.
    \end{align*}
    Meanwhile, since, for all $z,h\in\mathbb R^{2d}$, \eqref{ClassWeight} and \eqref{WeightGrowth} ensure that $|\delta^{(1)}_{h}\omega(z)|\lesssim |h|_{\bba}\omega_\ell(h)\omega_\ell(z)$ for all $h\in\mathbb R^{2d}$, by change of variables,
    $$
    \|(\delta^{(1)}_{\overline{z}}f)\omega_\ell\|_{\bbp}=\|(\delta^{(1)}_{-\overline{z}}\omega_\ell)f\|_{\bbp}\lesssim|\overline{z}|_\bba\omega_\ell(\overline{z})\|\omega_\ell f\|_{\bbp}.
    $$
    Iterating this estimate, it follows that
    $$
    \int_{|z|_\bba> 1} |\check\phi^\bba_k(-\overline{z})|\|\omega_\ell\delta^{(m)}_{\overline{z}}f\|_{\bbp}\, d  \overline{z}\lesssim \Big(\int_{|z|_\bba> 1} |\check\phi^\bba_k(-\overline{z})||\overline{z}|^m_{\bba}\omega_\ell(\overline{z})^m\, d  \overline{z}\Big)\|\omega_\ell f\|_{\bbp}\lesssim 2^{-j\beta }\|\omega_\ell f\|_{\bbp}.
    $$    
    Hence, we can conclude that 
    $$
    \|(\cR^\bba_jf)\omega_\ell\|_{\bbp}\lesssim 2^{-\beta j}\Big(\sup_{|\overline{z}|_\bba\le 1}\frac{\|(\delta^{([ \beta]+1)}_{\overline{z}}f) \omega_\ell\|_{\bbp}}{|\overline{z}|_{\bba}}+\|\omega_\ell f\|_{\bbp}\Big),
    $$
    and so conclude \eqref{WeightedBlock-Intermediate}.
    }
    \end{proof}
    {We next establish a weighted version of the estimate \eqref{AD0306} in Lemma \ref{SemigroupEstimate} for the kinetic semi-group $P_t$ in the case $\alpha=2$.\comm{JF 05/04: Considering only the case $\alpha=2$ makes naturally sense due to the weight.}
    To this aim, let us first establish a weighted version of the Bernstein inequality in Lemma \ref{BI00} (we refer to \cite[Lemma 2.4]{HZZZ22} for the special case $\bbp=(p,p)$).
    \bl
     For any positive integers $\bbk=(k_1,k_2)$, $\bbp,\bbp_1\in[1,\infty]$ with $\bbp_1\le \bbp$, and for $\cA_{\bbp_1,\bbp}$ defined as in \eqref{KJ1}, there is a constant $C=C(\bbk,\bbp,\bbp_1,d,\ell)>0$ such that for all $j\ge 0$, $f\in\sS'$,
     \begin{equation}\label{WeightedBernstein}
     	  	\|\big(\nabla^{k_1}_x\nabla^{k_2}_v\cR^\bba_jf\big)\omega_\ell\|_{\bbp}\lesssim 	2^{j\big((1+\alpha)k_1+k_2+\cA_{\bbp_1,\bbp}\big)}\|\big(\cR^\bba_jf\big)\omega_\ell\|_{\bbp_1}.
     \end{equation}
    \el
    \begin{proof}
    Using (again) the partition of unity defined by $\{\phi^\bba_k\}_{k\ge 0}$ and since, according to \eqref{Cx8}, $\text{supp}(\phi^\bba_j\phi^\bba_k)=0$ whenever $|k-j|>2$, we have
    $$
    \nabla^{k_1}_x\nabla^{k_2}_v\cR^\bba_jf=\sum_{k\ge 0,|k-j|\le 2}\check\phi^\bba_k*\nabla^{k_1}_x\nabla^{k_2}_v\cR^\bba_jf=\sum_{k\ge 0,|k-j|\le 2}\nabla^{k_1}_x\nabla^{k_2}_v\check\phi^\bba_k*\cR^\bba_jf.
    $$
    Applying the weighted Young inequality \eqref{WeightedYoung}, it follows that 
    \begin{align*}
    \|(\nabla^{k_1}_x\nabla^{k_2}_v\cR^\bba_jf)\omega_\ell\|_{\bbp}&\le \sum_{k\ge 0,|k-j|\le 2}\|(\nabla^{k_1}_x\nabla^{k_2}_v\check\phi^\bba_k*\cR^\bba_jf)\omega_\ell\|_{\bbp}\\
    &\lesssim\sum_{k\ge 0,|k-j|\le 2}\|(\cR^\bba_jf)\omega_\ell\|_{\bbp_1}\|\big(\nabla^{k_1}_x\nabla^{k_2}_v\check\phi^\bba_k\big)\omega_\ell\|_{\bbp_2},
    \end{align*}
    where $\frac 1{\bbp'}=\frac 1{\bbp_1}+\frac 1{\bbp_2}$
     Since each $\check\phi^\bba_k$ is a Schwartz function, 
      $\|\big(\nabla^{k_1}_x\nabla^{k_2}_v\check\phi^\bba_0)\omega_\ell\|_{\bbp_2}
    <\infty$. For $k\ge 1$, owing to the scaling property \eqref{SX4} and by change of variables, we get
    \begin{align*}
    \|\check\phi^\bba_k\omega_\ell\|_{\bbp_2}&=2^{(k-1)\big((1+\alpha)k_1+k_2+(2+\alpha)d\big)} \|\big(\nabla^{k_1}_x\nabla^{k_2}_v\check\phi^\bba_1\big)(2^{(k-1)\bba}\cdot)\omega_\ell\|_{\bbp_2}\\
    &=2^{(k-1)\big((1+\alpha)k_1+k_2+\cA_{\bbb1,\bbp_2}\big)}\|(\nabla^{k_1}_x\nabla^{k_2}_v\check\phi^\bba_1)\omega_\ell(2^{-(k-1)\bba}\cdot)\|_{\bbp_2}.
    \end{align*}
    Since $\omega_\ell(2^{-(k-1)\bba}\cdot)\lesssim_{C(\ell)}\omega_\ell$ and  $\cA_{\bbb1,\bbp_2}=\cA_{\bbp_1,\bbp}$, this gives the claim.
    \end{proof}
    }
	\textcolor{Dgreen}{Owing to this preliminary, we have the following  heat kernel estimate:\comm{JF 05/04: Beside the restriction $\beta>0$ made for natural simplicity (and since we never need $\beta<0$), compared to \eqref{AD0306}, due to the weight function, we need an additional dependency on the time horizon.}
		\bl\label{WeightedHeatEstimate} Let $P_t$ be defined as in \eqref{SCL11} for $\alpha=2$.
		For any $\bbp,\bbp_1\in[1,\infty]^2$ with $\bbp_1\le \bbp$, $\beta>0,\ \textcolor{black}{\gamma>0}$ and $0\le T<\infty$, there is a constant $C=C( d,\beta,\gamma,\bbp,\bbp_1,\ell,T)>0$ 
		such that for any $f\in \bB^{\beta}_{\bbp_1;\bba}(\omega_\ell)$, $0<t\le T$, and $j\ge 0$,
		\begin{align}
			\|\big(\cR^\bba_jP_t \nabla_vf\big)\omega_\ell\|_{\bbp}
			\le C  2^{j(\cA_{\bbp_1,\bbp}+1-\beta)}
			\big((2^{2j}t)^{-\gamma}\wedge 1\big)\|\omega_\ell f\|_{\bB^{\beta}_{\bbp_1;\bba}},\label{WeightedAD0306}
		\end{align}
		where $\cA_{\bbp_1,\bbp}$ is defined as in \eqref{KJ1}, with $\alpha=2$ and $\bba=(3,1)$.
		\el
}
		\begin{proof} \textcolor{Dgreen}{The proof essentially consists in replicating the original arguments of the items $(i)$ and $(ii)$ in \cite[Lemma 2.12]{HRZ23}, the auxiliary Lemma 2.11 therein and Lemma 3.1 in \cite{ZZ21}.
			}
		
		\textcolor{Dgreen}{
			 For the case $j=0$, starting from \eqref{SCL11} - with $p_t$ the density function as in \eqref{SCL1} -, by direct calculations, applying twice the weighted Young inequality \eqref{WeightedYoung} and using the invariance of $\mL^\bbp$ by $\transOp_t$, it follows that\comm{JF 07/04: Details of the second equality: By direct calculations, $\big(\check\phi^\bba_0*\transOp_t(p_t*f)\big)\omega_\ell=\transOp_{t}\big(\transOp_{-t}\check\phi^\bba_0*(p_t*f)\big)\omega_\ell=\transOp_{t}\Big(\big(\transOp_{-t}\check\phi^\bba_0*(p_t*f)\big)\transOp_{-t}\omega_\ell\Big)$ and invariance by translation of $\mL^\bbp$.}
			 \begin{align*}
			 	\|(\cR^\bba_0 P_t\nabla_vf)\omega_\ell\|_{\bbp}&=\|\big(\check\phi^\bba_0*\transOp_t(p_t*\nabla_vf)\big)\omega_\ell\|_{\bbp}=\|\big(\transOp_{-t}\check\phi^\bba_0*(p_t*\nabla_vf)\big)\transOp_{-t}\omega_\ell\|_{\bbp_1}\\
			 	&\lesssim \|\big(\transOp_{-t}\check\phi^\bba_0*(p_t*\nabla_vf)\big)\omega_\ell\|_{\bbp}=\|\big(\nabla_v\transOp_{-t}\check\phi^\bba_0*(p_t* f)\big)\omega_\ell\|_{\bbp}\\
			 	&\lesssim \|(\nabla_v\transOp_{-t}\check\phi^\bba_0) \omega_\ell\|_{\bbp_2}\|(p_t*f)\omega_\ell\|_{\bbp_1}\\
			 	&\lesssim \|(\nabla_v\transOp_{-t}\check\phi^\bba_0) \omega_\ell\|_{\bbp_2}\|p_t\omega_\ell\|_{\bbb1} \|f\omega_\ell\|_{\bbp_1},
			 \end{align*}
			 for $\bbp_2\in[1,\infty]^2$ such that $1/\bbp_2=1/\bbp_1+1/\bbp'$. Owing to the scaling properties of $p_t$ (with $\alpha=2$) and the polynomial form of $\omega_\ell$, $\sup_{t>0}\|p_t\omega_\ell\|_{\bbb1}<\infty$. Since 
			 $\nabla_v\transOp_{-t}\check\phi^\bba_j=\transOp_{-t}\nabla_v\check\phi^\bba_j+t\transOp_{-t}\nabla_x\check\phi^\bba_j$,  $|\transOp_{t}\nabla_v\omega_\ell|\lesssim_{C(T)}\omega_\ell$, and since $\check{\phi}^\bba_0$ is a Schwartz function, we have
			 \begin{align*}
			 \|\nabla_v\transOp_{-t}\check\phi^\bba_0 \omega_\ell\|_{\bbp_2}\lesssim_{C(T)} \|\nabla_v\check\phi^\bba_0 \omega_\ell\|_{\bbp_2}+\|\nabla_x\check\phi^\bba_0 \omega_\ell\|_{\bbp_2}
			 <\infty.
			 \end{align*}
			 Therefore, as $\beta>0$, the embedding \eqref{AB2} yields $\bB^\beta_{\bbp_1;\bba}\hookrightarrow \mL^{\bbp_1}$ and, for all $\gamma> 0$,
			 $$
			 \|\big(\cR^\bba_0P_t f\big)\omega_\ell\|_{\bbp}
			 \le C  
			 \|\omega_\ell f\|_{\bbp_1}\lesssim_{C(T)} \big(t^{-\gamma}\wedge 1\big)\|\omega_\ell f\|_{\bB^\beta_{\bbp_1;\bba}}.
			 $$}
		\textcolor{Dgreen}{
			For the general case $j\ge 1$,\comm{JF 05/04: what follows is the original argument from Lemma 2.4-$(i)$ \cite{HRZ23} slightly rewritten to be self-contained without invoking other side results.} as the $\phi^\bba_j$'s define a partition of unity, $\sum_{k\ge 0}\cR^\bba_kf=f$ and thus 
			\begin{align}\label{Partition}
			\cR^\bba_jP_t \nabla_vf&=\sum_{k\ge 0}\cR^\bba_j\transOp_{t} p_t*
			\transOp_{t}\cR^\bba_k\nabla_vf=\sum_{k\ge 0}\check\phi^\bba_j *\transOp_{t}p_t*\transOp_{t}\check\phi^\bba_k*\transOp_{t} \nabla_vf\\
			&=\sum_{k\ge 0}\check\phi^\bba_j*\transOp_{t}\check\phi^\bba_k *\transOp_{t}p_t*\transOp_{t} \nabla_vf.\nonumber
			\end{align}
			Since\comm{JF 17/04: From this line up to "Coming back" in the middle of the page, we could potentially refer to Lemma 2-11 \cite{HRZ23} - the proof slightly herein differs from the proof therein, but, up to slight troubles of not-direct and forgotten "copy-paste" in the paper \cite{HRZ23}, we could potentially save a few lines.} $\transOp_{t}\check\phi^\bba_k=(\phi^\bba(\xi_1,\xi_2-t\xi_1))\check{\ }=(\tilde\transOp_t\phi^\bba_k)\check{\ }$, we can write $\check\phi^\bba_j*\transOp_{t}\check\phi^\bba_k=:(\phi^\bba_j\tilde\transOp_t\phi^\bba_k)\check{\ }$. Observe next that (with the convention $2^{k-1}=0$ if $k=0$) the support of $\phi^\bba_j\tilde\transOp_t\phi^\bba_k$ is given by
			\begin{align*}
			\text{supp}\big(\phi^\bba_j\tilde\transOp_t\phi^\bba_k\big)&=\Big\{\xi=(\xi_1,\xi_2)\,:\,2^{j-1}\le |\xi_1|^{1/3}+|\xi_2|\le 2^{j+1}, \,2^{k-1}\le |\xi_1|^{1/3}+|\xi_2-t\xi_1|\le 2^{k+1} \Big\}.
			\end{align*}
			By triangular inequality, each element of this set has to satisfy: $|\xi_1|^{1/3}\le 2^{j+1}\wedge 2^{k+1}=2^{j\wedge k+1}$,
			\begin{align*}
				&(2^{k-1}-t|\xi_1|)\vee 2^{j-1} \le |\xi_1|^{1/3}+|\xi_2|\le 2^{j+1}\wedge (2^{k+1}+t|\xi_1|)\le 2^{j+1}\wedge (2^{k+1}+t2^{3(k+1)}),\\
				&(2^{j-1}-t|\xi_1|)\vee 2^{k-1}\le|\xi_1|^{1/3}+|\xi_2-t\xi_1|\le 2^{k+1}\wedge (2^{j+1}+t|\xi_1|)\le 2^{k+1}\wedge (2^{j+1}+t2^{3(j+1)}).
			\end{align*}
		    In view of these delimiters, given $j\ge 0$ and $t\in[0,T]$, for all $k$ outside the family of indices 
			$$
			\Theta^t_j:=\{ k\in\mathbb N\,:\,2^{k-1}\le 2^{j+1}+t2^{3(j+1)}\,\text{and}\,2^{j-1}\le 2^{k+1}+t2^{3(k+1) }\},
			$$
			the support set	$\text{supp}\big(\phi^\bba_j\tilde\transOp_t\phi^\bba_k\big)$ is empty and $\phi^\bba_j\tilde\transOp_t\phi^\bba_k=0$. 
			Additionally, for all $k\in \Theta^t_j$, 
			$$
			2^{k-1}\le \big(2^{j+1}+t2^{3(j+1)}\big)\le2^{j+1}\big(1+t2^{2(j+1)}\big), $$
			and so
			$$
			k\le j+2+(\ln(2))^{-1}\ln\big(1+t2^{2(j+1)}\big)=:j+2+\Theta_+. 
			$$
			From this upper-bound, we further deduce a lower-bound  on $\Theta^t_j$, observing that 
			$$
			2^{k+1}\ge 2^{j-1}\big(1+t2^{2(k+1)}\big)^{-1}\ge 2^{j-1}\big(1+t2^{2(j+3+\Theta_+)}\big)^{-1}
			$$
			and so
			$$
			k\ge \Big(j-2-(\ln(2))^{-1}\ln\big(1+t2^{2(j+2+\Theta_+)}\big)\Big)\vee 0=:\big(j-2-\Theta_-\big)\vee 0.
			$$
			Therefore, as $\beta>0$, $2^{\beta \Theta_-}=\big(1+t2^{j+1+\Theta_+}\big)^{\beta}$ and $2^{\Theta_+}=(1+t2^{2(j+1)})^2$,
			\begin{equation}\label{ClosedSummation}
				\sum_{k\in\Theta^t_j}2^{-k\beta}=\frac{2^{-\beta(j-2-\Theta_-)\vee 0}-2^{-\beta(j+2+\Theta_+)}}{1-2^{-\beta}}\lesssim_{C(\beta)}2^{-\beta j}\big(1+t2^{3(j+1)}\big)^\beta.
			\end{equation}
		    Coming back to \eqref{Partition}, 
		    it follows that
			$$
			\cR^\bba_jP_t f= \sum_{k\in\Theta^t_j }\check\phi^\bba_j*\transOp_{t}\check\phi^\bba_k *\transOp_{t}p_t*\transOp_{t}\nabla_v f
			=\sum_{k\in \Theta^t_j}\cR^\bba_j\transOp_{t} p_t*
			\transOp_{t}\cR^\bba_k\nabla_vf.
			$$
			Now, using again the weighted Young inequality \eqref{WeightedYoung}, we get 
			\begin{align*}
			\|\big(\cR^\bba_jP_t \nabla_vf\big)\omega_\ell\|_{\bbp_1}&\lesssim \sum_{ k\in\Theta^t_j}\|\big((\cR^\bba_j\transOp_{t}p_t)*(\transOp_{t}\cR^\bba_kf) \big)\omega_\ell\|_{\bbp_1}\lesssim\sum_{k\in\Theta^t_j} \|(\cR^\bba_j\transOp_{t}p_t)\omega_\ell\|_{\bbb1}\|(\transOp_{t}\cR^\bba_k\nabla_vf) \omega_\ell\|_{\bbp_1}.
			\end{align*}
			 Since, for any $z\in\mathbb R^{2d}$, $|\transOp_{-t}\omega_\ell(z)|\lesssim_{C(T,\ell)}\omega_\ell(z)$,
		applying successively the weighted Bernstein inequality \eqref{WeightedBernstein} (with $\bbk=(0,1)$, $\bbp=\bbp_1$) and the estimate \eqref{WeightedBlock} in Lemma \ref{lem:WeightedNorm} yields
			$$
			\|(\transOp_{t}\cR^\bba_k\nabla_vf) \omega_\ell\|_{\bbp_1}\lesssim 2^k\|(\cR^\bba_kf) \omega_\ell\|_{\bbp_1}\lesssim 2^{k(1-\beta)}\|f \omega_\ell\|_{\bB^{\beta}_{\bbp_1;\bba}}.
			$$
			Finally, extending the proof arguments of Lemma 3.1 in \cite{ZZ21} to a weighted setting, we may claim that \comm{JF 05/04: The estimate in \cite{ZZ21} is in fact the key estimate for getting \eqref{S1:00}. For our present paper, the case $\omega_\ell=1$ naturally drive the most important part in the estimate, and we could infer the result without demonstration. Nevertheless, for the sake of completeness and given that I could not find a single reference for the present estimate - \cite{HWZ20} gives something close but not exactly what we need.}  
			\begin{equation}\label{WeightedHeatKernel}
			\|(\cR^\bba_j\transOp_{t}p_t)\omega_\ell\|_{\bbb1}\lesssim  (2^{2j}t)^{-\gamma}\wedge 1.
			\end{equation}
			From this estimate, and using \eqref{ClosedSummation} with $\gamma+3\beta$ in place of $\gamma$, we obtain that
			\begin{align*}
				\sum_{k\in\Theta^t_j} \|(\cR^\bba_j\transOp_{t}p_t)\omega_\ell\|_{\bbb1}\|(\transOp_{t}\cR^\bba_kf) \omega_\ell\|_{\bbp_1}\lesssim \Big((2^{2j}t)^{-\gamma-\beta}\wedge 1\Big) \|f \omega_\ell\|_{\bB^{\beta}_{\bbp_1;\bba}}\sum_{k\in\Theta^t_j}2^{k(1-\beta)},
				\end{align*}
				and so
				\begin{align*}
				\sum_{k\in\Theta^t_j} \|(\cR^\bba_jP_t\nabla_vf)\omega_\ell\|_{\bbb1}&\lesssim 2^{j(1-\beta)}\Big((2^{2j}t)^{-\gamma-3\beta}\wedge 1\Big)(1+t2^{3j})^{\beta}  \|f \omega_\ell\|_{\bB^{\beta}_{\bbp_1;\bba}}.
			\end{align*}
		This gives \eqref{WeightedAD0306}.}
			
			\textcolor{Dgreen}{
			The proof of the estimate \eqref{WeightedHeatKernel} proceeds as follows: recalling \eqref{SX4} with $\alpha=2$,  $\phi^\bba_j(\xi)=2^{4(j-1)d}\phi^\bba_1(2^{(j-1)\bba}\xi)$ (for $2^{s\bba}\xi=(2^3\xi_1,2\xi_2)$). Using (again) the scaling properties of $p_t$, by change of variables, we have, for $\hbar:=(2^{j}\sqrt{t})^{-1}$, and $2^{-\bba}\overline{z}=(2^{-3}\overline{z}_1,2^{-1}\overline{z_2})$, 
			\begin{align*}
			\|(\cR^\bba_j\transOp_{t}p_t)\omega_\ell\|_{\bbb1}&=2^{-4d}\int \omega_\ell(z)\Big|\int\check\phi^\bba_1(2^{-\bba}\overline{z})p_1(x-\hbar^3\overline{x}+v-\hbar \overline{v},v-\hbar\overline{v}) d  \overline{z}\Big| d  z\\
			&=:\int \omega_\ell(z)\Big|\int\check\phi^\bba_1(2^{-\bba}\overline{z})H(z-\transOp_{\hbar}\hbar^{\bba}\overline{z}) d  \overline{z}\Big| d  z.
			\end{align*}
		    For $m\in\mathbb N$, defining next the operator $\triangle^m_z$ and its inverse $\triangle^{-m}_{\xi}$, as
			$$
			\triangle^m_z=\big(\triangle_x+\triangle_v\big)^m, \ \ (\triangle^{-m}_{\xi}f)\hat{\ }(\xi_1,\xi_2)=\big(|\xi_1|^2+|\xi_2|^2\big)^{-m}\hat{f}(\xi_1,\xi_2),
			$$ 
			we get
			\begin{align*}
			\int\check\phi^\bba_1(2^{-\bba}\overline{z})H(z-\transOp_{\hbar}\hbar^{\bba}) d  \overline{z}=\int\check(\triangle^{-m}_{\overline{z}}\phi^\bba_1)(2^{\bba}\overline{z})\triangle^m_{\overline{z}} H(z-\transOp_{\hbar}\hbar^{\bba}) d  \overline{z}.
			\end{align*}
			Using the growth property \eqref{WeightGrowth} of $\omega_\ell$, it follows that
			\begin{align*}
			\|(\cR^\bba_j\transOp_{t}p_t)\omega_\ell(z)\|_{\bbb1}
			\lesssim\|\big(\triangle^{-m}_{\overline{z}}\check\phi^\bba_1(2^{\bba}\overline{z})\big)\omega_\ell(\overline{z})\|_{\bbb1}\sup_{\overline{z}}\|\big(\triangle^m_{\overline{z}}H(\cdot-\hbar\transOp_{h}\overline{z})\big)\omega_\ell(\cdot-\hbar\transOp_{h}\overline{z})\|_{\bbb1}.
			\end{align*}
			Owing to the regularity of $p_1$, 
			$$
			\|\triangle^m_{\overline{z}}H(\cdot-\hbar\transOp_{h}\overline{z})\omega_\ell(\cdot-\hbar\transOp_{h}\overline{z})\|_{\bbb1}\lesssim \hbar^{2m},
			$$
			and since $\check\phi^\bba_1$ is a Schwartz function, $\|\triangle^{-m}_{\overline{z}}\check\phi^\bba_1(2^{-\bba}\overline{z})\omega_\ell(\overline{z})\|_{\bbb1}$ is finite. Hence
			$$
			\|\cR^{\bba}_j\transOp_{t}p_t\omega_\ell\|\lesssim \hbar^{2m}=(2^{2j}t)^{-m}.
			$$
		 As $m$ is arbitrary, \eqref{WeightedHeatKernel} follows.}
		\end{proof}
	
\section{Besov norm of scale functions}\label{APP_BESOV}
For $\phi\in\cS(\mR^{2d})$,
we define a dilation operator \textcolor{Dgreen}{as}: for $\lambda>0$,
$$
\phi_\lambda(z):=\cD_\lambda\phi (x,v):=\lambda^{(2+\alpha)d}\phi(\lambda^{1+\alpha} x, \lambda  v).
$$
Let $\bbp,\bbp'\in[1,\infty]^2$ with $\frac1{\bbp'}+\frac1{\bbp}=\bbb1$. For any $\beta\ge0$, by the scaling properties of the mixed norm $\Vert\,\cdot\Vert_{\bbp}$ \textcolor{Dgreen}{defined in \eqref{LP1} }and of the anisotropic distance \eqref{AnisoDist},   
 it is easy to see that\textcolor{Dgreen}{, for all $\lambda\geq 0$,}
\begin{align}\label{Es1}
\||\cdot|_\bba^\beta\phi_\lambda(\cdot)\|_{\bbp}= \lambda^{\bba\cdot d/\bbp'-\beta}\||\cdot|_\bba^\beta\phi(\cdot)\|_{\bbp}
= \lambda^{\cA_{\bbb1,\bbp}-\beta}\||\cdot|_\bba^\beta\phi(\cdot)\|_{\bbp}.
\end{align}
We first \textcolor{Dgreen}{establish}  the following crucial lemma for later use. 
\bl
Let $\psi$ be a smooth function so that $\hat\psi\in C^\infty_c(\mR^{2d}\setminus\{0\})$ and set $\psi_\ell=\mathcal D_\ell\psi$. 
For any $\bbp\in[1,\infty]^2$ and $\beta\geq 0$\comm{{\color{MColor}Better use the symbol $m$ than $\beta$ to avoid misleading/contradictory notation - this would be consistent given the proof arguments.} \textcolor{Dgreen}{JF 07/04: Although I still agree on the comment, the notation $\beta$ finds somehow its justification in the last lemma of this section ...}}, there is a constant $C=C(\beta,d,\bbp,\psi,\phi\textcolor{Dgreen}{)}>0$ 
such that for all $\lambda,\ell\geq 0$ and $t\in\mR$\comm{JF 07/04: Are we really looking at negative time somewhere ?},
\begin{align}\label{DD4}
\|\psi_\ell*\transOp_t\phi_\lambda\|_\bbp{\color{black}\le C}
 \lambda^{\cA_{\bbb1,\bbp}} (1+|t|\lambda^\alpha)^\beta
(\lambda/\ell)^{\beta}.
\end{align}
\el
\begin{proof}
\textcolor{Dgreen}{Set} $\hbar:=\lambda/\ell$. First of all, by the \textcolor{Dgreen}{very} definition \textcolor{Dgreen}{of $\mathcal D_\lambda$ }and change\textcolor{Dgreen}{s} of variable\textcolor{Dgreen}{s}, it is easy to see that\comm{JF 22/03: The previous estimate was $\psi_\ell*(\omega\transOp_t\phi_\lambda)
	=\cD_{\ell}(\psi*\cD_{\ell^{-1}}(\omega\transOp_t\cD_\lambda\phi))=\cD_{\ell}(\psi*\cD_{\hbar}(\cD_{\lambda^{-1}}\omega\cdot\transOp_{t\lambda^\alpha}\phi))$ introducing a weight $\omega$ which disappeared just after.}
\begin{align*}
\psi_\ell*(\textcolor{Dgreen}{\transOp_t}\phi_\lambda)\textcolor{Dgreen}{=\cD_{\ell}\big(\cD_{\ell^{-1}}(\psi_\ell* (\transOp_t\cD_\lambda\phi))\big)}
=\cD_{\ell}(\psi*\cD_{\ell^{-1}}(\textcolor{Dgreen}{\transOp_t}\cD_\lambda\phi))=\cD_{\ell}(\psi*\cD_{\hbar}(\cD_{\lambda^{-1}}\textcolor{Dgreen}{\transOp_{t\lambda^{\alpha}}}\phi)),
\end{align*}
\textcolor{Dgreen}{the last equality following from the property $  \transOp_{t}\mathcal D_{\lambda}\phi(x,v)
=\mathcal D_{\lambda}\transOp_{\lambda^\alpha t}\phi(x,v)$.}
By the \textcolor{Dgreen}{embedding \eqref{AB2}, the} Young inequality \textcolor{Dgreen}{\eqref{Con}}, \eqref{Es1} and \textcolor{Dgreen}{since}  $\|\transOp_tf\|_\bbp=\|f\|_\bbp$, we have
\begin{align}\label{AA2}
\|\psi_\ell*\transOp_t\phi_\lambda\|_\bbp
\textcolor{Dgreen}{=\ell^{\cA_{\bbb1,\bbp}}\|\psi*\cD_{\hbar}(\cD_{\lambda^{-1}}\transOp_{t\lambda^{\alpha}}\phi)\|_{\bbp}\lesssim} \ell^{\cA_{\bbb1,\bbp}}\|\psi\|_{\textcolor{Dgreen}{\bB^{0,1}_{\bbb1;\bba}}}\|\cD_{\hbar}\transOp_{t\lambda^\alpha}
\phi\|_\bbp
=\lambda^{\cA_{\bbb1,\bbp}}\|\psi\|_1\|\phi\|_\bbp,
\end{align}
which gives \eqref{DD4} for $\beta=0$. Next we consider the case $\beta>0$.
\textcolor{Dgreen}{Introducing the Laplacian operator, $\triangle:=\triangle_x+\triangle_v$, f}or any $m\in\mN$, since supp$(\hat\psi)\subset\mR^{2d}\setminus\{0\}$, $\Delta^{-m}\psi$ is a Schwartz function, we further have
\begin{align}\label{DD1}
\|\psi_\ell*\transOp_t\phi_\lambda\|_\bbp&=\ell^{\cA_{\bbb1,\bbp}}\|\Delta^{-m}\psi*\Delta^{m}\cD_{\hbar}\transOp_{t\lambda^\alpha}
\phi\|_\bbp
\leq\ell^{\cA_{\bbb1,\bbp}}\|\Delta^{-m}\psi\|_{\textcolor{Dgreen}{\bB^{0,1}_{\bbb1;\bba}}}\|\nabla^{2m}\cD_{\hbar}\transOp_{t\lambda^\alpha}
\phi\|_\bbp,
\end{align}
where $\nabla=(\nabla_x,\nabla_v)$.
Noting that by the chain rule,
\begin{align*}
|\nabla^{2m}\cD_\hbar \transOp_{t\lambda^\alpha}
\phi|\leq\sum_{k=0}^{2m}|\nabla_x^{2m-k}\nabla_v^{k}\cD_\hbar \transOp_{t\lambda^\alpha}
\phi|
=\sum_{k=0}^{2m}\hbar^{{(2m-k)(1+\alpha)+k}} |\cD_\hbar\nabla_x^{\textcolor{Dgreen}{2}m-k}\nabla^{k}_v \transOp_{t\lambda^{\alpha}}
\phi|,
\end{align*}
by \eqref{Es1}, we have
\begin{align}\label{DD2}
\|\nabla^{2m}\cD_\hbar \transOp_{t\lambda^\alpha}
\phi\|_\bbp
\leq\sum_{k=0}^{2m}\hbar^{2(1+\alpha)m-\alpha k} \hbar^{\cA_{\bbb1,\bbp}}\|\nabla_x^{2m-k}\nabla^{k}_v \transOp_{t\lambda^{\alpha}}
\phi\|_\bbp.
\end{align}
On the other hand, noting that \textcolor{Dgreen}{since}  $\nabla_x\transOp_t=\transOp_t\nabla_x$ and $\nabla_v\transOp_t=\transOp_t\circ(-t\nabla_x+\nabla_v)$,
$$
\nabla_x^{2m-k}\nabla^{k}_v\transOp_{t\lambda^{\alpha}}
\phi=\transOp_{t\lambda^{\alpha}}
\nabla_x^{2m-k}(-t\lambda^\alpha\nabla_x+\nabla_v)^k
\phi
=\transOp_{t\lambda^\alpha}
\sum_{i=0}^k\frac{k!}{i!(k-i)!}(-t\lambda^\alpha)^i
\nabla_x^{2m-k+i}\nabla^{k-i}_v\phi.
$$
\textcolor{Dgreen}{As} $\|\transOp_tf\|_\bbp=\|f\|_\bbp$, we have, for any $k=0,\cdots,2m$,
\begin{align}\label{DD3}
\|\nabla_x^{2m-k}\nabla^{k}_v\transOp_{t\lambda^\alpha}
\phi\|_\bbp
\leq\sum_{i=0}^k(|t|\lambda^\alpha)^i
 \|\nabla_x^{2m-k+i}\nabla^{k-i}_v\phi\|_\bbp\lesssim(1+|t|\lambda^\alpha)^{2m}.
\end{align}
Hence, by \eqref{DD1}-\eqref{DD3}, we have, for any $m\in\mN$,
\begin{align*}
\|\psi_\ell*\transOp_t\phi_\lambda\|_\bbp
&\lesssim\ell^{\cA_{\bbb1,\bbp}}\sum_{k=0}^{2m}\hbar^{2(1+\alpha)m-\alpha k} \hbar^{\cA_{\bbb1,\bbp}}(1+|t|\lambda^{\alpha})^{2m}
\\
&\lesssim\lambda^{\cA_{\bbb1,\bbp}} (1+|t|\lambda^\alpha)^{2m}
\sum_{k=0}^{2m}\hbar^{2(1+\alpha)m-\alpha k}.
\end{align*}
If $\hbar>1$, then
$$
\|\psi_\ell*\transOp_t\phi_\lambda\|_\bbp\lesssim\lambda^{\cA_{\bbb1,\bbp}}(1+|t|\lambda^\alpha)^{2m}
\hbar^{2(1+\alpha)m}\leq\lambda^{\cA_{\bbb1,\bbp}} (1+|t|\lambda^\alpha)^{2(1+\alpha)m}\hbar^{2(1+\alpha)m}.
$$
If $\hbar\leq1$, then
$$
\|\psi_\ell*\transOp_t\phi_\lambda\|_\bbp\lesssim\lambda^{\cA_{\bbb1,\bbp}} (1+|t|\lambda^\alpha)^{2m}
\hbar^{2m}.
$$
Thus by \eqref{AA2} and interpolation, we obtain \eqref{DD4}.
\end{proof}
\textcolor{Dgreen}{
Let us now recall the weight function previously defined in \eqref{Om1-bis}\comm{JF 04/04: Given the new (and prior) appendix section, all the properties of the weight function stated here, have been moved to this first appendix section.},
$$ \omega_\ell(z):=\big(1+(1+|x|^2)^{\frac1{1+\alpha}}+|v|^2\big)^{\ell/2},\ \ z=(x,v)\in\mR^{2d}.
$$
}
\textcolor{Dgreen}{Observing that
	\begin{equation}
	|\omega_\ell\phi_\lambda|\leq \cD_\lambda|\omega_\ell\phi|,
\end{equation}
	w}e have the following useful corollary.
\bl\label{08:18}
Let $\phi\in C^\infty_c(\mR^{2d})$ and $\phi_\lambda:=\cD_\lambda\phi$.
For any $T,\ell,\beta\geq 0$ and $\bbp\in[1,\infty]^2$, there is a constant $C=C\textcolor{Dgreen}{(T,\ell,d,\bbp,\beta,\phi)}>0$\comm{JF 07/04: Old : $C(T,\ell,\beta,\bbp,d,\phi)$} such that for all $t\in[0,T]$ and $\lambda\geq 1$,\comm{S. and JF 23/03: We have to emphasize the bound for the non-weighted case $\ell=0$.}
\begin{align}\label{Es3}
\|\omega_\ell\transOp_t\phi_\lambda\|_{\bB^\beta_{\bbp;\bba}}{\color{black}\le C}
\lambda^{(1+\alpha)\beta+\cA_{\bbb1,\bbp}}.
\end{align}
\textcolor{Dgreen}{In particular, for $\ell=0$ (namely $\omega_\ell=1$), we have
$$
\textcolor{Dgreen}{\|\transOp_t\phi_\lambda\|_{\bB^\beta_{\bbp;\bba}}}{\color{black}\le C}
\lambda^{(1+\alpha)\beta+\cA_{\bbb1,\bbp}}.
$$
}
\el
\begin{proof}
First of all we show \eqref{Es3} for $\ell=0$. Let $\beta\geq 0$ and $\bbp\in[1,\infty]^2$.
For $j\geq 1$, noting  that
$$
\cR_j^\bba\transOp_t\phi_\lambda=\check\phi^\bba_j*\transOp_t\phi_\lambda.
$$
\textcolor{Dgreen}{According to} \eqref{DD4} \textcolor{Dgreen}{- in the preceding lemma - applied to $\psi_\ell=\cD_\ell\psi=\check\phi^\bba_j$ with, in view of \eqref{SX4}, $\ell=2^{j}$  and $\psi(z)=2^{-1}\check\phi^\bba_2(2^{-\bba}z)$)\comm{JF 07/04: Not the best writing but I believe some additional details were needed here ...}}, there is a constant $C>0$ such that, for all $t\geq 0,\lambda>0$ and $j\geq 1$,
$$
\|\cR_j^\bba\transOp_t\phi_\lambda\|_{\bbp}\le C 
(1+t\lambda^\alpha)^{\beta}
\lambda^{\bba\cdot \frac{d}{\bbp'}}(\lambda/2^{j})^{\beta}.
$$
Hence, for any $t\geq 0$ and $\lambda>0$, \textcolor{Dgreen}{for the high frequency part of $2^{j\beta}\|\cR^\bba_j\transOp_{t}\phi_\lambda\|_{\bbp}$, $j\ge 1$\comm{JF 07/04: Additional mention related to the mention of the low frequency just after.}, we have }
$$
\sup_{j\geq 1}\Big(2^{j\beta}\|\cR_j^\bba\transOp_t\phi_\lambda\|_{\bbp}\Big)\le C 
(1+t\lambda^\alpha)^{\beta}
\lambda^{\bba\cdot \frac{d}{\bbp'}+\beta}.
$$
On the other hand, for the lower frequency part $j=0$,  by  \textcolor{MColor}{\eqref{Es1}}, we have, for $\lambda\geq 1$,
$$
\|\cR_0^\bba\transOp_t\phi_\lambda\|_{\bbp}\lesssim\|\transOp_t\phi_\lambda\|_{\bbp}=\|\phi_\lambda\|_{\bbp}=\lambda^{\bba\cdot \frac{d}{\bbp'}}\|\phi\|_\bbp
\lesssim (1+t\lambda^\alpha)^{\beta} 
\lambda^{\beta+\bba\cdot \frac{d}{\bbp'}}.
$$
Combining the above calculations, we obtain \textcolor{Dgreen}{
\begin{align}\label{Es3_Temp}
	\|\transOp_t\phi_\lambda\|_{\bB^\beta_{\bbp;\bba}}\lesssim_{C_T}
	\lambda^{(1+\alpha)\beta+\cA_{\bbb1,\bbp}}.
\end{align}}
Next we consider the general \textcolor{Dgreen}{case}
 $\ell>0$. \textcolor{Dgreen}{L}et $\psi\in C^\infty_c(\mR^{2d})$ such that \textcolor{Dgreen}{ for all $\lambda\ge 1$, $\psi\phi_\lambda=\phi_\lambda$
 	(since the supports of $\phi_\lambda$ shrink as $\lambda$ increases, $\psi$ can be simply taken equal to one on $ \text{supp}(\phi_1)$).}
  \textcolor{Dgreen}{B}y definition \textcolor{Dgreen}{of $\transOp_{t}$}, we clearly have
\begin{align*}
\|\omega_\ell\transOp_t\phi_\lambda\|_{\bB^\beta_{\bbp;\bba}}
=\|\omega_\ell\transOp_t(\textcolor{Dgreen}{\psi}\phi_\lambda)\|_{\bB^\beta_{\bbp;\bba}}
\textcolor{Dgreen}{=\|\big(\omega_\ell\transOp_t(\psi)\big)\transOp_t(\phi_\lambda)\|_{\bB^\beta_{\bbp;\bba}}}\le C 
\|\omega_\ell\Gamma_t\textcolor{Dgreen}{\psi}\|_{\textcolor{black}{\bB^{\beta,1}_{\boldsymbol{\infty},\bba}}} 
\|\transOp_t\phi_\lambda\|_{\bB^\beta_{\bbp;\bba}},
\end{align*}
\textcolor{black}{using Theorem 2 p. 177 in \cite{RunSic-97}, which readily extends to the current anisotropic setting, for the  control of the Besov norm of the product. 
}
Since $\psi$ has compact support, 
 it is easy to see that for any $T>0$, there is a bounded domain $D_T\textcolor{Dgreen}{\subset \mathbb R^{2d}}$ so that
$$
\mbox{supp}(\Gamma_t\textcolor{Dgreen}{\psi})\subset D_T,\ \ \forall t\in[0,T],\ \lambda\geq 1.
$$
Hence,
$$
\sup_{t\in[0,T]}\sup_{\lambda\geq 1}\|\omega_\ell\Gamma_t\textcolor{Dgreen}{\psi}\|_{\textcolor{Dgreen}{\bB^{\beta,1}_{\boldsymbol{\infty}};\bba}}
<\infty.
$$
The result now follows.
\end{proof}
\section{Convergence rate of empirical measure}\label{CONV_EMP_SAMP_MU0}
We first \textcolor{Dgreen}{state} the following elementary lemma.
\bl\label{LeA1}
Let $f_i, i=1,\cdots,N$, be a sequence of \textcolor{Dgreen}{Borel} functions on $\mR^{2d}$. Let $\bbp=(p_x,p_v)\in[1,\infty]^2$ and $q:=p_x\wedge p_v\wedge 2$. Then
$$
\left\|\sum_{i=1}^N|f_i|^2\right\|_{\bbp/2}\leq\left(\sum_{i=1}^N\|f_i\|^q_{\bbp}\right)^{2/q}.
$$
\el
\begin{proof}
\textcolor{Dgreen}{The estimate} follows by $(\sum_{i=1}^N|a_i|)^{\gamma}\leq\sum_{i=1}^N|a_i|^{\gamma}$ for $\gamma\in(0,1]$ and Minkowski's inequality.
\end{proof}

Now we show the following convergence rate estimate for the empirical measure of i.i.d random variables.\comm{JF 08/04: The modification are here due to the "typo" on the prefactor $m^m$ in the discrete BDG inequality on $\mL^\bbp$ for which no references could have been found ... The dependency of the prefactor on $m$ remains so implicit. Jf 09/04: Given these changes, the remark C.3 has been put in commented.} 
\bl\label{Le34} 
Let $(\xi_1,\cdots,\xi_N)$ be i.i.d random variables in $\mR^{2d}$ with common distribution $\mu$ \textcolor{Dgreen}{and l}et $\bar\mu_N:=\frac 1 N\sum_{i=1}^N\delta_{\textcolor{Dgreen}{\{}\xi_i\textcolor{Dgreen}{\}}}$ be the \textcolor{Dgreen}{related} empirical distribution measure.
Let $\textcolor{Dgreen}{\bbp_1}\in[1,\infty)^2$\comm{JF 25/03: Same correction as before: To not create confusion with the integrability index of $\mu_0$, $\bbp_0\rightarrow \bbp_1$.}  and $\textcolor{Dgreen}{\bbp_1}\leq\bbp\in(1,\infty)^2$.
If $\textcolor{Dgreen}{\bbp_1}=\bbp$, we let $\ell=0$; otherwise, we choose $\ell>\cA_{\textcolor{Dgreen}{\bbp_1},\bbp}$.
For any continuous $\phi:\mR^{2d}\to\mR$ \textcolor{Dgreen}{and for all {\black $m> 1$} for which $\mathbb E[|\xi_1|^{\ell m}]<\infty$},
there is a constant $C=C(d,\textcolor{Dgreen}{\bbp_1},\bbp,\textcolor{Dgreen}{m},\ell)>\textcolor{Dgreen}{1}$ such that \textcolor{Dgreen}{for all $N\ge 1$},
\begin{equation}\label{INIT}
\|\phi*(\bar\mu_N-\mu)\|_{L^m(\Omega;\mL^{\textcolor{Dgreen}{\bbp_1}})}\le C  
\Big(1+(\mE\textcolor{Dgreen}{\big[}|\xi_1|^{\ell m}_{\ZXC \bba}\textcolor{Dgreen}{\big]})^{1/m}\Big)N^{\frac1q-1}\|\omega_\ell\phi\|_\bbp,
\end{equation}
where  $q=p_x\wedge p_v\wedge 2$ \textcolor{Dgreen}{and where $\omega_\ell$ is the weight function defined in \eqref{Om1-bis}}.
\el
\begin{proof} 
For each $i=1,\cdots,N$, we define a family of \textcolor{Dgreen}{centered} i.i.d. random fields on $\mR^{2d}$ with zero mean by
$$
Y_i(z):=\phi*(\delta_{\textcolor{Dgreen}{\{}\xi_i\textcolor{Dgreen}{\}}}-\mu)(\textcolor{Dgreen}z),\ \ z\in\mR^{2d}.
$$
Let $\omega_\ell$ be defined \textcolor{Dgreen}{as} in \textcolor{Dgreen}{\eqref{Om1-bis}} with $\ell$ as in the \textcolor{Dgreen}{lemma} \textcolor{Dgreen}{and l}et $\bbp_2\in[1,\infty]^2$ be defined by $\frac1{\textcolor{Dgreen}{\bbp_1}}=\frac1{\bbp}+\frac1{\bbp_2}$. By H\"older's inequality \textcolor{Dgreen}{and applying \eqref{WeightIntegrability} with $\ell>\bba\cdot\frac d{\bbp_2}=\cA_{\bbp_1,\bbp}$}, we have\comm{JF 25/03: A recall to why $\|\omega^{-1}_\ell\|_{\bbp_2}<\infty$ would have been appreciated, I guess. Update 07/04: Done.}
$$
\Big\|\sum_{i=1}^NY_i\Big\|_{\textcolor{Dgreen}{\bbp_1}}\leq\Big\|\sum_{i=1}^N\omega_\ell Y_i\Big\|_{\bbp}\|\omega^{-1}_\ell\|_{\bbp_2}\lesssim \Big\|\sum_{i=1}^N\omega_\ell Y_i\Big\|_{\bbp}.
$$
Since $\mL^\bbp$ is a UMD space for $\bbp\in(1,\infty)^2$ (see Corollary \ref{cor:UMD} \textcolor{Dgreen}{in the Appendix \ref{sec:Type} below}), by \textcolor{Dgreen}{the functional} BDG's inequality\textcolor{Dgreen}{ (see \cite{VerYar-19} and references therein):
	$$
		\mathbb E\Big[\| M_N\|^m_\bbp\big]\le C_{m,\bbp} \mathbb E\Big[\|[ 
        M]^{1/2}_N\|^m_\bbp\Big],
	$$ 
	for discrete-time $\mL^\bbp$-valued martingale $M$} {\black : $k\in\{1,...,N\} \to M_k$}, we have \comm{JF 07/04: NOT SO FAST: On one hand, it seems that you have considered that the prefactor in the BDG  inequality as if $\mL^\bbp$ is $\mathbb R$ which is naturally untrue; On the second hand, in the case we have $\mathbb R$-valued martingale, the factor is of order $m^{m/2}$, not $m^m$.}
\begin{align*}
\mE\Big\|\sum_{i=1}^N \omega_\ell Y_i\Big\|_{\bbp}^m
&\textcolor{Dgreen}{\le (C_{m,\bbp})^m}
\mE\Big\| \Big(\sum_{i=1}^{N} |\omega_\ell Y_i|^2\Big)^{1/2}\Big\|^{m}_\bbp
=\textcolor{Dgreen}{(C_{m,\bbp})^m}
\mE\left\|\sum_{i=1}^{N} |\omega_\ell Y_i|^2
\right\|^{m/2}_{\bbp/2}.
\end{align*}
By Lemma \ref{LeA1}, we get for $q:={p_x\wedge p_v\wedge 2}$,
$$
\mE\Big\|\sum_{i=1}^N \omega_\ell Y_i\Big\|_{\bbp}^m
\lesssim \textcolor{Dgreen}{(C_{m,\bbp})^m}\mE\left(\sum_{i=1}^{N}\|\omega_\ell Y_i\|^q_\bbp\right)^{m/q}.
$$
Noting that\textcolor{Dgreen}{, by the growth property \eqref{WeightGrowth} of $\omega_\ell$, and for $\mu(\omega_\ell):=\int \omega_\ell(z)\mu(dz)$,}\comm{\textcolor{MColor}{To be checked. I think the control for $\|\omega_\ell Y_i\|_\bbp$ is Ok but follows from Young and manipulations therein of the type described here. JF 26/03: Checked.}}
\begin{align*}
|\omega_\ell Y_i|(z)&\textcolor{Dgreen}{\le |\omega_\ell(z)\phi*(\delta_{\{\xi_i\}}+\mu)(z)|}\\
&\lesssim|(\omega_\ell\phi)*(\delta_{\{\xi_i\}}\textcolor{Dgreen}{+}\mu)|(z)
+|\phi*(\delta_{\textcolor{Dgreen}{\{}\xi_i\textcolor{Dgreen}{\}}}\textcolor{Dgreen}{+}\mu)|(z)(\omega_\ell(\xi_i)+\mu(\omega_\ell)).
\end{align*}
\textcolor{Dgreen}{B}y Young's inequality and \textcolor{Dgreen}{as} $1\leq\omega_\ell$, \textcolor{Dgreen}{it follows that}
$$
 \|\omega_\ell Y_i\|_\bbp
 \lesssim\|\omega_\ell\phi\|_\bbp
 +\|\phi\|_\bbp(\omega_\ell(\xi_i)\textcolor{Dgreen}{+\omega_\ell(\mu)}+1)\leq \|\omega_\ell\phi\|_\bbp(\omega_\ell(\xi_i)\textcolor{Dgreen}{+\omega_\ell(\mu)}+1).
$$
Combining the above calculations, we get
\begin{align*}
\mE\textcolor{Dgreen}{\big[}\|\phi*(\bar\mu_N-\mu)\|_{\textcolor{Dgreen}{\bbp_1}}^m\textcolor{Dgreen}{\big]}
=\frac1{N^{m}}\mE\Big\|\sum_{i=1}^NY_i\Big\|_{\textcolor{Dgreen}{\bbp_1}}^m
\lesssim
\frac{\textcolor{Dgreen}{(C_{m,\bbp})^m}}{N^{m}} N^{\frac mq}\|\omega_\ell\phi\|_\bbp\mE\textcolor{Dgreen}{\Big[}\big(1+\omega_\ell(\xi_1)\big)^m\textcolor{Dgreen}{\Big]}\textcolor{Dgreen}{.}
\end{align*}
The proof is complete.
\end{proof}
\textcolor{Dgreen}{Now, recalling the dilatation operator $\mathcal D_\lambda$ from Appendix \ref{APP_BESOV} and the notation:
$$
\phi_\lambda(z)=\mathcal D_\lambda\phi(x,v)=\lambda^{(2+\alpha)d}\phi(\lambda^{1+\alpha}x,\lambda v), \ \lambda>0,
$$
w}e have the following useful result.
\bt\label{THM_CI}
Let $(\xi_1,\cdots,\xi_N)$ be i.i.d random variables in $\mR^{2d}$ with common distribution $\mu$.
Let $\bar\mu_N:=\frac 1 N\sum_{i=1}^N\delta_{\textcolor{Dgreen}{\{}\xi_i\textcolor{Dgreen}{\}}}$ be the empirical distribution measure.
Let $\bbp_{\textcolor{Dgreen}{\bbb1}}\in[1,\infty)^2$  and $\bbp_{\textcolor{Dgreen}{\bbb1}}\leq\bbp\in(1,\infty)^2$.
If $\bbp_{\textcolor{Dgreen}{\bbb1}}=\bbp$, we let $\ell=0$; otherwise, we choose $\ell>\cA_{\bbp_{\textcolor{Dgreen}{\bbb1}},\bbp}$. \textcolor{Dgreen}{Then, given 
	$$
	q=p_x\wedge p_v\wedge 2,
	$$
	for any $\beta<0$ and any integer $m{\color{black}>1}$ satisfying
	$$
	\mu\big(|\cdot|^{\ell m}\big):=\mathbb E\big[|\xi|^{\ell m}_\bba\big]<\infty,
	$$
	there exists a constant $C>0$ such that, for all $N,\lambda\ge 1$,
	\begin{equation}\label{Con1}
	\|\phi_\lambda*(\bar\mu_N-\mu)\|_{L^m(\Omega;\bB^{\beta,1}_{\bbp_{\textcolor{Dgreen}{\bbb1}};\bba})}\le C  \Big(1+\mE\big[|\xi_1|_{\bba}^{\ell m}\big]\Big)^{1/m}N^{\frac1q-1}\lambda^{\cA_{\bbb1,\bbp}}.
	\end{equation}
}
\et
\begin{proof}
By Minkowski's inequality, we have
\begin{align*}
\|\phi_\lambda*(\bar\mu_N-\mu)\|_{L^m(\Omega;\bB^{\beta,1}_{\bbp_{\textcolor{Dgreen}{\bbb1}};\bba})}
&=  \Big\|\sum_{j\ge0} 2^{\beta j}\|\cR^\bba_j(\phi_\lambda*(\bar\mu_N-\mu))\|_{\bbp_{\textcolor{Dgreen}{\bbb1}}}\Big\|_{\textcolor{Dgreen}{L}^m(\Omega)}\\
&\leq \sum_{j\ge0} 2^{\beta j}  \|\cR^\bba_j(\phi_\lambda*(\bar\mu_N-\mu))\|_{L^m(\Omega;\mL^{\bbp_{\textcolor{Dgreen}{\bbb1}}})}.
\end{align*}
Noting that
$$
\cR^\bba_j(\phi_\lambda*(\bar\mu_N-\mu))=(\cR^\bba_j\phi_\lambda)*(\bar\mu_N-\mu)=(\check\phi^\bba_j*\phi_\lambda)*(\bar\mu_N-\mu),
$$
by \eqref{INIT}, we have\textcolor{Dgreen}{, for $\omega_\ell$ defined as in \eqref{Om1-bis},} 
\begin{align*}
\|\cR^\bba_j(\phi_\lambda*(\bar\mu_N-\mu))\|_{L^m(\Omega;\mL^{\bbp_{\textcolor{Dgreen}{\bbb1}}})}
&\le C  
\Big(1+(\mE\textcolor{Dgreen}{\big[}|\xi_1|_\bba^{\textcolor{Dgreen}{\ell m}}\textcolor{Dgreen}{\big]})^{1/m}\Big)N^{\frac1q-1}\|\omega_\ell\cR^\bba_j\phi_\lambda\|_{\bbp}.
\end{align*}
	{\color{black}M}ultiplying the above by $2^{\beta j}$ (recall that $\beta<0$) and summing the resulting expression over $j\ge 0$, we get
\begin{align*}
\|\phi_\lambda*(\bar\mu_N-\mu)\|_{L^m(\Omega;\bB^{\beta,1}_{\bbp_{\textcolor{Dgreen}{\bbb1}};\bba})}&\lesssim \Big(1+(\mE\textcolor{Dgreen}{\big[}|\xi_1|_\bba^{\textcolor{Dgreen}{\ell m}}\textcolor{Dgreen}{\big]})^{1/m}\Big)N^{\frac1q-1}\sum_{j\ge 0}2^{\beta j}\| (\cR^\bba_j\phi_\lambda)\omega_\ell\|_{\bbp}\\
&\lesssim \Big(1+(\mE\big[|\xi_1|_\bba^{\ell m}\big])^{1/m}\Big)N^{\frac1q-1}\sup_{j\ge 0}\| (\cR^\bba_j\phi_\lambda)\omega_\ell\|_{\bbp}
\end{align*}
Applying \eqref{WeightedBlock} in Lemma \ref{lem:WeightedNorm}, we eventually get
\begin{equation*}
\sup_{j\ge 0}\| (\cR^\bba_j\phi_\lambda)\omega_\ell\|_{\bbp}\le \|\omega_\ell\phi_\lambda\|_{\bbp}
 =\lambda^{\mathcal A_{\bbb1,\bbp}},
\end{equation*}
which ends the proof.
\end{proof}
\section{Characterization of $\mL^{\bbp}$ as a UMD space and as a Banach space of martingale type.}\label{sec:Type}

Let us first recall the definition of UMD spaces.  

\bd[\cite{Pisier-16}, Chapter 5] 
{\black
A Banach space $(E,\|\cdot\|_E)$ is said to be a \textit{unconditional martingale difference}  ($\text{UMD}$ in short) space if there exists $p\in (1,\infty)$ and $c=c(E,p)>0$ such that for every $\{-1,1\}$-valued sequence $\{\epsilon_n\}_n$ and for every martingale $\{M_n\}_{n\ge 1}$ defined \textcolor{Dgreen}{o}n some filtered probability space $(\Omega,\mathcal F,\{\mathcal F_n\}_{n\ge 0},\mathbb P)$ and converging as $n\rightarrow \infty$ in $L^{p}(\Omega;E)$, we have
		\begin{equation}\label{UMDProp}
			\sup_n\Vert\sum_{j=0}^n\epsilon_j d M_j\Vert_{L^{p}(\Omega;E)}\le c\sup_n\Vert M_n\Vert_{L^{p}(\Omega;E)},
		\end{equation}
		for $dM_j:=M_j-M_{j-1}$, $dM_0=0$.	}	
		\ed

		Simple instances of UMD spaces are the Euclidean space, Hilbert spaces and $L^{p}$\textcolor{Dgreen}{-spaces} for $p\in(1,\infty)$. Further, for any UMD space $E$, and any measure space $(A,\mathcal A,\nu)$, every $L^p(A;E)$, with $p\in(1,\infty)$ is a UMD space, \cite{Pisier-16}, Corollary 5.22. As such\textcolor{Dgreen}{, for any $p_x\in(1,\infty)$, $L^{p_x}(\mathbb R^d)$ is a UMD space, and by iteration, }\comm{JF 09/04: Since, we bother to give a few lines of explanations for the $M$-type property, we can afford an additional half-sentence of explanation here.}we can immediately \textcolor{black}{derive} \textcolor{Dgreen}{the corollary:}  
		\bc\label{cor:UMD}
		For all $\bbp\in (1,\infty)^2$, $\mL^\bbp$ is a $\text{UMD}$-space.
		\ec
		 
		For the characterization of $\mL^{\bbp}$ as a Banach space of $M$-type or martingale type $p$ - \textcolor{Dgreen}{characterization} which is used to apply the martingale inequality in Theorem  \ref{thm:Hausenblas-BDG} for the proof of Theorem \ref{thm:MartPart-Stable} -, and for the sake of completeness, let us briefly recall the definition of these spaces\textcolor{Dgreen}{.}\comm{JF 09/04: Following the "discovery" of the Section 3.5 in \cite{HNVW-16}, what follows as been hugely simplified, the proposition 3.5.30 gives us a way to largely simplify the proof of $\mL^\bbp$ being of $M$-type. Given that the proof is three lines (with details) the result has been relabeled as a corollary instead of a lemma (the relabel in the main text has been also made).} 
	\bd\label{def:MartingaleType}  A Banach space $(E,\textcolor{Dgreen}{\|}\cdot\textcolor{Dgreen}{\|}_E)$ is said to be of martingale type $p$, for $p\in[1,\infty]$, if  there exists a constant $C=C(E,p)$ such that for all finite $E$-valued martingale $\{M_k\}_{k=1}^n$, it holds that, for $dM_k=M_k-M_{k-1}$ (with the convention $M_0=0$)
	\begin{equation}\label{HaussenblasType}
		\sup_{1\le k\le n}\textcolor{Dgreen}{\|}  M_k\textcolor{Dgreen}{\|}_{L^p(\Omega;E)}\le C  \Big(\sum_{k=1}^{n}\mathbb E[\textcolor{Dgreen}{\|}dM_{k}\textcolor{Dgreen}{\|}_E^{p}] \Big)^{1/p}.
	\end{equation}	
	\ed 
	Characteristically, every Banach space is of martingale type $1$, Hilbert spaces are of  martingale type $2$ \textcolor{Dgreen}{and we may refer to \cite{Pisier-16} for a detailed account of the geometric properties related to \eqref{HaussenblasType} and to \cite{HNVW-16} for related martingale transforms.}
	
	\textcolor{Dgreen}{	
		 \textcolor{black}{Similarly to the UMD property}, the martingale property \eqref{HaussenblasType} \textcolor{black}{holds} for Lebesgue spaces: for $E$ a Banach space of martingale type $p\in[1,2]$ and $(A,\mathcal A,\nu)$ a measure space, any $L^r(A;E)$ space with $r\in(1,\infty)$ is of martingale type $r\wedge p$ (see  \cite{HNVW-16}, Proposition 3.5.30
). 
\textcolor{black}{We have the following result:}
		\bc\label{MixedBanachType}
			For all $\bbp\in (1,\infty)^2$, $\mL^\bbp$ is a Banach space of  martingale type $p$ for all $1\le p \le p_x\wedge p_v\wedge 2$.
		\ec
		\begin{proof}
		Given $\bbp=(p_x,p_v)$, since $\mathbb R^{2d}$ is of martingale type $2$, $L^{p_x}(\mathbb R^d)$ is of martingale type $p_x\wedge 2$. As such, the iterated space $L^{p_v}(\mathbb R^{d};L^{p_x}(\mathbb R^{d}))$ is of martingale type $p_v\wedge (p_x\wedge 2)$. Since once \eqref{HaussenblasType} is satisfied for some $p$, the property also holds true for all integrability index in $[1,p]$ (see Corollary 3.5.28, \cite{HNVW-16}). This gives the claim. 
		\end{proof}
	}

\section{Gronwal\textcolor{Dgreen}{l} inequality of Volterra type}\label{sec:AppD}

In this appendix, we  show a Gronwall inequality of Volterra type that is crucial for  the proof of \textcolor{Dgreen}{Theorem \ref{main01}}.
First of all, we give the following elementary estimate.\comm{JF 16/03: I would suggest to relabel the following lemma into a corollary, in order to keep the subsequent lemmas more valuable.}
\bl\label{lem:TimeSingIntegral}
Let $T>0$ and $a_i,b_i\in(-1,\infty)$ with $a_i+b_i>-1$, $i=1,2$. Then there is a constant $C=C(T,a_1,b_1,a_2,b_2)>0$ such that for all $t\in[0,T]$ and \textcolor{Dgreen}{all Borel} function $f:[0,T]\to \mR_+$,
\begin{align}\label{0215:00}
\int_0^t (t-s)^{a_1}s^{b_1} \left(\int_0^s (s-r)^{a_2}r^{b_2}f(r) d  r\right) d  s\le C \sum_{i=1}^2\int_0^t (t-s)^{a_i}s^{b_i}f(s) d  s.
\end{align}
In particular, if $b_1\ge b_2$, then
\begin{align}\label{0215:01}
\int_0^t (t-s)^{a_1}s^{b_1} \left(\int_0^s (s-r)^{a_2}r^{b_2}f(r) d  r\right) d  s\le C \int_0^t (t-s)^{a_2}s^{b_2}f(s) d  s.
\end{align}
\el
\begin{proof}
By Fubini\textcolor{Dgreen}{-Tonelli}'s theorem we have
$$
\int_0^t (t-s)^{a_1}s^{b_1} \left(\int_0^s (s-r)^{a_2}r^{b_2}f(r) d  r\right) d  s=\int_0^tK(t,r) r^{b_2} f(r) d  r,
$$
where
$$
K(t,r):=\int_r^t (t-s)^{a_1}s^{b_1}(s-r)^{a_2} d  s.
$$
{\bf (Case $b_1\ge0$).} Since $a_1+1>0$, we have
\begin{align}\label{SQ1}
K(t,r)\leq T^{b_1} \int_r^t (t-s)^{a_1}(s-r)^{a_2} d  s\lesssim (t-r)^{a_1+a_2+1}\leq T^{a_1+1} (t-r)^{a_2}.
\end{align}
{\bf (Case $b_2\le b_1<0$).}  In view of $a_1+b_1>-1$, we have
\begin{align}\label{SQ2}
K(t,r)\leq \int_r^t (t-s)^{a_1}(s-r)^{a_2+b_1} d  s\lesssim (t-r)^{a_1+b_1+a_2+1}\leq T^{a_1+b_1+1} (t-r)^{a_2}.
\end{align}
{\bf (Case $b_1<b_2\le0$).} In this case, \textcolor{Dgreen}{for all $s\in[r,t]$, $s^{b_1}=s^{b_1-b_2}s^{b_2}\le (s-r)^{b_1-b_2}r^{b_2}$,
	and} we have
\begin{align*}
K(t,r)\leq r^{b_1-b_2}\int_r^t (t-s)^{a_1}(s-r)^{a_2+b_2} d  s\lesssim (t-r)^{a_1+a_2+b_2+1}r^{b_1-b_2}\leq T^{a_2+b_2+1} (t-r)^{a_1}r^{b_1-b_2}.
\end{align*}
{\bf (Case $b_1<0<b_2$).}  In this case, we have
\begin{align*}
K(t,r)\leq T^{b_2} r^{b_1-b_2}\int_r^t (t-s)^{a_1}(s-r)^{a_2} d  s\lesssim (t-r)^{a_1+a_2+1}r^{b_1-b_2}\leq T^{a_2+1} (t-r)^{a_1}r^{b_1-b_2}.
\end{align*}
Combining the above calculations, we obtain \eqref{0215:00}. \eqref{0215:01} is from \eqref{SQ1} and \eqref{SQ2}.
\end{proof}

Now we can show the following Gronwall inequality.
\bl[Gronwall's inequality of Volterra type]\label{0214:lem00}
Let $T>0$, $n\in\mN$ and $a_i,b_i\in (-1,\infty)$ with  $a_i+b_i>-1$ for all $i=1,..,n$. Assume that $f,g:[0,T]\to\mR_+$ are two \textcolor{Dgreen}{Borel} measurable functions and satisfy that for some $c_0>0$ and
almost all $t\in(0,T]$,
\begin{align}\label{0215:03}
    f(t)\le g(t)+c_0\sum_{i=1}^n\int_0^t (t-s)^{a_i} s^{b_i} f(s) d  s.
\end{align}
Then there is a constant $C=C(T,n,(a_i,b_i)_{i=1}^n,c_0)>0$\comm{JF 16/03: Although not essential, it may be more precise to say that $C>1$.} such that for almost all $t\in(0,T]$,
\begin{align*}
    f(t)\le Cg(t)+C\sum_{i=1}^n\int_0^t (t-s)^{a_i} s^{b_i} g(s) d  s.
\end{align*}
\el
\begin{proof}
\textcolor{Dgreen}{The case $n=1$ yielding
\begin{align}\label{0215:03-Temp}
	f(t)\le g(t)+c_0\int_0^t (t-s)^{a_1} s^{b_1} f(s) d  s \, \Rightarrow  f(t)\le Cg(t)+C\int_0^t (t-s)^{a_1} s^{b_1} g(s) d  s
\end{align}
has been established in \cite[Lemma A.4]{Ha23}.}
 Now, let $n\textcolor{Dgreen}{>1}$ and $\{a_i\textcolor{Dgreen}{\}_{i=1}^n},\textcolor{Dgreen}{\{}b_i\}_{i=1}^n\subset(-1,\infty)$ 
satisfying $a_i+b_i>-1$. Without loss of generality, 
we assume that $b_1\ge b_2\ge\cdot\cdot\ge b_n$. 
Let 
$$
F_1(t):=g(t)+\textcolor{Dgreen}{c_0}\sum_{k=2}^{n}\int_0^t (t-s)^{a_k}s^{b_k}f(s) d  s,
$$
and, for any $i=2,..,n-1$,  
\begin{align*}
F_i(t):=g(t)+\textcolor{Dgreen}{c_0}\sum_{k=1}^{i-1}\int_0^t (t-s)^{a_k}s^{b_k}g(s) d  s+\textcolor{Dgreen}{c_0}\sum_{k=i+1}^{n}\int_0^t (t-s)^{a_k}s^{b_k}f(s) d  s.
\end{align*}
Then \eqref{0215:03} reads \textcolor{Dgreen}{as}
$$
f(t)\textcolor{Dgreen}{\le} F_1(t)+\textcolor{Dgreen}{c_0}\int_0^t (t-s)^{a_1}s^{b_1}f(s) d  s.
$$
Thus \textcolor{Dgreen}{applying \eqref{0215:03-Temp} }
 and \textcolor{Dgreen}{next} \eqref{0215:01} \textcolor{Dgreen}{(with the ordering $b_1\ge b_k$ for $k\ge 2$)}, we have
\begin{align*}
f(t)&\lesssim_{\textcolor{Dgreen}{C}} F_1(t)+\int_0^t (t-s)^{a_1}s^{b_1}F_1(s) d  s\\
&\textcolor{Dgreen}{= F_1(t)+ \int_0^t (t-s)^{a_1}s^{b_1}g(s)\, d  s+c_0\sum_{k=2}^n\int_0^t(t-s)^{a_1}s^{b_1}\Big(\int_0^s(s-r)^{a_k}r^{b_k} f(r)\, d  r\Big)\, d  s}\\
&\textcolor{Dgreen}{\stackrel{\eqref{0215:01}}{\lesssim} F_1(t)+ \int_0^t (t-s)^{a_1}s^{b_1}g(s)\, d  s+c_0\sum_{k=2}^n\int_0^t(t-s)^{a_k}s^{b_k} f(r)\, d  r}\\
&\textcolor{Dgreen}{=} F_1(t)+\int_0^t (t-s)^{a_1}s^{b_1}g(s) d  s\textcolor{Dgreen}{+c_0\sum_{k=3}^n\int_0^t(t-s)^{a_k}s^{b_k} f(r)\, d  r+c_0\int_0^t(t-s)^{a_2}s^{b_2} f(r)\, d  r}\\
&\textcolor{Dgreen}{\lesssim} F_2(t)+\int_0^t (t-s)^{a_2}s^{b_2}f(s) d  s.
\end{align*}
\textcolor{Dgreen}{Iterating the preceding estimate, successively using  \eqref{0215:03-Temp} and \eqref{0215:01},} we obtain
\begin{align*}
f(t)\lesssim F_{n-1}(t)+\int_0^t (t-s)^{a_n}s^{b_n}f(s) d  s,
\end{align*}
which in turn\textcolor{Dgreen}{, applying a last time \eqref{0215:03-Temp},} implies the desired estimate. 
\end{proof}

Now we apply the previous lemma to the function given in \eqref{AZ12}: 
\begin{align*}
	G_\beta(t,s)=t^{\frac{\beta-\beta_0}{\alpha}
	}(t-s)^{-\frac1{\alpha}}+t^{\frac{\beta+\Gap}{\alpha}
	}(t-s)^{-\frac{\Gap+\beta_0+1}{\alpha}}, \ \textcolor{Dgreen}{0\le s < t<\infty}, 
\end{align*}
where $\alpha\in(1,2]$, $\beta\geq 0$ and $\beta_0,\Gap$ \textcolor{Dgreen}{are} as in {\bf (H)}.
\bl\label{0214:lem02}
Let $T>0$, $\beta\in\textcolor{Dgreen}{\big[0,}(\beta_0+\alpha)\wedge (\alpha-\Gap)\textcolor{Dgreen}{\big)}$ 
and $\gamma\in [0,\alpha\wedge(\alpha-1+\beta-\beta_0))$. \textcolor{Dgreen}{Then the following properties hold.}
\begin{enumerate}[(i)]
\item There is a $p\geq 1$ large enough and $C_T>0$ so that for all $f\in L^p([0,t])$,
\begin{align}\label{E4}
\int_0^t G_\beta(t,s)s^{-\frac\gamma\alpha}f(s) d  s\le C_T\|f\|_{L^p([0,t])}.
\end{align}
\item There is a constant $C_T>0$ such that for all $t\in[0,T]$ and  measurable $f:[0,T]\to\mR_+$,
\begin{align}\label{0214:00}
\int_0^t G_\beta(t,s)s^{-\frac\gamma\alpha}\left(\int_0^s G_\beta(s,r)f(r) d  r\right) d  s\le C_T \int_0^t G_\beta(t,s)f(s) d  s.
\end{align}
\item Assume that $f,g,h:[0,T]\to\mR_+$ are three Borel {\color{black}non-negative} 
measurable functions and \textcolor{Dgreen}{satisfying, }
for some constant $c_0>0$ and  almost all $t\in(0,T]$,
\begin{align}\label{0214:03}
    f(t)\le g(t)+c_0\int_0^t G_\beta(t,s)s^{-\frac{\gamma}{\alpha}}f(s) d  s+\int_0^t G_\beta(t,s)h(s) d  s.
\end{align} 
Then there is a constant $C=C(T,\alpha,\beta,\beta_0,\Gap,c_0)>0$ such that for almost all $t\in(0,T]$,
{\black\begin{align*}
    f(t)\le Cg(t)+C\int_0^t G_\beta(t,s)(s^{-\frac{\gamma}{\alpha}}g(s)+h(s)) d  s.
\end{align*}}
\end{enumerate}
\el
\begin{proof}
(i) \textcolor{Dgreen}{As the exponents $\frac {-1}{\alpha}$, $\frac {-\gamma}{\alpha}$ and $-\frac {\Gap+\beta_0+1}{\alpha}$ are all strictly above $-1$, one can take $p'\in(1,\infty)$ close enough to $1$ so that 
	 $\|G_\beta(t,\cdot)\cdot^{-\tfrac{\gamma}{\alpha}}\|_{L^{p'}((0,T))}<\infty$. The statement then  follows by H\"older's inequality.}

\noindent
(ii) 
Since $\beta-\beta_0, \beta+\Lambda\in[0,\alpha)$, we have \textcolor{Dgreen}{$t^{\ell}-s^{\ell}\le (t-s)^{\ell}$ for $\ell\in\{(\beta-\beta_0)/\alpha,(\beta+\Gap)/\alpha\}$ and so}
\begin{align}
G_\beta(t,s)&\le \Big(s^{{\frac{\beta-\beta_0}{\alpha}}}+(t-s)^{{\frac{\beta-\beta_0}{\alpha}}}\Big)(t-s)^{-\frac1\alpha}  +\Big(s^{\frac{\beta+\Lambda}{\alpha}}+(t-s)^{\frac{\beta+\Lambda}{\alpha}}\Big) (t-s)^{-\frac{\Lambda+\beta_0+1}\alpha} \no\\
&\quad=:\sum_{i=1}^{\textcolor{Dgreen}{4}} (t-s)^{a_i}s^{b_i}\textcolor{Dgreen}{\stackrel{s^\ell\le t^{\ell} }{\le} G_\beta(t,s)+\sum_{i\in\{2,4\}}(t-s)^{a_i}\le G_\beta(t,s)+2t^{\tfrac{\beta-\beta_0}{\alpha}}(t-s)^{a_2-\tfrac{\beta-\beta_0}{\alpha}}}\nonumber\\
&\le 2G_\beta(t,s),\label{0214:01}
\end{align}
\textcolor{Dgreen}{taking}
$$
a_1=-\tfrac1\alpha, b_1=\tfrac{\beta-\beta_0}{\alpha}, \quad a_2=\tfrac{\beta-\beta_0-1}{\alpha}, b_2=0, \quad a_3=-\tfrac{\Lambda+\beta_0+1}\alpha, b_3=\tfrac{\beta+\Lambda}{\alpha},\quad\textcolor{Dgreen}{a_4=\tfrac{\beta-\beta_0-1}{\alpha},b_4=0.}
$$
\textcolor{Dgreen}{Note that}
\begin{align*}
\textcolor{Dgreen}{\alpha\big(}a_i+b_i\textcolor{Dgreen}{\big)}-\gamma=-1+\beta-\beta_0-\gamma>-\alpha,\quad i=1,2,3,4.
\end{align*}
Then \eqref{0214:00} is directly from \eqref{0215:00}.
Indeed, \eqref{0215:00} and \eqref{0214:01} imply that
\begin{align*}
\int_0^t G_\beta(t,s)s^{-\frac\gamma\alpha}\left(\int_0^s G_\beta(s,r)f(r) d  r\right) d  s&=\int_0^t G_\beta(t,s)s^{-\frac\gamma\alpha}\left(\int_0^s G_\beta(s,r)\textcolor{Dgreen}{r}^{-\frac\gamma\alpha}[\textcolor{Dgreen}{r}^{\frac\gamma\alpha}f(r)] d  r\right) d  s\\
&\textcolor{Dgreen}{\le\int_0^t\Big(\sum_{i=1}^4 (t-s)^{a_i}s^{b_i-\tfrac{\gamma}{\alpha}}\Big)\Big(\sum_{j=1}^4\int_0^s (s-r)^{a_j}r^{b_j-\tfrac{\gamma}{\alpha}} [r^{\tfrac{\gamma}{\alpha}}f(r)]\, d  r\Big)\, d  s}\\
&\textcolor{Dgreen}{\stackrel{\eqref{0215:00}}{\lesssim}2\sum_{i=1}^4\int_0^t(t-r)^{a_i}r^{b_i-\tfrac{\gamma}{\alpha}}[r^{\tfrac{\gamma}{\alpha}}f(r)]\, d  r}\\
&\lesssim \int_0^t G_\beta(t,r)r^{-\frac\gamma\alpha}[r^{\frac\gamma\alpha}f(r)] d  r=\int_0^t G_\beta(t,r)f(r) d  r.
\end{align*}
(iii) \textcolor{Dgreen}{Starting from \eqref{0214:03}, by  Lemma \ref{0214:lem00} (with the source term $g(t)+\int_0^t G_\beta(t,s)h(s) d  s$) and $(ii)$, }
we have
{\begin{align*}
f(t)&\lesssim g(t)+\int_0^t G_\beta(t,s)s^{-\frac{\gamma}{\alpha}}\left(g(s)+\int_0^s G_\beta(s,r)h(r) d  r\right) d  s\\
&\lesssim g(t)+\int_0^t G_\beta(t,s)\left(s^{-\frac{\gamma}{\alpha}}g(s)+h(s)\right) d  s,
\end{align*}
\black where the last inequality is from \eqref{0214:00}.}
This completes the proof.

\end{proof}

\end{appendices}

\end{document}

\end{document}

 Similarly, the following results hold.
\bl\label{RJP00}
For any $\bbp\in[1,\infty]^2$ and $\beta>0$, there is a constant $C>0$ such that for all $t\ge0$, $j\ge-1$ and $N\in\mN$
\begin{align}\label{08:20}
\|\cR_j^a\phi_N\|_{\bbp}\le CN^{{\color{MColor}\modulateorder}((2+{\color{MColor}\rate})d-\bba\cdot \frac{d}{\bbp})}\hbar^{\beta}.
\end{align}
where $\hbar:=2^{-j}N^{{\color{MColor}\modulateorder}}$.

In particular, for any $\beta\ge0$
\begin{align*}
\|\phi_N\|_{\bB^\beta_{\bbp;\bba}}\le C N^{{\color{MColor}\modulateorder}(\beta+(2+\alpha)d-\bba\cdot \frac{d}{\bbp})}.
\end{align*}
 
\el

$$
\nabla_v(\transOp_t\phi_N)=\transOp_t(\nabla_v\phi_N-t\nabla_x\phi_N),
$$
and
$$
\nabla_x(\transOp_t\phi_N)=\transOp_t\nabla_x\phi_N,\ \ \p_t\transOp_t\phi_N=-v\cdot\transOp_t(\nabla_x\phi_N)
$$
and
$$
\Delta_v(\transOp_t\phi_N)=\transOp_t(\Delta_v\phi_N-2t(\nabla_v\cdot\nabla_x)\phi_N+t^2\Delta_x\phi_N),
$$
Moreover,

\br
In case 2, in order to obtain the propagation of chaos, we need 
\begin{align*}
{\color{MColor}\modulateorder}(\beta_0+4d-\bba\cdot\frac{d}{\bbp_0\vee2}),{\color{MColor}\modulateorder}(-\beta_b+4d-\bba\cdot\frac{d}{\bbp_b'\vee2})<\frac12
\end{align*}
and
\begin{align}\label{08:22}
-\beta_b+\bba\cdot\frac{d}{\bbp_b}<1+\theta.
\end{align}
We note that \eqref{08:22} allow us to consider some supercriticle cases., i.e.
\begin{align*}
-\beta_b+\bba\cdot\frac{d}{\bbp_b}>1.
\end{align*}

\er